\documentclass[11pt,letterpaper]{article}

\usepackage{amsmath, amssymb, amsthm, bm, bbm}
\usepackage[margin=1.1in]{geometry}
\usepackage{authblk}

\usepackage[utf8]{inputenc} 
\usepackage{hyperref}       
\usepackage{url}            
\usepackage{booktabs}       
\usepackage{microtype}      
\usepackage{xcolor}    
\usepackage{hyperref}
\usepackage{enumerate}
\hypersetup{colorlinks=true, linkcolor=blue, citecolor=blue, filecolor=blue, urlcolor=blue}
\usepackage[square]{natbib}
\usepackage{graphicx}
\usepackage{wrapfig}
\usepackage{multirow}
\usepackage{algorithm}
\usepackage{algorithmic}
\usepackage{xargs}
\usepackage[most]{tcolorbox}
\usepackage{subcaption}
\usepackage{enumitem}

\usepackage{amsmath,amsfonts,bm}

\def\1{\bm{1}}

\def\vg{{\bm{g}}}

\def\vv{{\bm{v}}}

\def\mA{{\bm{A}}}
\def\mB{{\bm{B}}}
\def\mC{{\bm{C}}}
\def\mD{{\bm{D}}}

\def\mF{{\bm{F}}}
\def\mG{{\bm{G}}}

\def\mI{{\bm{I}}}

\def\mM{{\bm{M}}}

\def\mO{{\bm{O}}}
\def\mP{{\bm{P}}}
\def\mQ{{\bm{Q}}}

\def\mS{{\bm{S}}}

\def\mU{{\bm{U}}}
\def\mV{{\bm{V}}}
\def\mW{{\bm{W}}}
\def\mX{{\bm{X}}}
\def\mY{{\bm{Y}}}
\def\mZ{{\bm{Z}}}

\DeclareMathAlphabet{\mathsfit}{\encodingdefault}{\sfdefault}{m}{sl}
\SetMathAlphabet{\mathsfit}{bold}{\encodingdefault}{\sfdefault}{bx}{n}

\def\gC{{\mathcal{C}}}

\def\gN{{\mathcal{N}}}
\def\gO{{\mathcal{O}}}
\def\gP{{\mathcal{P}}}

\def\gT{{\mathcal{T}}}

\newcommand{\R}{\mathbb{R}}

\usepackage{amsmath}
\usepackage{amssymb}
\usepackage{mathtools}
\usepackage{amsthm}
\usepackage{thmtools,thm-restate}
\usepackage{nicefrac}
\usepackage{multirow}
\usepackage{tikz-cd} 
\usepackage{subcaption}
\usepackage{multicol,caption}

\usepackage{wrapfig}
\usepackage{algorithm}

\usepackage[flushleft]{threeparttable} 
\usepackage{colortbl}
\definecolor{bgcolor}{rgb}{0.8,1,1}
\definecolor{bgcolor2}{rgb}{0.8,1,0.8}
\definecolor{niceblue}{rgb}{0.0,0.19,0.56}

\usepackage{hyperref}
\hypersetup{colorlinks,linkcolor={blue},citecolor={blue},urlcolor={blue}}

\usepackage{pifont}
\definecolor{PineGreen}{RGB}{0,110,51}
\definecolor{BrickRed}{RGB}{143,20,2}
\newcommand{\cmark}{{\color{PineGreen}\ding{51}}}%
\newcommand{\xmark}{{\color{BrickRed}\ding{55}}}%

\usepackage{thmtools}

\newcommand{\Exp}[1]{\mathbb{E} \left[ #1\right]}

\definecolor{shadecolor}{gray}{0.9}
\declaretheoremstyle[
headfont=\normalfont\bfseries,
notefont=\mdseries, notebraces={(}{)},
bodyfont=\normalfont,
postheadspace=0.5em,
spaceabove=1pt,
mdframed={
  skipabove=8pt,
  skipbelow=8pt,
  hidealllines=true,
  backgroundcolor={shadecolor},
  innerleftmargin=4pt,
  innerrightmargin=4pt,
  innertopmargin=10pt,    
  innerbottommargin=6pt
}
]{shaded}

\declaretheorem[style=shaded,within=section]{definition}
\declaretheorem[style=shaded,sibling=definition]{theorem}

\declaretheorem[style=shaded,sibling=definition]{assumption}

\declaretheorem[style=shaded,sibling=definition]{lemma}

\usepackage{xspace}
\usepackage{xspace}

\newcommand{\algname}[1]{{\sf  #1}\xspace}

\newcommand{\la}{\left\langle}
\newcommand{\ra}{\right\rangle}

\newcommand{\eqdef}{:=}

\newcommand{\red}{\color{red}}

\usepackage{hyperref}
\hypersetup{hypertexnames=false,colorlinks,linkcolor={blue},citecolor={blue},urlcolor={blue}}

\usepackage{natbib}
\usepackage{sidecap}
\renewcommand{\bibname}{References}
\renewcommand{\bibsection}{\subsubsection*{\bibname}}

\usepackage{makecell}
\usepackage[utf8]{inputenc} 
\usepackage[T1]{fontenc}    
\usepackage{hyperref}       
\usepackage{url}            
\usepackage{booktabs}       
\usepackage{amsfonts}       
\usepackage{nicefrac}       
\usepackage{microtype}      
\usepackage{xcolor}         
\usepackage{algorithm}
\usepackage{algorithmic}
\usepackage{xargs}
\usepackage{enumitem}
\usepackage[colorinlistoftodos,prependcaption,textsize=tiny]{todonotes}
\newcommandx{\unsure}[2][1=]{\todo[linecolor=red,backgroundcolor=red!25,bordercolor=red,#1]{#2}}
\newcommandx{\change}[2][1=]{\todo[linecolor=blue,backgroundcolor=blue!25,bordercolor=blue,#1]{#2}}
\newcommandx{\info}[2][1=]{\todo[linecolor=OliveGreen,backgroundcolor=OliveGreen!25,bordercolor=OliveGreen,#1]{#2}}
\newcommandx{\improvement}[2][1=]{\todo[linecolor=Plum,backgroundcolor=Plum!25,bordercolor=Plum,#1]{#2}}

\allowdisplaybreaks[1]
\mathtoolsset{showonlyrefs=true}

\usepackage{mdframed}
\usepackage{etoolbox}
\newmdenv[
skipabove=0pt,
skipbelow=0pt,
hidealllines=true,
backgroundcolor=shadecolor,
innerleftmargin=2pt,
innerrightmargin=2pt,
innertopmargin=3.5pt,
innerbottommargin=3.5pt
]{shadedbox}

\usepackage[most]{tcolorbox}

\definecolor{phasepink}{RGB}{252,236,240}
\definecolor{phasegreen}{RGB}{235,246,236}

\newtcolorbox{phasebox}[2]{%
  enhanced,
  frame hidden,
  boxrule=0pt,
  sharp corners,
  colback=#1,
  left=1mm,
  right=1mm,
  top=0.7mm,
  bottom=0.7mm,
  boxsep=0pt,
  width=\dimexpr\linewidth-2.2cm\relax,
  before skip=2pt,
  after skip=2pt,
  overlay={
    \node[
      anchor=west,
      font=\footnotesize\bfseries,
      align=center,
      text width=1.8cm
    ] at ([xshift=6pt]frame.east) {#2};
  }
}

\usepackage{tabularx}
\usepackage{threeparttable}
\usepackage[table]{xcolor}
\usepackage{pifont}

\newcolumntype{L}[1]{>{\raggedright\arraybackslash}p{#1}}
\newcolumntype{C}[1]{>{\centering\arraybackslash}p{#1}}
\newcolumntype{Y}{>{\centering\arraybackslash}X}

\usepackage[nameinlink,noabbrev]{cleveref}
\crefname{assumption}{assumption}{assumptions}

\title{\bf Muon with Nesterov Momentum: Heavy-Tailed Noise and (Randomized) Inexact Polar Decomposition}

\newcommand{\affmark}[1]{\textsuperscript{#1}}

\author{\bf
Sayantan Choudhury\affmark{1} \quad
Xiaoran Cheng\affmark{2} \quad
Martin Takáč\affmark{1} \\
\bf Sen Na\affmark{3} \quad
Mladen Kolar\affmark{1,4}
}

\date{\affmark{1} MBZUAI \quad
\affmark{2} Penn State University \quad
\affmark{3} Georgia Institute of Technology \\
\affmark{4} University of Southern California}

\usepackage[nameinlink,noabbrev]{cleveref}

\begin{document}
\maketitle

\begin{abstract}
    Most first-order optimizers treat matrix-valued parameters as vectors, ignoring the intrinsic geometry of hidden-layer weights in neural networks. \algname{Muon} addresses this mismatch by updating along the polar factor of a momentum matrix, but its theoretical understanding has lagged behind practice. In particular, practical implementations incorporate Nesterov momentum, compute the polar factor only approximately, and operate with stochastic gradients that may be heavy-tailed. We close this gap by developing a convergence theory for \algname{Muon} with Nesterov momentum and inexact polar decomposition in non-convex matrix optimization under heavy-tailed noise. Our analysis builds on a unified framework for inexact polar decomposition that captures practical iterative approximations such as Newton-Schulz and quantifies how their errors propagate through the optimization dynamics. Under uniform bounds on polar errors, we establish an optimal iteration and sample complexity of $\gO (\varepsilon^{\nicefrac{-(3\alpha-2)}{(\alpha-1)}})$ for finding an $\varepsilon$-stationary point, where $\alpha\in(1,2]$~denotes the heavy-tail index. For the inexact-polar setting with $\sigma_1=0$, we also provide guarantees that do not require prior knowledge of $\alpha$. We analyze a randomized low-rank polar decomposition that is substantially more efficient than full-space methods while remaining compatible with our theory. Numerical experiments further demonstrate the effectiveness of the proposed inexact and randomized~variants.
\end{abstract}

\tableofcontents

\newpage
\section{Introduction}

Most widely used first-order optimizers \citep{kingma2014adam, loshchilov2017decoupled} still treat matrix-valued parameters as vectors, even though the weights of hidden linear layers are inherently structured matrices \citep{sato2025convergence}. Recent norm-based perspectives suggest that this geometric choice is consequential, and that different tensor structures may benefit from different update norms \citep{bernstein2024old}.
\algname{Muon}, introduced by \cite{jordan2024muon}, is a compelling realization of this idea: it replaces the raw momentum matrix with its orthogonalized polar factor before applying the update. For solving the optimization problem $\min_{\mX \in \R^{m \times n}} f(\mX)$, the update rule of \algname{Muon} with Nesterov momentum is given by: 
\begin{eqnarray}\label{eq:muon_inexact}
\mG_k & = &  \frac{1}{B}\sum_{i = 1}^B \mG_k^i, \\ 
\mC_k & = & \beta \mC_{k - 1} + \mG_k, \\
\mM_k & = & \beta \mC_k + \mG_k, \\
\mX_{k+1} & = &  \mX_k - \eta \gT(\mM_k),    
\end{eqnarray}
where $\beta \in (0,1)$ and $\eta > 0$ are the momentum parameter and step size, respectively, $B$ is the batch size, $\mG^i_k$ is an estimate of the gradient $\nabla f(\mX_k)$, and $\gT(\mM_k)$ is an approximation to the polar factor of $\mM_k$. When computed exactly, $\gT(\mM_k)=\mU_k\mV_k^\top$ for the singular value decomposition $\mM_k=\mU_k\bm{\Sigma}_k\mV_k^\top$, so \algname{Muon} replaces the raw momentum direction with its orthogonalized counterpart.     

This orthogonalization step is central to the method, but it also constitutes the primary computational bottleneck. For large models, computing the exact polar factor of $\mM_k$ is often prohibitively expensive. Consequently, practical implementations rely on a small number of iterations of fast matrix polynomial methods, such as Newton-Schulz \citep{higham1986computing,jordan2024muon} and PolarExpress \citep{amsel2025polar}, to obtain $\gT(\mM_k)$, an efficient approximation of the polar factor. This discrepancy creates a noticeable gap between theory and practice. Existing analyses of \algname{Muon} typically focus on either exact polar decomposition or Polyak-momentum variants \citep{shen2025convergence, chang2025convergence, sato2025convergence, shulgin2025beyond}. In contrast, the default PyTorch implementation employs Nesterov momentum together with an inexact polar decomposition \citep{pytorch_muon_doc, pytorch2019}. This practical variant is already being deployed at scale: \citet{liu2025muonscalablellmtraining} trained Moonlight, 3B and 16B parameter language models, using \algname{Muon} and reported roughly $2\times$ computational efficiency gain over \algname{AdamW}. More recently, \algname{Muon} with Nesterov momentum and inexact polar decomposition has been used to train one of the largest open-source language models, DeepSeek-V4 \citep{deepseekai2026deepseekv4}.

A second mismatch between existing theory and modern practice lies in the assumptions on the stochastic gradients $\mG_k^i$. Most optimization analyses assume bounded variance \citep{lan2020first}, while empirical evidence suggests that stochastic gradients in modern deep networks can exhibit heavy-tailed noise \citep{simsekli2019tail, battash2024revisiting}. Consequently, a realistic theory of \algname{Muon} should simultaneously account for three ingredients: Nesterov momentum, inexact polar decomposition, and heavy-tailed stochastic gradients.

Our goal in this paper is to provide precisely such a theory for non-convex matrix optimization.~We develop a unified framework for analyzing inexact polar decomposition (Assumption \ref{assume:general_inexact}), and establish convergence guarantees (Theorem \ref{theorem:inexact}) for \algname{Muon} with Nesterov momentum under heavy-tailed noise (Assumption \ref{assume:general_bounded_moment}). In Table \ref{tab:comparison}, we provide a detailed comparison of our analysis with existing prior works on \algname{Muon}.

\begin{table*}[ht]
\centering
\setlength{\tabcolsep}{3pt}
\renewcommand{\arraystretch}{1.4}
\caption{\small 
Comparisons of the theory of \algname{Muon} across prior works. In the fourth column, $\sigma_1\geq 0$ and $\alpha\in(1,2]$ denote the heavy tail parameters in Assumption \ref{assume:general_bounded_moment}. The classical bounded-variance regime corresponds to $\sigma_1=0$ and $\alpha=2$. The last column measures $\min_{0 \leq k \leq K} \Exp{\|\nabla f(\mX_k)\|_F}$. Our rate uses the uniform polar bounds of Theorem~\ref{theorem:inexact}, fixed batch size, and uniformly bounded initial error.}
\label{tab:comparison}

\begin{tabular}{lcccc}
\toprule
\makecell{\textbf{Paper}} & \makecell{\textbf{Nesterov} \\ \textbf{Momentum}} & \makecell{\textbf{Inexact Polar} \\ \textbf{Decomposition}} & \makecell{\textbf{Heavy Tailed} \\ \textbf{Noise?}} & \makecell{\textbf{Convergence} \\ \textbf{Rate}}\\
\midrule
\citet{riabinin2025gluon}  & \xmark & \xmark & $\sigma_1 = 0, \alpha = 2$ & $\gO \left( \frac{1}{K^{\nicefrac{1}{4}}}\right)$ \\
\citet{kovalev2025understanding}  & \xmark & \xmark & $\sigma_1 = 0, \alpha = 2$ & $\gO \left( \frac{1}{K^{\nicefrac{1}{4}}}\right)$ \\
\citet{li2025note}           & \xmark & \xmark & $\sigma_1 = 0, \alpha = 2$ & $\gO \left( \frac{1}{K^{\nicefrac{1}{4}}}\right)$ \\
\citet{shen2025convergence}  & \xmark & \xmark & $\sigma_1 = 0, \alpha = 2$ & $\gO \left( \frac{1}{K^{\nicefrac{1}{4}}}\right)$ \\
\citet{sato2025convergence}  & \cmark & \xmark & $\sigma_1 = 0, \alpha = 2$ & --- \\
\citet{chang2025convergence} & \cmark & \xmark & $\sigma_1 = 0, \alpha = 2$ & $\gO \left( \frac{\log K}{K^{\nicefrac{1}{4}}}\right)$ \\
\citet{shulgin2025beyond}    & \xmark & \cmark & $\sigma_1 = 0, \alpha = 2$ & $\gO \left( \frac{1}{K^{\nicefrac{1}{4}}}\right)$ \\
\citet{kim2026convergence}    & \xmark & \cmark & $\sigma_1 = 0, \alpha = 2$ & $\gO \left( \frac{1}{K^{\nicefrac{1}{4}}}\right)$ \\
\citet{he2025low} & \xmark & \cmark & $\sigma_1 = 0, \alpha \in (1, 2]$ & $\gO \left( \frac{\log K}{K^{\nicefrac{(\alpha - 1)}{(3 \alpha - 2)}}} \right)$ \\
\citet{yu2026sign} & \cmark & \xmark & $\sigma_1 \geq 0, \alpha \in (1, 2]$ & $\gO \left( \frac{1}{K^{\nicefrac{(\alpha - 1)}{(3 \alpha - 2)}}} \right)$ \\
\cellcolor{phasegreen}\textbf{Ours} (Theorem \ref{theorem:inexact})& \cellcolor{phasegreen} \cmark & \cellcolor{phasegreen} \cmark & \cellcolor{phasegreen} $\sigma_1 \geq 0, \alpha \in (1, 2]$ & \cellcolor{phasegreen} $\gO \left( \frac{1}{K^{\nicefrac{(\alpha - 1)}{(3 \alpha - 2)}}} \right)$\\
\bottomrule	
\end{tabular}
\end{table*}            
   
Finally, we analyze a \textbf{randomized} low-rank polar decomposition procedure (Algorithm \ref{alg:rand_ns_update}) that is significantly more efficient than full-space polar decomposition methods while remaining fully compatible with our theoretical framework. This algorithm exploits the widely observed phenomenon that the parameters of modern neural networks are often approximately low rank \citep{hao2024flora, hu2022lora, he2025low, zhao2024galore}.~A~related low-rank polar decomposition method was recently proposed \cite{he2025low}; however, their analysis relies on a stopping criterion involving the computationally expensive quantity $\|\mM_k - \gT(\mM_k)\|_*$.\footnote{$\|\cdot\|_*$ denotes the nuclear norm of a matrix.} In contrast, our analysis avoids this requirement altogether, thereby removing a key computational bottleneck and making the resulting method considerably more practical for large-scale settings. \\

\paragraph{Main Contributions.} We summarize the main contributions of the paper below. \vspace{-0.2cm}
\begin{itemize}
\item \textit{Heavy-tailed noise analysis for non-convex matrix optimization.} 
We provide a convergence analysis under relaxed assumptions on the stochastic gradients (Assumption~\ref{assume:general_bounded_moment}), allowing for heavy-tailed noise beyond standard bounded-variance settings. This enables us to capture the behavior of \algname{Muon} in modern deep learning applications. To our knowledge, this is the first analysis that simultaneously treats the
specific Nesterov update in \eqref{eq:muon_inexact}, generalized heavy-tailed
noise with $\sigma_1>0$, and an inexact polar oracle.

\item \textit{Optimal convergence guarantee.} 
Under uniform polar-error bounds and uniformly bounded initial error, we establish optimal iteration and sample complexity for \algname{Muon} with inexact polar decomposition and fixed batch size. Specifically, Theorem \ref{theorem:inexact} implies an $\varepsilon$-stationarity complexity of $\gO ( \varepsilon^{\nicefrac{-(3\alpha - 2)}{(\alpha - 1)}} )$. To our knowledge, this is the first work that achieves optimal sample complexity for \algname{Muon} with Nesterov momentum, even under bounded-variance assumption.

\item \textit{Unified inexact polar decomposition framework.} 
We develop a unified framework for inexact polar decomposition via Assumption~\ref{assume:general_inexact}, under which the Newton-Schulz iteration arises as a special case (see Proposition \ref{prop:ns}). This framework enables us to precisely quantify how approximation errors propagate through the optimization dynamics (see Theorem \ref{theorem:inexact}).

\item \textit{Adaptivity to unknown tail index $\alpha\in(1,2]$.} 
For the inexact-polar setting with $\sigma_1 = 0$ in Assumption \ref{assume:general_bounded_moment}, we establish convergence in Corollary \ref{corollary:inexact_without_alpha} without requiring prior knowledge of the tail parameter $\alpha$, yielding an adaptive method that is practically applicable~when the distributional properties of stochastic gradients are unknown.

\item \textit{Randomized polar decomposition.} 
Inspired by \cite{halko2011finding}, we analyze a randomized variant of polar decomposition (Algorithm \ref{alg:rand_ns_update}) that enables efficient low-rank implementations. With Frobenius normalization and either theoretical polynomial, Lemma~\ref{lem:uniform_polar_bounds} verifies the uniform polar bounds required by Theorem~\ref{theorem:inexact}.

\item \textit{Numerical experiments.} 
We complement our theory on two standard benchmarks in Section~\ref{sec:experiment}: nanoGPT pretraining \cite{Penedo2024Fineweb, Jordan2024Modded} and CNN training on CIFAR-10 \cite{Krizhevsky2009Learning}. Randomized \algname{Muon} performs comparably to the full-space \algname{Muon} while reducing per-step optimizer FLOPs by roughly $8.5\times$ on nanoGPT ($2136\to 251$~GFLOPs) and $1.8\times$ on CIFAR-10 ($14.80\to 8.08$~GFLOPs). Additionally, replacing the dense Gaussian sketch with a coordinate (Kaczmarz) sketch yields a further reduction in both settings. We also measure a maximum realized alignment loss of $0.59$ on the finite nanoGPT runs (Appendix~\ref{app:gamma}).
\end{itemize}        

\paragraph{Related Works.} 
Recent work has increasingly focused on geometry-aware and structure-exploiting optimization methods. For instance, \algname{Muon} \citep{jordan2024muon} orthogonalizes matrix-valued momentum using an approximate polar factor for hidden-layer updates. This perspective is naturally extended by the norm-constrained LMO framework underlying \algname{Scion} \citep{pethick2025training}. Building on these insights, \algname{Gluon} \citep{riabinin2025gluon} unifies \algname{Muon} and \algname{Scion}, effectively bridging the gap between theoretical guarantees and layer-wise practical implementation. 
Alongside these methods, second-order approaches such as \algname{Shampoo}~\citep{gupta2018shampoo} exploit tensor-structured preconditioning, which \algname{SOAP} \citep{vyas2024soap} further refines by incorporating Adam-style updates within \algname{Shampoo}'s eigenbasis to improve stability and efficiency. Taking a fundamentally different approach, \algname{Lion} \citep{chen2023symbolic} employs a sign-based momentum discovered via symbolic search, deliberately prioritizing memory efficiency over complex matrix geometry.

\section{(Randomized) Inexact Polar Decomposition}\label{section:motivation}
Given the SVD of the momentum matrix $\mM_k = \mU_k \bm{\Sigma}_k \mV_k^\top$, the exact polar factor is the solution to $\max_{\|\mO \|_{\mathrm{op}} \le 1} \langle \mM_k, \mO \rangle$, and the maximum is attained at $\mO = \mU_k \mV_k^\top$. However, computing this factor exactly is often prohibitively expensive in high-dimensional settings. Practical implementations, therefore, rely on iterative approximation schemes such as Newton-Schulz \citep{higham1986computing} and PolarExpress \citep{amsel2025polar}, which avoid an explicit SVD and instead repeatedly apply a polynomial map to a normalized matrix. Throughout this section, for an odd scalar polynomial $p$ and a rectangular matrix $\mZ=\mU\bm{\Sigma}\mV^\top$, the notation $p(\mZ)$ denotes the corresponding singular-value map $\mU p(\bm{\Sigma})\mV^\top$. Starting from $\mZ_0 = \mM_k / \delta$ with $\delta \geq \| \mM_k\|_{\rm op}$ (practitioners often use $\delta = \| \mM_k\|_F$), these methods generate iterates $\mZ_{t+1} = \varphi(\mZ_t), t = 0,1,\dots,q-1$, where $\varphi$ is chosen to drive the singular values toward $1$. Equivalently, if $p_0(x)=x$ and $p_{t+1}(x)=\varphi(p_t(x))$,~then~$\mZ_q = p_q(\mM_k/\delta)$ serves as a fast, hardware-friendly approximation of $\mU_k \mV_k^\top$. Two standard choices~of~polynomials~are: \vspace{-0.2cm}

\begin{itemize}
\item Newton-Schulz with a degree-$3$ polynomial: $\varphi(x) = \frac{1}{2}(3x - x^3)$,
\item Newton-Schulz with a degree-$5$ polynomial: $\varphi(x) = \frac{1}{8}(15x - 10x^3 + 3x^5)$.
\end{itemize}

For a fixed polynomial $\varphi$, \algname{Muon} implements $\gT(\mM_k) = p_q \left(\nicefrac{\mM_k}{\delta}\right)$ with a small number of iterations, typically $q$ in $5-10$.

An inexact polar map should remain nearly feasible in operator norm and retain most of the optimal alignment with $\mM_k$. Let $\mathcal F_k$ denote the history before drawing the current mini-batch.
Here $\mathcal F_k$ contains the initialization and all earlier mini-batches and polar draws; in particular, $\mX_k$ and $\mC_{k-1}$ are $\mathcal F_k$-measurable.
Define $\mathcal G_k\eqdef\sigma(\mathcal F_k,\mG_k^1,\ldots,\mG_k^B)$, the history after forming the current gradients and momentum but before the polar draw. Thus $\mathcal F_k\subseteq\mathcal G_k\subseteq\mathcal F_{k+1}$; Appendix~\ref{app:preparation} details this sampling order. We model the polar oracle as $\gT(\mM_k)=\gT_k(\mM_k;\zeta_k)$, where $\gT_k$ is a fixed measurable map and the fresh seed $\zeta_k$ is independent of $\mathcal G_k$.

\begin{assumption}\label{assume:general_inexact}
We assume there exist deterministic bounds $\gamma_k \in [0, 1)$ and $\nu_k \ge 0$ such that
\begin{equation*}
\textstyle
\Exp{ \langle \mM_k, \gT(\mM_k) \rangle \mid \mM_k} \geq (1 - \gamma_k) \| \mM_k\|_*\quad\quad \text{and}\quad\quad \Exp{\| \gT(\mM_k)\|_{\text{op}}^2 \mid \mM_k} \leq (1 + \nu_k)^2.	
\end{equation*}
\end{assumption}

Under this oracle model, the conditional law of $\gT(\mM_k)$ given $\mathcal G_k$ depends only on $\mM_k$, so both bounds also hold when conditioning on $\mathcal G_k$.
Assumption~\ref{assume:general_inexact} recovers the exact polar decomposition $\gT(\mM_k)=\mU_k\mV_k^\top$ with $\gamma_k=\nu_k=0$.~Indeed, $\|\gT(\mM_k)\|_{\mathrm{op}} = \|\mU_k\mV_k^\top\|_{\mathrm{op}} = 1$, and
\begin{equation*}
\langle \mM_k, \gT(\mM_k)\rangle
= \langle \mM_k, \mU_k\mV_k^\top\rangle
= \mathrm{tr}\!\left(\mV_k \bm{\Sigma}_k \mU_k^\top \mU_k \mV_k^\top\right)
= \mathrm{tr}(\bm{\Sigma}_k)
= \|\mM_k\|_*.	
\end{equation*}
Moreover, the approximation condition $\|\gP(\mM_k)-\gT(\mM_k)\|_{\mathrm{op}} \le \varepsilon_k$ with $\varepsilon_k<1$ (where $\gP(\mM_k) \eqdef \mU_k \mV_k^\top$) from \cite{shulgin2025beyond} is a special case of Assumption~\ref{assume:general_inexact}. Indeed, we note that
\begin{align*}
\langle \mM_k, \gT(\mM_k)\rangle
&= \langle \mM_k, \gP(\mM_k)\rangle - \langle \mM_k, \gP(\mM_k)-\gT(\mM_k)\rangle \\
&\ge \|\mM_k\|_* - \|\mM_k\|_* \|\gP(\mM_k)-\gT(\mM_k)\|_{\mathrm{op}} \ge (1-\varepsilon_k)\|\mM_k\|_*,	
\end{align*}
while the triangle inequality gives $\|\gT(\mM_k)\|_{\mathrm{op}} \le \|\gP(\mM_k)\|_{\mathrm{op}} + \|\gT(\mM_k)-\gP(\mM_k)\|_{\mathrm{op}} \le 1+\varepsilon_k$. Beyond these assumptions, we show that the few-step Newton-Schulz updates~\citep{higham1986computing} discussed above also fall within this framework.    

\begin{restatable}{proposition}{propns}
\label{prop:ns}
Suppose $\varphi(x) = \frac{1}{2}(3x-x^3)$ or $\varphi(x) = \frac{1}{8}(15x-10x^3+3x^5)$. Let $p_0(x)=x$ and $p_{t+1}(x)=\varphi(p_t(x))$ for $t=0,1,\dots,q-1$. Let $\mM=\mU\bm{\Sigma}\mV^\top$ be a compact SVD of a rank-$r$ matrix with singular values $\sigma_1,\dots,\sigma_r>0$, and let $\delta\geq \|\mM\|_{\mathrm{op}}$. Define
\begin{equation*}
\gT \left( \mM \right) \eqdef p_q \left(\nicefrac{\mM}{\delta}\right) = \mU \operatorname{diag} \left( p_q \left( \nicefrac{\sigma_1}{\delta} \right), \cdots, p_q \left( \nicefrac{\sigma_r}{\delta}\right) \right) \mV^\top.	
\end{equation*}
Then, $\gT(\mM)$ satisfies Assumption~\ref{assume:general_inexact} with
$\nu = 0$ and 
$\gamma = 1 - \frac{\sum_{i=1}^r \sigma_i p_q \left(\nicefrac{\sigma_i}{\delta}\right)}{\sum_{i=1}^r \sigma_i}$.
\end{restatable}

The proposition above shows that Assumption~\ref{assume:general_inexact} captures the deterministic inexact polar decompositions used by \algname{Muon}. A particularly relevant work is \cite{kim2026convergence}, which finds error bounds on finite-step Newton-Schulz, similar to Proposition \ref{prop:ns}, and provides an analysis of \algname{Muon} with Polyak momentum under bounded variance noise. However, when the matrices are large, even a small number of Newton-Schulz steps requires several dense matrix multiplications in the full ambient space. To reduce this cost, we propose a randomized lifted variant that first compresses $\mM$ into a low-dimensional subspace and then performs the inexact polar computation in that reduced space.
\begin{algorithm}[ht]
\caption{Lifted Randomized Polar Decomposition}\label{alg:rand_ns_update}
\begin{algorithmic}[1]
\REQUIRE Matrix $\mM \in \mathbb{R}^{m \times n}$, sketch size $1\le\ell\le\min(m,n)$, integers $h,q\ge0$, and scaling $\delta\ge\|\mM\|_{\mathrm{op}}$ for $\mM\ne\bm0$.
\STATE If $\mM=\bm0$, return $\bm0$; for the uniform guarantee, set $\delta=\|\mM\|_F$.
\begin{phasebox}{phasepink}{Project Down}	
\STATE Draw a fresh Gaussian sketch $\bm\Omega \sim \mathcal{N}(0,1)^{n \times \ell}$, independently of the input and prior history;
\STATE Form $\mY := (\mM \mM^\top)^h \mM \bm\Omega$; let $\mQ := \operatorname{orth}(\mY)$ span its actual range;
\STATE Set $\mB := \mQ^\top \mM \in \mathbb{R}^{\operatorname{rank}(\mY)\times n}$;
\end{phasebox}
\begin{phasebox}{phasegreen}{Iterative \\ Approximation}
\STATE Set $\mZ_0 := \mB/\delta$;
\FOR{$t = 0, 1,\dots,q-1$}
\STATE $ \mZ_{t+1} := \varphi \left( \mZ_t \right)$;
\ENDFOR
\end{phasebox}
\begin{phasebox}{phasepink}{Project Up}
\STATE \textbf{Return} $\gT_{h,q}( \mM; \bm\Omega) := \mQ \mZ_q$.
\end{phasebox}
\end{algorithmic}			
\end{algorithm}

Algorithm \ref{alg:rand_ns_update} draws a Gaussian sketch $\bm\Omega \in \mathbb{R}^{n \times \ell}$ (typically $\ell = s+p$) and forms $\mY := (\mM\mM^\top)^h \mM \bm\Omega$ to capture the dominant column space. Then, an orthonormal basis $\mQ := \operatorname{orth}(\mY)$ is computed via a QR decomposition and used to form $\mB := \mQ^\top \mM$, with $\operatorname{rank}(\mY)\le\ell$ rows. We then compute an inexact polar decomposition of $\mB$, and lift the result back via multiplication with $\mQ$. The randomness of the method is induced solely by the Gaussian sketch $\bm\Omega$, which is drawn afresh independently of $\mathcal G_k$ at every optimizer iteration $k$. Its computational efficiency arises from performing the inexact polar decomposition in the reduced $\ell$-dimensional subspace. Our approach builds on the randomized sketching framework of \cite{halko2011finding}. As in this work, Algorithm~\ref{alg:rand_ns_update} first constructs a low-dimensional sketch of the range of the input matrix. The key difference is that we avoid computing an SVD of the compressed matrix; instead, we obtain an approximate polar factor through inexpensive iterative Newton-Schulz updates.

The compression induced by the sketching matrix yields a significant computational advantage. Compared with applying Newton-Schulz directly in the full space, which for $\mM \in \R^{m \times n}$ incurs a cost of $\gO \left(q(4d_1d_0^{2}+2d_0^{3})\right)$ with $d_0:=\min(m,n)$ and $d_1:=\max(m,n)$, our method first compresses $\mM\in\R^{m\times n}$ into $\mB \in \R^{\ell \times n}$, where $\ell = s+p \ll d_0$, and then performs the same iteration only in this $\ell$-dimensional subspace. The resulting complexity becomes $\gO \left((4h+6)mn\ell + q(4n\ell^{2}+2\ell^{3}) \right)$, so the cubic dependence on $d_0$ is replaced by a cubic dependence on the much smaller sketch size $\ell$. In the square case $m=n=d$, this reduces the scaling from $\gO(qd^{3})$ to $\gO \left((4h+6)d^{2}\ell + q(4d\ell^{2}+2\ell^{3})\right)$. For instance, with $(d,\ell,q,h)=(4096,256,5,1)$, the factorization cost is reduced by roughly $40\times$.

The next proposition shows that the randomized lifted scheme (Algorithm \ref{alg:rand_ns_update}) remains spectrally bounded and preserves a controlled amount of alignment with the momentum matrix $\mM$.

\begin{restatable}{proposition}{proprandomizedpolar}
\label{prop:randomized_polar}
Let $\mM \in \R^{m \times n}$ have rank $r$ with singular values
$\sigma_1 \ge \sigma_2 \ge \cdots \ge \sigma_r > 0$. Let $1 \leq s < r$,
let $p \ge 2$ and $\ell=s+p\le \min(m,n)$, and let $h,q$ be nonnegative
integers. Define $\varrho_s \eqdef \nicefrac{\sigma_{s + 1}}{\sigma_s} \in [0, 1]$.
For any $\delta\ge \|\mM\|_{\mathrm{op}}$, the output
$\gT_{h, q}(\mM; \bm\Omega)$ from Algorithm \ref{alg:rand_ns_update}, with
$\bm\Omega\sim\mathcal{N}(0,1)^{n\times \ell}$ and
$\varphi(x) = \frac{1}{2} (3x - x^3)$ or
$\varphi(x) = \frac{1}{8} (15x - 10x^3 + 3x^5)$, satisfies
\begin{equation}\label{eq:randomized_polar}
\mathbb{E}_{\bm\Omega} \left[\langle \mM, \gT_{h,q}( \mM; \bm\Omega )\rangle\right] \ge
\frac{1}{\delta}
\left[
\sum_{j=1}^s \sigma_j^2
-
\frac{s}{p-1} \varrho_s^{4h}\sum_{j > s}\sigma_j^2
\right]_+ . \tag{2}
\end{equation}
Moreover, almost surely,
$\left\| \gT_{h, q} \left( \mM; \bm\Omega \right) \right\|_{\mathrm{op}} \leq 1$.
\end{restatable}

Proposition~\ref{prop:randomized_polar} gives a spectrum-dependent alignment bound; Appendix~\ref{appendix:discuss_prop} discusses when its bracket is positive. For the convergence rates, we also need alignment bounded away from zero uniformly over iterations and horizons. With $\delta=\|\mM\|_F$ for $\mM\ne\bm0$ and $\gT(\bm0)=\bm0$, Lemma~\ref{lem:uniform_polar_bounds} proves that either theoretical polynomial admits uniform bounds $\Bar\nu=0$ and $\Bar\gamma=1-d_0^{-1/2}$ in full space, or $\Bar\gamma=1-d_0^{-3/2}$ for Algorithm~\ref{alg:rand_ns_update}, where $d_0=\min(m,n)$. These conservative bounds require no spectral gap and verify the uniform polar conditions in Theorem~\ref{theorem:inexact}. Other polynomial variants require separate verification. Our experiments also use a faster Kaczmarz-style sparse sketch, described in Appendix~\ref{app:kaczmarz}.

\section{Main Theoretical Results}\label{section:main_results}

In this section, we establish convergence guarantees for \algname{Muon} with Nesterov momentum under heavy-tailed stochastic gradient noise. Our analysis proceeds in two steps. We first study an idealized version of \algname{Muon} with exact polar decomposition, which isolates the effects of the momentum~dynamics and heavy-tailed noise. We then extend the argument to inexact polar decompositions, capturing practical implementations based on approximate polar iterations. Throughout, the initial iterate $\mX_0$ is deterministic; let $f_\star$ denote a finite lower bound on $f$, let $d_0 \eqdef \min\{m,n\}$ for $\mX\in\mathbb{R}^{m\times n}$. We use the histories $\mathcal F_k$ and $\mathcal G_k$ defined in Section~\ref{section:motivation}; in particular, $\mX_k$ is $\mathcal F_k$-measurable. Unless otherwise stated, $\gO(\cdot)$ hides constants depending only on $\alpha$ and $d_0$. We assume that $f$ is $L$-smooth as stated in the following assumption.

\begin{assumption}\label{assume:lipschitz}
We assume that $f$ has an $L$-Lipschitz gradient in the Frobenius norm, that is, $\|\nabla f(\mX)-\nabla f(\mY)\|_F\leq L\|\mX-\mY\|_F$ for all $\mX,\mY$. In particular, 
\begin{equation*}
\textstyle
  f(\mY) \leq f(\mX) + \la \nabla f(\mX), \mY - \mX \ra + \frac{L}{2} \| \mX - \mY\|_F^2.  
\end{equation*}
\end{assumption}

Assumption~\ref{assume:lipschitz} is standard in the analysis of first-order methods~\citep{loizou2021stochastic,ward2020adagrad,nesterov2004introductory} and is also commonly adopted in recent analyses of \algname{Muon}~\citep{sato2025convergence,chang2025convergence,shen2025convergence}.

Existing analyses of \algname{Muon} typically focus on a bounded-variance condition, $\mathbb{E}[\| \mG_k^i - \nabla f (\mX_k) \|_F^2 \mid \mathcal{F}_k] \leq \sigma_0^2$ (see \citep{sato2025convergence,chang2025convergence}). To better accommodate the training regimes encountered in modern deep learning, we instead adopt the following generalized moment assumption.

\begin{assumption}[\textit{Generalized Heavy-Tailed Noise}]\label{assume:general_bounded_moment}
Conditional on $\mathcal{F}_k$, we assume that $\mG_k^1,\ldots,\mG_k^B$ are independent, and satisfy $\Exp{\mG_k^i \mid \mathcal{F}_k} = \nabla f(\mX_k)$ and
\begin{equation}\label{eq:general_bounded_moment}
\textstyle
\Exp{ \| \mG_k^i - \nabla f (\mX_k)\|_F^\alpha \mid \mathcal{F}_k } \leq \sigma_0^\alpha + \sigma_1^\alpha \| \nabla f(\mX_k) \|_F^\alpha
\end{equation}
with $\sigma_0, \sigma_1 \geq 0$ and $\alpha \in (1, 2]$ for all $i \in [B]$.
\end{assumption}

Assumption \ref{assume:general_bounded_moment} covers several familiar regimes. When $\alpha=2$ and $\sigma_1=0$, it reduces to the standard bounded-variance condition. When $\alpha \in (1,2)$, it accommodates heavy-tailed stochastic gradients, while the term involving $\sigma_1$ further allows the noise to grow with the gradient magnitude.~Related assumptions have been studied for vector-valued problems in \cite{liu2024nonconvex}; the special case $\sigma_1=0$ recovers the $\alpha$-bounded moment assumption considered in prior works \citep{hubler2024gradient,zhang2020adaptivemethodsgoodattention,chezhegov2024clipping}. To the best of our~knowledge, \cite{he2025low} is the only prior work that analyzes \algname{Muon} with $\alpha \in (1, 2]$, but it focuses on Polyak~momentum and $\sigma_1 = 0$. Moreover, \cite{liu2024nonconvex} provides a simple counterexample showing that Assumption \ref{assume:general_bounded_moment} with $\sigma_1 = 0$ may fail to hold even for 1-dimensional regression problems, thereby highlighting the need for $\sigma_1 > 0$ in practical settings. The generalization to $\alpha \in (1, 2]$ is particularly relevant in light of empirical evidence for heavy-tailed gradient noise in both image classification and large language model training \citep{simsekli2019tail,battash2024revisiting,zhang2020adaptivemethodsgoodattention,ahn2023linear}. Finally, Assumption \ref{assume:general_bounded_moment} yields a bound on the noise level of $\mG_k$~(Lemma \ref{lemma:batch_variance}) that decreases monotonically with the batch size $B$. This shows that mini-batching reduces the effective noise level, and will be used in both the exact and inexact analyses below.

\subsection{Warm-up: Analysis for Exact Polar Decomposition}

We begin with the idealized setting in which the polar decomposition is computed exactly, namely $\gT (\mM_k) = \mU_k \mV_k^\top$ whenever $\mM_k = \mU_k \bm{\Sigma}_k \mV_k^\top$ is the SVD of $\mM_k$. This setting isolates the effect of Nesterov momentum without the additional complication of approximation error in the polar step.

It is convenient to rescale the momentum matrices by $(1 - \beta)$, i.e., $\widetilde{\mC}_k \eqdef (1-\beta) \mC_k$ and $\widetilde{\mM}_k \eqdef (1-\beta) \mM_k$. This rescaling does not change the polar factor, since multiplying a matrix by a positive scalar only rescales its singular values. Thus, $\gT(\widetilde{\mM}_k) = \gT(\mM_k)=\mU_k\mV_k^\top$. However, $\widetilde{\mM}_k$ satisfies a nice simple recursion.

\begin{restatable}{lemma}{lemmascaledrecurrence}
\label{lemma:scaled_recurrence}
For $\widetilde{\mC}_k \eqdef (1-\beta) \mC_k$ and $\widetilde{\mM}_k \eqdef (1-\beta) \mM_k$, we have, for all $k\ge1$,
\begin{equation*}
\widetilde{\mC}_k = \beta\widetilde{\mC}_{k-1} + (1-\beta)\mG_k \quad\quad\text{and}\quad\quad
\widetilde{\mM}_k  = \beta\widetilde{\mM}_{k-1} + (1-\beta)\big((1+\beta)\mG_k-\beta\mG_{k-1}\big).
\end{equation*}
\end{restatable}
 
Our analysis begins with a descent inequality for the exact-polar update (Lemma~\ref{lemma:exact_descent}). Following standard \algname{Muon} analyses \citep{shen2025convergence,chang2025convergence}, this bound depends on the momentum error term $\mS_k = \widetilde{\mM}_k - \nabla f(\mX_k)$. The main technical challenge is therefore to control $\mS_k$, which is not itself a martingale difference because of Nesterov momentum.
We address this by unrolling the recursion for $\mS_k$ into three components: the initial bias $\mS_0$, a drift term $\mD_k=\nabla f(\mX_{k-1})-\nabla f(\mX_k)$, and a geometrically weighted sum of stochastic gradient noises $\bm\xi_k=\mG_k-\nabla f(\mX_k)$ (Lemma~\ref{lemma:eps_unroll}). Smoothness bounds the drift, while Lemma~\ref{lemma:batch_variance}, the martingale inequality in Lemma~\ref{lemma:bound_martingle}, and the conditional weighted estimate in Lemma~\ref{lemma:weighted_conditional_martingale} control the noise terms. The latter estimate handles the gradient-dependent noise using only first moments of the gradient norm. Finally, for \algname{Muon} with Nesterov momentum and exact polar decomposition, we establish the following convergence guarantee under heavy-tailed noise.

\begin{restatable}{theorem}{corollaryexactwithalpha}
\label{corollary:exact_with_alpha}
Suppose Assumptions \ref{assume:lipschitz} and \ref{assume:general_bounded_moment} hold. For a horizon $K\ge2$, \algname{Muon} \eqref{eq:muon_inexact} with exact polar decompositions, step size $\eta = \nicefrac{1}{K^{\nicefrac{(2 \alpha - 1)}{(3 \alpha - 2)}}}$, momentum parameter $\beta = 1 - \nicefrac{1}{K^{\nicefrac{\alpha}{(3 \alpha - 2)}}}$, and batch size $B > \{ 2 \sqrt{\pi} ( 1 + \sqrt{d_0} ) \gC_\alpha^{\nicefrac{1}{\alpha}} \sigma_1 \}^{\nicefrac{\alpha}{\alpha - 1}}$ satisfies
\begin{multline}\label{eq:exact_with_alpha}
\min_{0 \leq k \leq K-1} \Exp{\|\nabla f(\mX_k)\|_F} \leq \gO \left(
\frac{f(\mX_0)-f_\star}{\rho_{\rm ex} {\red K^{\nicefrac{\alpha - 1}{3\alpha - 2}}} }
+ \frac{L}{\rho_{\rm ex} {\red K^{\nicefrac{\alpha - 1}{3 \alpha - 2}}} }
+ \frac{\sigma_0}{\rho_{\rm ex} B^{\nicefrac{\alpha - 1}{\alpha}} {\red K^{\nicefrac{\alpha - 1}{3 \alpha - 2}}}} \right.
\\
\left.+ \frac{L}{\rho_{\rm ex} K^{\nicefrac{2\alpha - 1}{3\alpha - 2}}} 
+ \frac{\Exp{\| \mS_0 \|_F}}{\rho_{\rm ex} K^{\nicefrac{2\alpha - 2}{3\alpha - 2}}} + \frac{\sigma_1 \| \nabla f(\mX_0)\|_F}{\rho_{\rm ex} B^{\nicefrac{\alpha - 1}{\alpha}} K}\right),
\end{multline}
where $\rho_{\rm ex}\hskip-1pt = \hskip-1pt 1 - (1\hskip-2pt+\hskip-2pt\sqrt{d_0})2\sqrt{\pi}\gC_\alpha^{\nicefrac{1}{\alpha}}\sigma_1/B^{\nicefrac{\alpha-1}{\alpha}}\hskip-1pt>\hskip-1pt0$ and $0<\gC_\alpha \leq 2$ is a scalar depending only~on~$\alpha$.
\end{restatable}

The choices of $\eta$ and $\beta$ do not require the smoothness constant $L$. For all asymptotic rates below, fix the problem parameters and $\mX_0$, take $B$ independent of $K$ and the target accuracy, and assume $\Exp{\|\mS_0\|_F}$ is bounded uniformly in $K$. With $B$ satisfying the batch size condition, the dominant term on the right side of \eqref{eq:exact_with_alpha} is $\gO(\nicefrac{1}{K^{\nicefrac{(\alpha - 1)}{(3 \alpha - 2)}}} )$. This aligns with the concurrent work of \cite{yu2026sign} that only studies exact polar decomposition. Moreover, to obtain an $\varepsilon$-accurate point, the algorithm requires $\gO ( \varepsilon^{\nicefrac{-(3\alpha - 2)}{(\alpha - 1)} })$ iterations, which match the optimal complexity for non-convex minimization under heavy-tailed noise~\citep{zhang2020adaptivemethodsgoodattention, yu2026sign}. In particular, when $\alpha = 2$, this bound recovers the $\gO ( \varepsilon^{-4} )$ complexity, which is known to be optimal under the standard bounded-variance setting~\citep{lan2012optimal, ghadimi2013stochastic}. Compared with $\gO \left( \log K / K^{1/4} \right)$ rate obtained for the Nesterov variant under bounded variance in \cite[Theorem~3.1(1)]{chang2025convergence}, our result covers the heavy-tailed noise regime ($\alpha \in (1, 2), \sigma_1 > 0$) and yields a faster asymptotic rate in the special case $\alpha=2, \sigma_1=0$.

The lower bound on $B$ is required only because we allow more general regime $\sigma_1>0$ in Assumption~\ref{assume:general_bounded_moment}. When $\sigma_1=0$, no such restriction is needed, and even $B=1$ is admissible. Therefore, a large batch size is not a limitation of our approach. This dependence of $B$ on $\sigma_1$ aligns with the findings of \cite{liu2024nonconvex} for the $\sigma_1 > 0$ regime. Notably, since the batch size threshold is independent of $K$, the sample complexity has the same order as the iteration complexity for fixed $B$, i.e., $\gO ( \varepsilon^{\nicefrac{-(3\alpha - 2)}{(\alpha - 1)}} )$.

\subsection{Analysis for Inexact Polar Decomposition}

We now turn to the practically relevant setting in which the polar decomposition is computed only approximately, captured by Assumption~\ref{assume:general_inexact}. The proof follows a similar overall structure to the exact case, but the descent argument must now account for the loss of alignment between the momentum matrix $\widetilde{\mM}_k$ and the approximate polar factor $\gT(\mM_k)$. The key observation is that Assumption~\ref{assume:general_inexact} implies
$\mathbb{E}[\langle \widetilde{\mM}_k, \gT(\mM_k)\rangle \mid \mathcal G_k]
\geq
(1-\gamma_k)\|\widetilde{\mM}_k\|_*$,
because $\widetilde{\mM}_k=(1-\beta)\mM_k$. Thus, after taking outer expectations, the inexact polar step still preserves a fraction of the alignment achieved in the exact case. Together with the control of the momentum drift as before, this yields the following result.

\begin{restatable}{theorem}{theoreminexact}
\label{theorem:inexact}
Suppose Assumptions \ref{assume:general_inexact}, \ref{assume:lipschitz}, and \ref{assume:general_bounded_moment} hold. Assume deterministic constants $\Bar\gamma\in[0,1)$ and $0\le\Bar\nu<\infty$, independent of the iteration $k$ and horizon $K$, satisfy $\gamma_k\le\Bar\gamma$ and $\nu_k\le\Bar\nu$ in every run considered. Then, for any $\eta>0$ and $\beta\in(0,1)$, \algname{Muon} \eqref{eq:muon_inexact} with inexact polar decomposition and fixed batch size $B > \{ \nicefrac{2 \sqrt{\pi} \left( 1 + \sqrt{d_0} (1 + \Bar{\nu}) \right) \gC_\alpha^{\nicefrac{1}{\alpha}} \sigma_1 }{(1 - \Bar{\gamma})} \}^{\nicefrac{\alpha}{\alpha - 1}}$ satisfies
\begin{multline*}
\min_{0 \leq k \leq K-1} \Exp{\|\nabla f(\mX_k)\|_F}
\leq
\gO \left(
\frac{f(\mX_0)-f_\star}{\rho{\red \eta K}}
+ \frac{\beta^2 L (1 + \Bar{\nu})^2{\red \eta} }{\rho{\red (1 - \beta)} }
+ \frac{(1 + \Bar{\nu}) \sigma_0{\red (1-\beta)^{\nicefrac{\alpha-1}{\alpha}}}}{\rho B^{\nicefrac{\alpha-1}{\alpha}}} \right.\\
\left. + \frac{L \eta \left( 1 + \Bar{\nu} \right)^2}{\rho}
+ \frac{(1 + \Bar{\nu}) \Exp{\| \mS_0 \|_F}}{\rho(1 - \beta) K }
+ \frac{(1 + \Bar{\nu}) \sigma_1 \|\nabla f(\mX_0)\|_F}{\rho B^{\nicefrac{\alpha-1}{\alpha}} K}\right),
\end{multline*}
where $\rho \eqdef 1 - \Bar{\gamma} - \left( 1+\sqrt{d_0} \left( 1 + \Bar{\nu} \right) \right) 2 \sqrt{\pi} \gC_\alpha^{\nicefrac{1}{\alpha}} \sigma_1/B^{\nicefrac{\alpha-1}{\alpha}}>0$ and $\gC_\alpha\in(0,2]$ depends only~on~$\alpha$.
\end{restatable}

The uniform constants $\Bar\gamma,\Bar\nu$ and fixed $B$ ensure that $\rho>0$ is independent of $K$. Larger polar-error bounds increase the required batch size and worsen the constants. Per-iteration bounds $\gamma_k<1$ alone do not ensure a rate if their maxima approach one as $K$ grows. Under these uniform bounds and the initial-error condition stated above, choosing $\eta=K^{-(2\alpha-1)/(3\alpha-2)}$ and $\beta=1-K^{-\alpha/(3\alpha-2)}$ gives $\gO(K^{-(\alpha-1)/(3\alpha-2)})$ convergence (dominant terms are highlighted in red). Thus the iteration and sample complexity is $\gO(\varepsilon^{-(3\alpha-2)/(\alpha-1)})$, since $B$ is fixed independently of accuracy. Lemma~\ref{lem:uniform_polar_bounds} supplies the required uniform polar bounds $\Bar\gamma, \Bar\nu$ for the full-space and Gaussian constructions with Frobenius normalization.

\paragraph{Convergence Guarantees without the Knowledge of $\alpha$.} 
The parameter choices in the previous theorems depend on the heavy-tail parameter $\alpha$ in Assumption \ref{assume:general_bounded_moment}, which may be difficult to estimate in practice. We therefore also provide a universal schedule that guarantees convergence without requiring knowledge of $\alpha$ in the inexact-polar setting with $\sigma_1=0$.

\begin{restatable}{corollary}{corollaryinexactwithoutalpha}
\label{corollary:inexact_without_alpha}
Suppose the assumptions and uniform polar bounds of Theorem~\ref{theorem:inexact} hold with $\sigma_1=0$. Set $\rho_0\eqdef1-\Bar\gamma>0$, independent of $K$.
For a horizon $K\ge2$, \algname{Muon} \eqref{eq:muon_inexact} with inexact polar decomposition, step size $\eta = \frac{1}{K^{3/4}}$, momentum parameter $\beta = 1 - \frac{1}{\sqrt{K}}$, and any batch size $B\ge1$ satisfies
\begin{multline*}
\min_{0 \leq k \leq K-1} \Exp{\|\nabla f(\mX_k)\|_F}
\leq \gO \left(
\frac{f(\mX_0)-f_\star}{\rho_0 K^{\frac{1}{4}}}
+ \frac{L(1+\Bar{\nu})^2}{\rho_0 K^{\frac{3}{4}}}
+ \frac{(1+\Bar{\nu})\Exp{\| \mS_0 \|_F}}{\rho_0\sqrt{K}} \right.\\
\left. + \frac{L(1+\Bar{\nu})^2}{\rho_0 K^{\frac{1}{4}}}
+ \frac{(1+\Bar{\nu})\sigma_0}
{\rho_0 B^{\frac{\alpha - 1}{\alpha}} {\red K^{\frac{\alpha - 1}{2 \alpha}}}}
\right).
\end{multline*}
\end{restatable}


Corollary~\ref{corollary:inexact_without_alpha} removes the need to know $\alpha$ when choosing the parameters $\eta$ and $\beta$. In the regime $\sigma_1=0$, the batch size can also be chosen arbitrarily, so no algorithmic parameter depends on $\alpha$. Under the uniform polar bounds and the initial-error condition stated above, the rate is $\gO (K^{-\nicefrac{(\alpha-1)}{2\alpha}})$, which leads to an iteration complexity of $\gO ( \varepsilon^{\nicefrac{- 2 \alpha}{(\alpha - 1)}} )$. However, when $\alpha=2$, this recovers the optimal $\gO ( \varepsilon^{-4})$ complexity.    

\section{Numerical Experiments}\label{sec:experiment}

In this section, we present numerical experiments evaluating randomized \algname{Muon} with Nesterov momentum against six baselines: \algname{AdamW}, \algname{SGD} with Nesterov momentum, \algname{Muon} with Polyak momentum, \algname{Muon} with Nesterov momentum, randomized \algname{Muon} with Polyak momentum, and randomized \algname{Muon} with Nesterov momentum, but Kaczmarz sketch.
We consider two standard benchmarks: pretraining a nanoGPT model on the FineWeb dataset \citep{Penedo2024Fineweb} and training a CNN on CIFAR-10 \citep{Krizhevsky2009Learning}. For each benchmark, we include additional ablation studies in Appendix \ref{app:nano_ablation} and \ref{app:cifar_ablation} to examine the sensitivity of randomized \algname{Muon} with Nesterov momentum. Appendix~\ref{app:gamma} reports a maximum realized alignment loss of $0.59$ on the finite nanoGPT runs; these measurements do not certify uniform bounds over all horizons. Full experimental setup details for both benchmarks are deferred to Appendix \ref{app:expdetails}.

\subsection{Pretraining of nanoGPT on FineWeb10B}\label{subsec:nanogpt}

We pretrain a 135M-parameter nanoGPT \citep{Jordan2024Modded} on FineWeb10B \citep{Penedo2024Fineweb} using the \texttt{modded-nanogpt} speedrun codebase of~\cite{Jordan2024Modded}, following the setup adopted in recent analyses of inexact and low-rank \algname{Muon} \citep{shulgin2025beyond,he2025low}. We report validation perplexity on the FineWeb10B validation set \citep{radford2019language,vaswani2017attention,Touvron2023LLaMA} to amplify small differences in validation loss. Each configuration is averaged over $5$ random seeds. Table~\ref{tab:nanogpt_e5} reports the experimental results, and Figure~\ref{fig:nanogpt_e5} plots the convergence trajectories for the optimizer comparison and the effect of varying the randomized rank. Full implementation details are deferred to Appendix~\ref{app:expdetails_nanogpt}.

\begin{figure}[ht]
    \centering
    \begin{subfigure}[t]{0.5\textwidth}
        \centering
        \includegraphics[width=\linewidth]{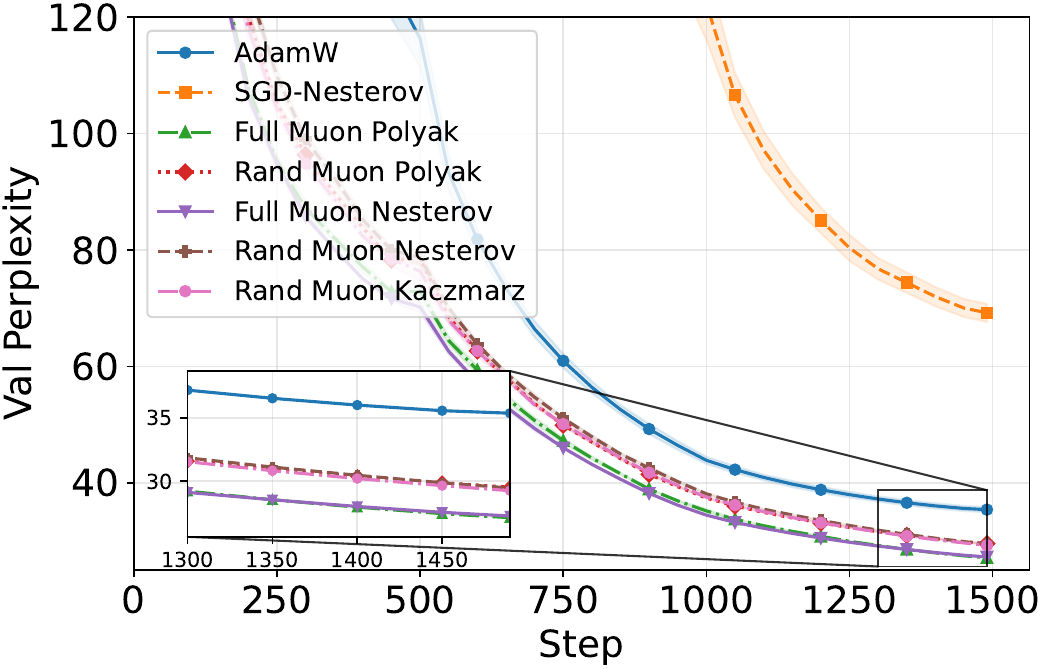}
        \caption{nanoGPT: optimizer}\label{fig:nanogpt_e5_methods}
    \end{subfigure}\hfill
    \begin{subfigure}[t]{0.5\textwidth}
        \centering
        \includegraphics[width=\linewidth]{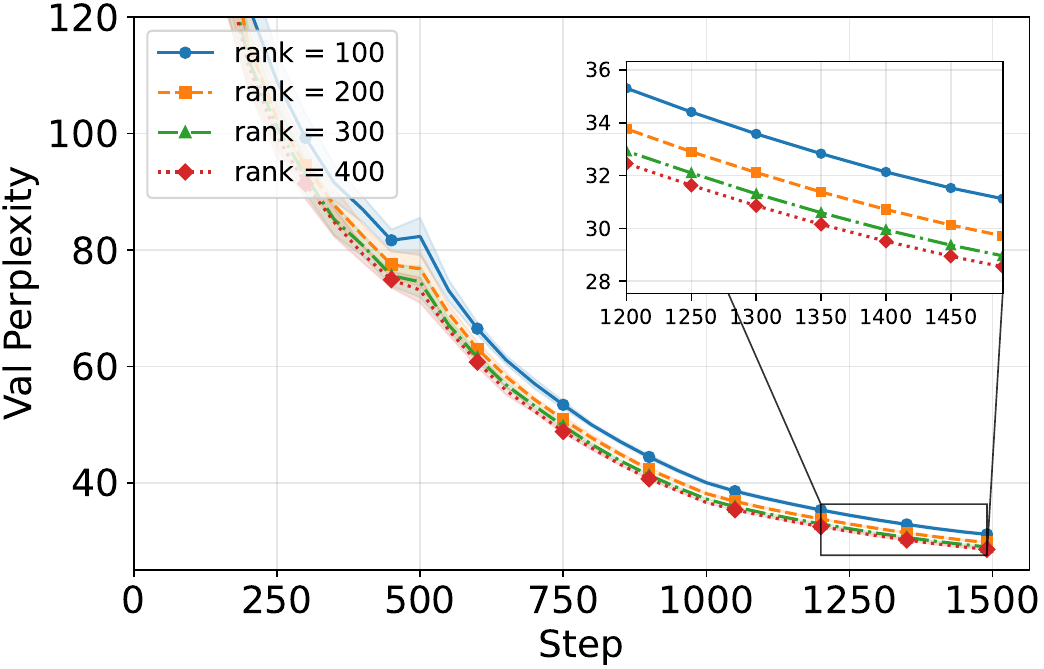}
        \caption{nanoGPT: rank $s$}\label{fig:nanogpt_abl_e2_rank}
    \end{subfigure}\hfill
    \begin{subfigure}[t]{0.5\textwidth}
        \centering
        \includegraphics[width=\linewidth]{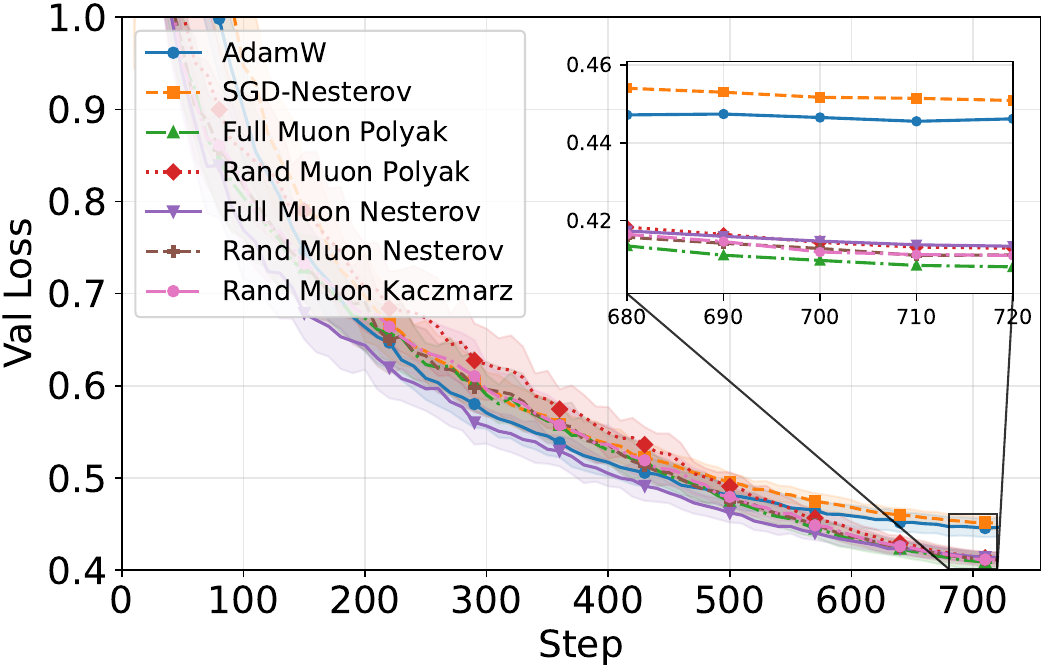}
        \caption{CIFAR-10: optimizer}\label{fig:cifar10_e5_methods}
    \end{subfigure}\hfill
    \begin{subfigure}[t]{0.5\textwidth}
        \centering
        \includegraphics[width=\linewidth]{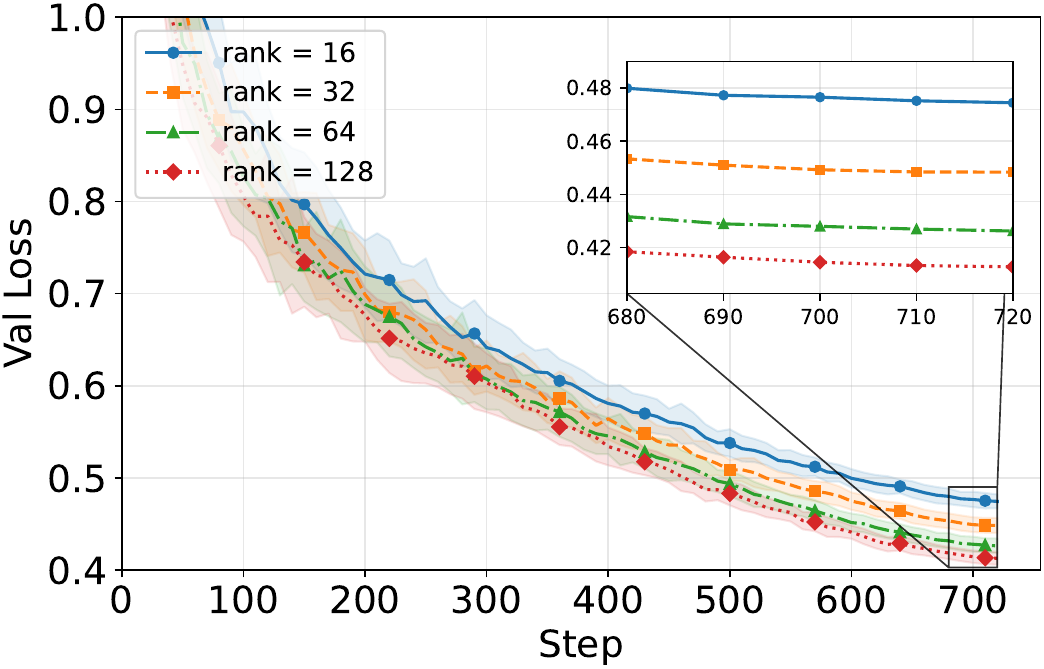}
        \caption{CIFAR-10: rank $s$}\label{fig:abl_e2_rank}
    \end{subfigure}
    \caption{\small Convergence trajectories on the two benchmarks. (a,b) nanoGPT: optimizer comparison and randomized \algname{Muon} (Nesterov) at varying randomized rank $s$, reporting validation perplexity. (c,d) CIFAR-10: same comparison reporting validation loss.}
    \label{fig:nanogpt_e5}\label{fig:cifar10_e5}
\end{figure}

Figure~\ref{fig:nanogpt_e5_methods} shows that all \algname{Muon} variants clearly outperform \algname{AdamW} and \algname{SGD}-Nesterov. Moreover, the randomized \algname{Muon} variants perform comparably to their full-space counterparts under both Polyak and Nesterov momentum. As reported in Table~\ref{tab:nanogpt_e5}, the gap between full and randomized \algname{Muon} is reasonable under both Polyak and Nesterov momentum. This comparable result is obtained with a randomized rank of only $200$, about $26\%$ of the corresponding full rank. Figure~\ref{fig:nanogpt_abl_e2_rank} further shows that validation perplexity decreases monotonically as the randomized rank grows. Even at rank $400$, the sketch uses only about $52\%$ of the corresponding full rank while further reducing validation perplexity. Figure~\ref{fig:nanogpt_abl_e3_q} additionally shows that increasing the Newton-Schulz or PolarExpress inner-iteration count $q$ improves final perplexity monotonically, though the gains diminish beyond $q = 7$.

Despite the comparable perplexity gap, the reduction in optimizer FLOPs is substantial. To quantify this, we report per-step FLOPs \citep{Hoffmann2022Training, Narayanan2021Efficient, deepseekai2026deepseekv4} in Table \ref{tab:nanogpt_e5}. Since the forward and backward passes are identical across methods, we focus on the optimizer-update components, which capture all optimizer-specific steps. Full-space \algname{Muon} requires roughly $2136$ GFLOPs per step (where $1$ GFLOP $= 10^9$ FLOPs), whereas randomized \algname{Muon} requires only about $251$ GFLOPs, which is nearly an order-of-magnitude reduction. The Polyak variant is marginally cheaper than Nesterov since Nesterov adds one extra vector combination in the momentum update. 

The cost can be further reduced by using a canonical-basis sketch rather than a dense Gaussian sketch, so that each column of $\bm\Omega$ has a single nonzero entry. Then $\mM\bm\Omega$ reduces to a rescaled column selection rather than a dense matrix-matrix multiplication. With this Kaczmarz-style sparse sketch, the per-step optimizer cost drops to $217$ GFLOPs while keeping almost the same validation perplexity. As we can see from Figure~\ref{fig:nanogpt_abl_e5_projector} and Table~\ref{tab:nanogpt_e5}, the Kaczmarz-style drawing achieves nearly identical final perplexity to the Gaussian drawing and also has almost the same rate of decreasing while further reducing per-step computational cost. The sparse drawing scheme is described in~detail~in~Algorithm~\ref{alg:kaczmarz_rand_ns_update}.

\begin{table}[t]
    \centering
    \caption{\small Final validation perplexity for nanoGPT, test accuracy for CIFAR-10, wall-clock time, and per-step optimizer GFLOPs. Best metric, the lowest wall-clock time and GFLOPs among \algname{Muon} variants are in \textbf{bold}.}\vspace{-0.2cm}
    \label{tab:nanogpt_e5}\label{tab:cifar10_e5}
    \resizebox{\linewidth}{!}{%
    \begin{tabular}{lcccccc}
        \toprule
        \multirow{2}{*}[-2pt]{\textbf{Method}} & \multicolumn{3}{c}{\textbf{nanoGPT}} & \multicolumn{3}{c}{\textbf{CIFAR-10}} \\
         \cmidrule(lr){2-4} \cmidrule(lr){5-7}
         & \textbf{Val.~Perplexity} & \textbf{Time (s)} & \textbf{GFLOPs} & \textbf{Test Accuracy} & \textbf{Time (s)} & \textbf{GFLOPs} \\
        \midrule
        \algname{AdamW}                      & $35.3749 \pm 0.4474$ & $1227.8989 \pm 4.4531$  & $7.2621$            & $0.9179 \pm 0.0025$          & $1.7958 \pm 0.0066$          & $0.0237$           \\
        \algname{SGD} (Nes.)                 & $71.0208 \pm 1.8195$ & $1669.6667 \pm 17.5563$ & $3.1123$            & $0.9212 \pm 0.0014$          & $1.7373 \pm 0.0054$          & $0.0118$           \\
        \algname{Muon} (Polyak)              & $\mathbf{27.1115 \pm 0.0246}$ & $\mathbf{1205.9631 \pm 5.0362}$ & $2135.5947$ & $\mathbf{0.9376 \pm 0.0012}$ & $\mathbf{4.8547 \pm 0.0363}$ & $14.7920$ \\
        Rand \algname{Muon} (Polyak)         & $29.5218 \pm 0.0287$ & $1229.1785 \pm 4.5745$  & $250.8060$          & $0.9363 \pm 0.0015$          & $7.0608 \pm 0.0284$          & $8.0723$           \\
        \algname{Muon} (Nes.)                & $27.2457 \pm 0.0201$ & $1208.5847 \pm 5.8304$  & $2135.7551$         & $0.9368 \pm 0.0012$          & $4.9300 \pm 0.0225$          & $14.7960$          \\
        Rand \algname{Muon} (Nes.-Gauss.)    & $29.4565 \pm 0.0482$ & $1216.2849 \pm 13.4836$ & $250.9665$          & $0.9364 \pm 0.0013$          & $7.1153 \pm 0.0230$          & $8.0763$           \\
        Rand \algname{Muon} (Nes.-Kacz.)     & $29.2445 \pm 0.0348$ & $1237.2930 \pm 4.8522$  & $\mathbf{217.4581}$ & $0.9366 \pm 0.0011$          & $8.1434 \pm 0.0515$          & $\mathbf{7.5373}$  \\
        \bottomrule
    \end{tabular}}
\end{table}

\subsection{CNN on CIFAR-10}\label{subsec:cifar10}

To assess randomized \algname{Muon} on a different data modality, we train a small CNN on CIFAR-10~\citep{Krizhevsky2009Learning} using the CIFAR-10 Airbench codebase of~\cite{Jordan2024Cifar94}. Each configuration is averaged over $50$ random seeds. Table~\ref{tab:cifar10_e5} reports the experimental results, and Figure~\ref{fig:cifar10_e5} plots the validation-loss convergence trajectories. Full implementation details are deferred to Appendix~\ref{app:expdetails_cifar10}.

Figure~\ref{fig:cifar10_e5_methods} and Table~\ref{tab:cifar10_e5} show that the \algname{Muon} variants reach essentially the same test accuracy and outperform \algname{AdamW} and \algname{SGD}-Nesterov. The full and randomized variants are nearly indistinguishable on this benchmark. The convergence curves of Figure~\ref{fig:cifar10_e5_methods} are correspondingly indistinguishable across \algname{Muon} variants. This narrow separation may be due to the limited capacity of the compact CNN, which restricts the potential gains from orthogonalization-based optimization over well-tuned baselines. Figure~\ref{fig:abl_e2_rank} shows that validation loss decreases monotonically as the randomized rank grows from $s=16$ to $s=128$, with diminishing returns at larger ranks. Figure~\ref{fig:abl_e3_q} similarly shows that increasing the Newton-Schulz or PolarExpress inner-iteration count $q$ improves final loss, though the gains are marginal beyond $q = 5$.

Following the same FLOP-counting convention as in Section~\ref{subsec:nanogpt}, we focus on the per-step optimizer GFLOPs reported in Table~\ref{tab:cifar10_e5}. Full-space \algname{Muon} spends roughly $14.79$ GFLOPs per step, whereas randomized \algname{Muon} drops to about $8.07$ GFLOPs, approximately a $45\%$ reduction. The canonical-basis (Kaczmarz) sketch lowers the cost further to $7.54$ GFLOPs while attaining a marginally higher test accuracy than the Gaussian sketch. Figure~\ref{fig:abl_e5_projector} and Table~\ref{tab:cifar10_e5} show that the Kaczmarz-style drawing not only reduces computational cost but also attains a marginally lower final validation loss than the Gaussian drawing. As in the nanoGPT experiment, the Polyak variant is marginally~cheaper~than~Nesterov.

\begin{figure}[h]
    \centering
    \begin{subfigure}[t]{0.5\textwidth}
        \centering
        \includegraphics[width=\linewidth]{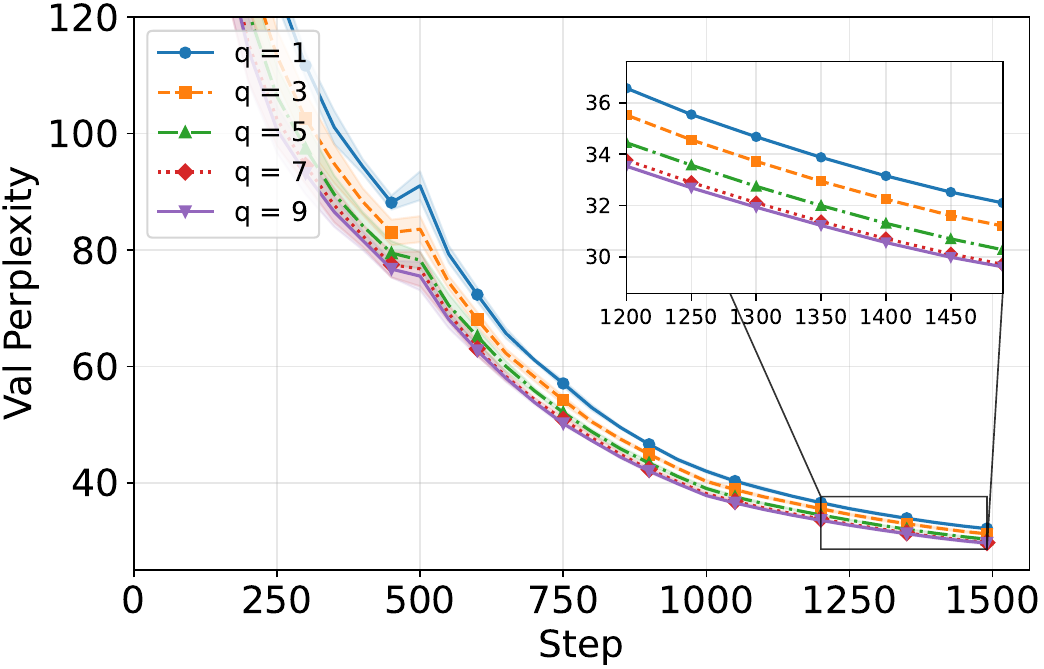}
        \caption{nanoGPT: $q$}\label{fig:nanogpt_abl_e3_q}
    \end{subfigure}\hfill
    \begin{subfigure}[t]{0.5\textwidth}
        \centering
        \includegraphics[width=\linewidth]{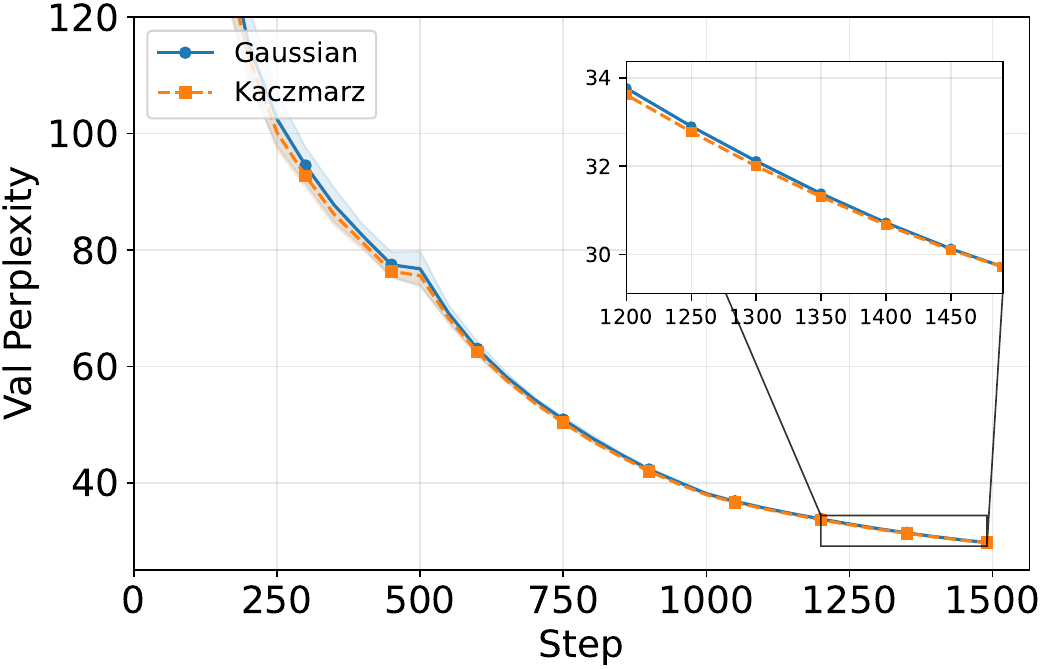}
        \caption{nanoGPT: Sketching}\label{fig:nanogpt_abl_e5_projector}
    \end{subfigure}\hfill
    \begin{subfigure}[t]{0.5\textwidth}
        \centering
        \includegraphics[width=\linewidth]{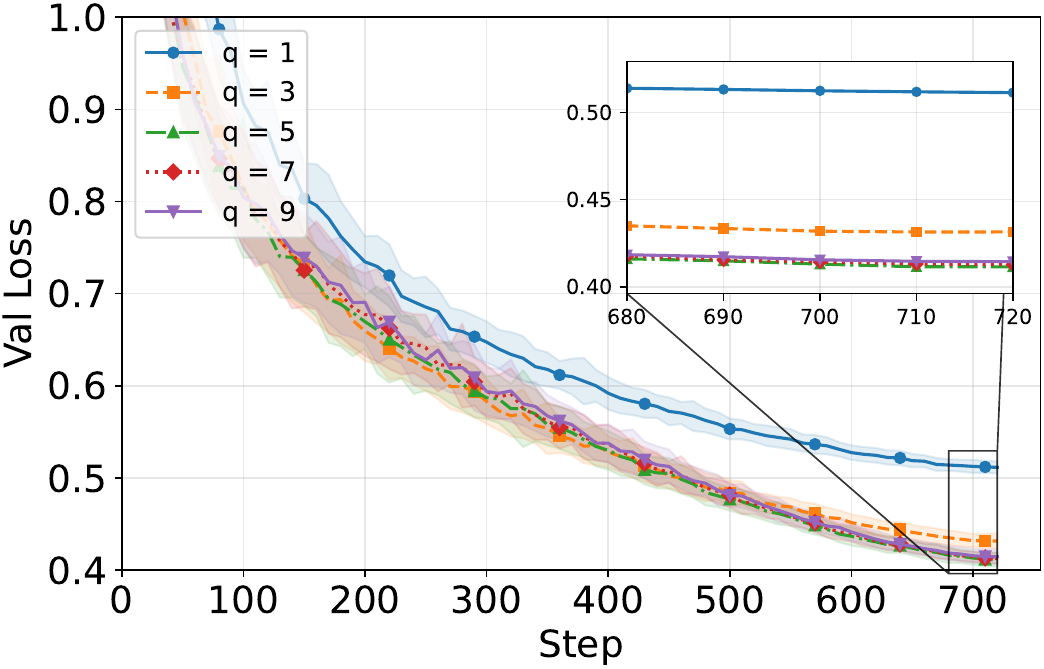}
        \caption{CIFAR-10: $q$}\label{fig:abl_e3_q}
    \end{subfigure}\hfill
    \begin{subfigure}[t]{0.5\textwidth}
        \centering
        \includegraphics[width=\linewidth]{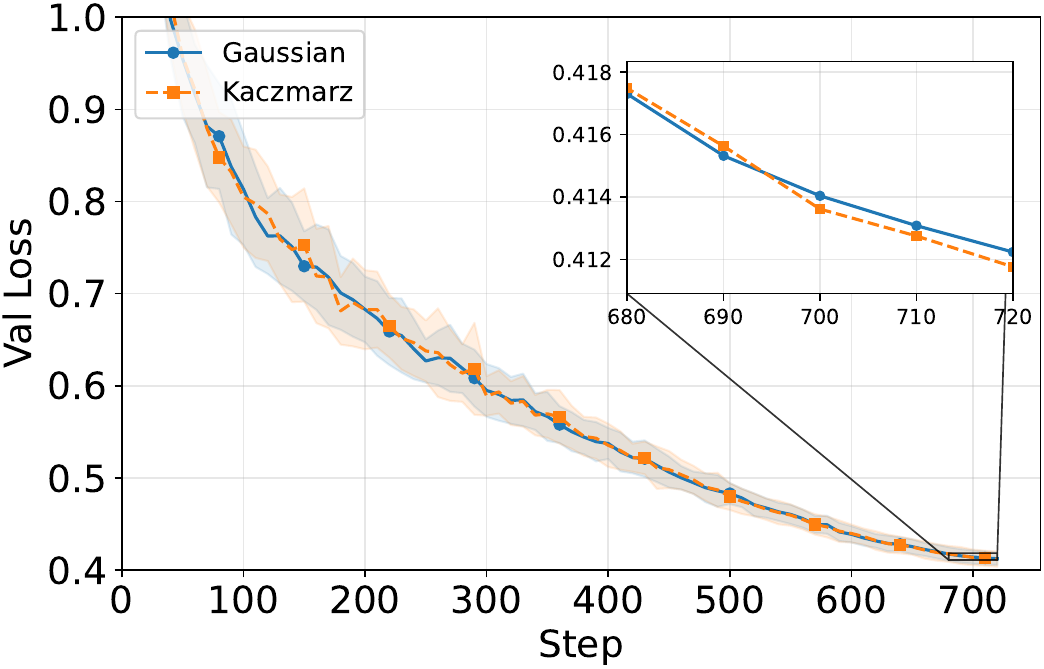}
        \caption{CIFAR-10: Sketching}\label{fig:abl_e5_projector}
    \end{subfigure}
    \caption{\small (a,b) nanoGPT: validation perplexity when varying the Newton-Schulz or PolarExpress inner-iteration count $q$ and the randomized sketching. (c,d) CIFAR-10: validation loss for the same two ablations.}
    \label{fig:ablations_main}
\end{figure}

\section{Conclusion}

We developed a convergence theory for \algname{Muon} that incorporates Nesterov momentum, inexact polar decomposition, and heavy-tailed stochastic gradients in non-convex matrix optimization.  
Our analysis introduces a unified inexact-polar framework that captures practical approximations such as Newton-Schulz, while quantifying how approximation errors influence the optimization dynamics.  
Under generalized heavy-tailed noise, uniform polar-error bounds, and uniformly bounded initial error, we established an optimal $\gO (\varepsilon^{-(3\alpha-2)/(\alpha-1)} )$ iteration and sample complexity for finding an $\varepsilon$-stationary point.
Beyond full-space polar decomposition, we further analyzed a randomized low-rank polar procedure that remains compatible with our theory and substantially reduces optimizer-side computational cost.  
Experiments on nanoGPT \mbox{pretraining}~and~CIFAR-10 demonstrated the practical effectiveness of the proposed inexact and randomized variants.  

\subsection*{AI use statement}

We used generative AI tools to assist with writing and polishing the manuscript. We have reviewed all AI-assisted text and take responsibility for the final manuscript, including text, claims, proofs, and experimental results.

\bibliography{iclr2026_conference}


\newpage
\appendix
\part*{Appendix}
\vspace{-0.2cm}

\section{Preparation Lemmas}\label{app:preparation}

\paragraph{Sampling history.}
The two histories used in the analysis correspond to different stages of
iteration $k$. The $\sigma$-field $\mathcal F_k$ contains the initialization and
all randomness from earlier iterations, including earlier mini-batches and
polar draws. Thus $\mX_k$ and $\mC_{k-1}$ are $\mathcal F_k$-measurable, whereas
the current stochastic gradients have not yet been observed. After drawing
$\mG_k^1,\ldots,\mG_k^B$ and forming $\mG_k$, $\mC_k$, and $\mM_k$, the history is
\[
\mathcal G_k=\sigma(\mathcal F_k,\mG_k^1,\ldots,\mG_k^B),
\qquad
\mathcal F_k\subseteq\mathcal G_k\subseteq\mathcal F_{k+1}.
\]
In particular, $\mM_k$, $\mS_k=\widetilde{\mM}_k-\nabla f(\mX_k)$, and
$\bm\xi_k=\mG_k-\nabla f(\mX_k)$ are $\mathcal G_k$-measurable.
The polar seed $\zeta_k$ is then drawn independently of $\mathcal G_k$ as
specified in the oracle model preceding Assumption~\ref{assume:general_inexact}, and the polar update forms
$\mX_{k+1}$. The next history $\mathcal F_{k+1}$ includes both the current
mini-batch and the polar draw.

Conditioning on $\mathcal F_k$ therefore averages over the current gradient
noise: Assumption~\ref{assume:general_bounded_moment} gives
$\Exp{\bm\xi_k\mid\mathcal F_k}=\bm0$. In contrast,
$\Exp{\bm\xi_k\mid\mathcal G_k}=\bm\xi_k$, since that noise is already known.
The polar descent proof conditions on $\mathcal G_k$ to keep $\mM_k$ and
$\mS_k$ fixed while averaging over the fresh polar randomness. These
conditioning steps cannot be interchanged. Even when the polar map is
deterministic, $\mathcal F_k$ and $\mathcal G_k$ generally remain distinct
because of the current mini-batch. Natural filtrations of generic martingale
sequences below are denoted by $\mathcal N_t$ to distinguish them from these
algorithmic histories.

\begin{lemma}
For $\mA \in \R^{m \times n}$, let $d_0 \eqdef \min\{m,n\}$. Then
\begin{equation*}
\| \mA \|_F \leq \| \mA \|_* \leq \sqrt{\operatorname{rank}(\mA)} \| \mA \|_F\quad\quad \text{and}\quad\quad \| \mA \|_F \leq \sqrt{d_0} \| \mA\|_{\mathrm{op}}.
\end{equation*}
\end{lemma}

\begin{lemma}
For matrices $\mA, \mB \in \R^{m \times n}$, we have $\left|\left\langle \mA, \mB \right\rangle\right| \leq \left\| \mA \right\|_F \left\| \mB \right\|_F$.
\end{lemma}

\begin{lemma}[\textit{Jensen's Inequality}]
Let $X$ be a random variable for which the following expectations exist.
If $\varphi$ is convex, then
\begin{equation}\label{eq:jensen_convex}
\varphi\left( \Exp{X} \right) \leq \Exp{\varphi(X)}.
\end{equation}
If $\varphi$ is concave, then
\begin{equation}\label{eq:jensen_concave}
\Exp{\varphi(X)} \leq \varphi\left( \Exp{X} \right).
\end{equation}		
\end{lemma}

\begin{lemma}
For $x, y \geq 0$ and $\alpha \in (1, 2]$, we have $\left( x + y \right)^{\nicefrac{1}{\alpha}} \leq x^{\nicefrac{1}{\alpha}} + y^{\nicefrac{1}{\alpha}}$.
\end{lemma}

\begin{lemma}
For $\mM \in \R^{m \times n}$ and $\vg \sim \gN_m(\bm{0}, \mI)$, we have	\begin{equation}\label{eq:normal_frob_lowerbound}
\Exp{ \left\| \vg^\top \mM \right\|} \geq \sqrt{\frac{2}{\pi}} \| \mM \|_F.
\end{equation}	
\end{lemma}

\begin{proof}
If $\mM=\bm{0}$, the result is immediate. Otherwise, let $\mM = \mU \bm{\Sigma} \mV^\top$ be the compact singular value decomposition of $\mM$, with rank $r$ and singular values $\sigma_1,\ldots,\sigma_r>0$. Define $\widehat{\vg} \eqdef \mU^\top \vg$. Since $\mU$ has orthonormal columns, $\widehat{\vg} \sim \gN_r(\bm{0}, \mI)$. Hence
\begin{align*}
\Exp{\| \vg^\top \mM \|}
&= \Exp{ \sqrt{ \sum_{i = 1}^r \sigma_i^2 \widehat{\vg}_i^2}} = \sqrt{\sum_{j = 1}^r \sigma_j^2}
\Exp{ \sqrt{ \sum_{i = 1}^r \frac{\sigma_i^2}{\sum_{j = 1}^r \sigma_j^2} \widehat{\vg}_i^2}} \\
&\geq \sqrt{\sum_{j = 1}^r \sigma_j^2}
\sum_{i = 1}^r \frac{\sigma_i^2}{\sum_{j = 1}^r \sigma_j^2}
\Exp{ \left| \widehat{\vg}_i \right| } = \sqrt{\frac{2}{\pi}} \sqrt{\sum_{j = 1}^r \sigma_j^2} = \sqrt{\frac{2}{\pi}} \| \mM \|_F,	
\end{align*}
where the inequality uses the concavity of $\sqrt{\cdot}$ with weights $\sigma_i^2/\sum_{j=1}^r \sigma_j^2$.	
\end{proof}

\begin{lemma}\label{lemma:NSpoly_prop}
For $\varphi(x) \eqdef \frac{3}{2}x - \frac{1}{2}x^3$, we have, for every $x \in [0,1]$,
\begin{equation*}
0\le \varphi(x)\le 1, \qquad \varphi(x)\ge x, \qquad \varphi'(x)\ge 0.
\end{equation*}
\end{lemma}

\begin{proof}
Note that for $x \in [0,1]$,
\begin{equation*}
\varphi(x)-x = \frac{1}{2}x(1-x^2)\ge 0,
\end{equation*}
so $\varphi(x)\ge x$. Also,
\begin{equation*}
1-\varphi(x) = 1-\frac{3}{2}x+\frac{1}{2}x^3 = \frac{1}{2}(1-x)^2(2+x)\ge 0,
\end{equation*}
using $x \in [0,1]$. Therefore $\varphi(x) \leq 1$. Finally,
\begin{equation*}
\varphi'(x)=\frac{3}{2}(1-x^2)\ge 0 \qquad\text{for }x\in[0,1].
\end{equation*}
\end{proof}

\begin{lemma}\label{lemma:quintic_NSpoly_prop}
For $\varphi(x) \eqdef \frac{1}{8}\left(15x-10x^3+3x^5\right)$, we have, for every $x\in[0,1]$,
\begin{equation*}
0\le \varphi(x)\le 1, \qquad \varphi(x)\ge x, \qquad \varphi'(x)\ge 0.
\end{equation*}
\end{lemma}

\begin{proof}
First, we show that $\varphi(x)\ge x$. Indeed,
\begin{equation*}
\varphi(x)-x = \frac{1}{8}\left(15x-10x^3+3x^5\right)-x  = \frac{1}{8}x\left(7-10x^2+3x^4\right) = \frac{1}{8}x(1-x^2)(7-3x^2).
\end{equation*}
Since $x\in[0,1]$, we have $x\ge 0$, $1-x^2\ge 0$, and $7-3x^2\ge 4>0$. Hence $\varphi(x)\ge x$. In particular, since $x\ge 0$, this also implies $\varphi(x)\ge 0$.

Next,
\begin{equation*}
1-\varphi(x) = 1-\frac{1}{8}\left(15x-10x^3+3x^5\right) = \frac{1}{8}\left(8-15x+10x^3-3x^5\right) = \frac{1}{8}(1-x)^3(8+9x+3x^2).
\end{equation*}
Since $x\in[0,1]$, we have $(1-x)^3\ge 0$ and $8+9x+3x^2>0$. Thus $\varphi(x)\le 1$. Finally,
\begin{align*}
\varphi'(x) = \frac{1}{8}\left(15-30x^2+15x^4\right) = \frac{15}{8}(1-x^2)^2.
\end{align*}
Hence $\varphi'(x)\ge 0$ for every $x\in[0,1]$.
\end{proof}

\begin{lemma}\label{lemma:bound_martingle}
Let $p \in (1,2]$, and let $\mW_t \in \R^{m \times n}$ be a matrix martingale difference sequence with finite $p$-th moments with respect to the natural filtration
$\mathcal N_t=\sigma(\mW_1,\ldots,\mW_t)$, with $\mathcal N_0$ trivial, meaning
\begin{equation*}
\Exp{\mW_t \mid \mathcal N_{t-1}} = \bm{0}.
\end{equation*}
Then, for any $T\ge 1$,
\begin{equation}\label{eq:buckholder_extension}
\Exp{ \left\| \sum_{t = 1}^T \mW_t \right\|_F } \leq 2 \sqrt{\pi}\Exp{ \left( \sum_{t = 1}^T \| \mW_t \|_F^p \right)^{\nicefrac{1}{p}}}.
\end{equation}	
\end{lemma}

\begin{proof}
Following \cite{kornilov2023accelerated}, let $\vg \sim \gN_m(\bm{0},\mI)$ be independent of $\mW_1,\ldots,\mW_T$, and define $\vv_t \eqdef \vg^\top \mW_t$. Since $\sigma(\vv_1,\ldots,\vv_{t-1}) \subseteq \sigma(\mW_1,\ldots,\mW_{t-1},\vg)$, the tower property and the independence of $\vg$ give
\begin{align*}
\Exp{\vv_t \mid \vv_1,\ldots,\vv_{t-1}} = \Exp{ \Exp{\vg^\top \mW_t \mid \mW_1,\ldots,\mW_{t-1},\vg} \mid \vv_1,\ldots,\vv_{t-1}} = \bm{0}.
\end{align*}
Thus, by \cite[Lemma 4.3]{liu2024nonconvex},
\begin{equation*}
\Exp{ \left\| \sum_{t = 1}^T \vv_t \right\|} \leq 2 \sqrt{2}\Exp{ \left( \sum_{t = 1}^T \| \vv_t\|^p \right)^{\nicefrac{1}{p}}}.
\end{equation*}
On the other hand, using \eqref{eq:normal_frob_lowerbound} conditionally on $\sum_{t=1}^T \mW_t$,
\begin{align*}
\Exp{ \left\| \sum_{t = 1}^T \vv_t \right\|} = \Exp{ \left\| \vg^\top \left(\sum_{t = 1}^T \mW_t\right) \right\| } \geq \sqrt{\frac{2}{\pi}} \Exp{ \left\| \sum_{t = 1}^T \mW_t \right\|_F }.
\end{align*}
It remains to bound the right-hand side. Since $(z_1+\cdots+z_T)^{1/p}$ is concave on $\mathbb R_+^T$, Jensen's inequality conditional on $\mathcal N_T$ gives
\begin{equation*}
\Exp{ \left( \sum_{t = 1}^T \| \vv_t\|^p \right)^{\nicefrac{1}{p}}} \leq
\Exp{ \left( \sum_{t = 1}^T \Exp{ \left\| \vg^\top \mW_t \right\|^p \mid \mathcal N_T} \right)^{\nicefrac{1}{p}}} =
\Exp{ \left( \sum_{t = 1}^T \Exp{ \left\| \vg^\top \mW_t \right\|^p \mid \mW_t} \right)^{\nicefrac{1}{p}}}
\end{equation*}
where the last equality uses the independence of $\vg$ from $\mathcal N_T$. For each $t$, an SVD of $\mW_t$ and the concavity of $x \mapsto x^{p/2}$ for $p \leq 2$ imply
\begin{equation*}
	\Exp{ \left\| \vg^\top \mW_t \right\|^p \mid \mW_t }
	\leq
	\| \mW_t \|_F^p.
\end{equation*}
Therefore,
\begin{equation*}
\Exp{ \left( \sum_{t = 1}^T \| \vv_t\|^p \right)^{\nicefrac{1}{p}}} \leq \Exp{ \left( \sum_{t = 1}^T \| \mW_t \|_F^p \right)^{\nicefrac{1}{p}}}.
\end{equation*}
Combining the previous displays yields
\begin{equation*}
\Exp{ \left\| \sum_{t = 1}^T \mW_t \right\|_F } \leq 2 \sqrt{\pi}\Exp{ \left( \sum_{t = 1}^T \| \mW_t \|_F^p \right)^{\nicefrac{1}{p}}}.
\end{equation*}
This completes the proof.
\end{proof}

\begin{lemma}[\textit{Weighted stochastic-gradient noise bound}]
\label{lemma:weighted_conditional_martingale}
Under Assumption~\ref{assume:general_bounded_moment}, let
$\bm\xi_s=\mG_s-\nabla f(\mX_s)$ and suppose
$\Exp{\|\nabla f(\mX_s)\|_F^\alpha}<\infty$ for $s=0,\ldots,k$.
Then, for deterministic real coefficients $a_0,\ldots,a_k$,
\begin{equation}\label{eq:weighted_conditional_martingale}
\Exp{\left\|\sum_{s=0}^k a_s\bm\xi_s\right\|_F}
\leq \frac{2\sqrt{\pi}\gC_\alpha^{1/\alpha}}{B^{(\alpha-1)/\alpha}}
\left[
\sigma_0\left(\sum_{s=0}^k|a_s|^\alpha\right)^{1/\alpha}
+\sigma_1\sum_{s=0}^k|a_s|\Exp{\|\nabla f(\mX_s)\|_F}
\right].
\end{equation}
Here $\gC_\alpha\in(0,2]$ is the constant from Lemma~\ref{lemma:batch_variance}.
\end{lemma}

\begin{proof}
By the sampling order above, $\bm\xi_s$ is $\mathcal G_s$-measurable, hence
$\mathcal F_{s+1}$-measurable. Assumption~\ref{assume:general_bounded_moment}
gives $\Exp{\bm\xi_s\mid\mathcal F_s}=\bm0$.
Lemma~\ref{lemma:batch_variance} gives
\[
\Exp{\|\bm\xi_s\|_F^\alpha\mid\mathcal F_s}
\leq \frac{\gC_\alpha}{B^{\alpha-1}}
\left(\sigma_0^\alpha+\sigma_1^\alpha\|\nabla f(\mX_s)\|_F^\alpha\right),
\]
which also ensures finite $\alpha$-th moments of $\bm\xi_s$.
Put $\mW_{s+1}=a_s\bm\xi_s$ and
$\mathcal N_s=\sigma(\mW_1,\ldots,\mW_s)$, with $\mathcal N_0$ trivial.
Since $\mathcal N_s\subseteq\mathcal F_s$, the tower property gives
\[
\Exp{\mW_{s+1}\mid\mathcal N_s}
=\Exp{a_s\Exp{\bm\xi_s\mid\mathcal F_s}\mid\mathcal N_s}=\bm0.
\]
Thus Lemma~\ref{lemma:bound_martingle}, with $p=\alpha$, yields
\begin{equation}\label{eq:weighted_realized_moment}
\Exp{\left\|\sum_{s=0}^k a_s\bm\xi_s\right\|_F}
\leq 2\sqrt{\pi}\Exp{\left(\sum_{s=0}^k|a_s|^\alpha\|\bm\xi_s\|_F^\alpha\right)^{1/\alpha}}.
\end{equation}
To bound the realized moment on the right, set $\tau=1/\alpha$,
$c_s=|a_s|^\alpha$, and define
\begin{align*}
Z_r&=\sum_{s=0}^r c_s\|\bm\xi_s\|_F^\alpha
+\frac{\gC_\alpha\sigma_0^\alpha}{B^{\alpha-1}}\sum_{s=r+1}^k c_s
\quad(0\leq r\leq k), \\
Z_{-1}&=\frac{\gC_\alpha\sigma_0^\alpha}{B^{\alpha-1}}\sum_{s=0}^k c_s.
\end{align*}
For each $r$, all terms involving $\bm\xi_0,\ldots,\bm\xi_{r-1}$ are
$\mathcal F_r$-measurable. Conditional Jensen's inequality for the concave
function $x\mapsto x^\tau$, the conditional moment bound, and
$(x+y)^\tau\leq x^\tau+y^\tau$ give
\begin{align*}
\Exp{Z_r^\tau\mid\mathcal F_r}
&\leq
\left(
\sum_{s=0}^{r-1}c_s\|\bm\xi_s\|_F^\alpha
+c_r\Exp{\|\bm\xi_r\|_F^\alpha\mid\mathcal F_r}
+\frac{\gC_\alpha\sigma_0^\alpha}{B^{\alpha-1}}\sum_{s=r+1}^k c_s
\right)^\tau \\
&\leq \left(Z_{r-1}+\frac{\gC_\alpha\sigma_1^\alpha c_r}{B^{\alpha-1}}
\|\nabla f(\mX_r)\|_F^\alpha\right)^\tau \\
&\leq Z_{r-1}^\tau+
\frac{\gC_\alpha^{1/\alpha}\sigma_1|a_r|}{B^{(\alpha-1)/\alpha}}
\|\nabla f(\mX_r)\|_F.
\end{align*}
Taking expectations and iterating this inequality from $r=k$ down to $r=0$ yields
\begin{multline*}
\Exp{\left(\sum_{s=0}^k|a_s|^\alpha\|\bm\xi_s\|_F^\alpha\right)^{1/\alpha}}
=\Exp{Z_k^\tau} \\
\leq \frac{\gC_\alpha^{1/\alpha}}{B^{(\alpha-1)/\alpha}}
\left[
\sigma_0\left(\sum_{s=0}^k|a_s|^\alpha\right)^{1/\alpha}
+\sigma_1\sum_{s=0}^k|a_s|\Exp{\|\nabla f(\mX_s)\|_F}
\right].
\end{multline*}
Combining this with \eqref{eq:weighted_realized_moment} proves the claim.
\end{proof}

The proof conditions one time index at a time, keeping the gradient-dependent
contribution in terms of $\Exp{\|\nabla f(\mX_s)\|_F}$, as required by the
descent argument. It does not replace all realized moments inside the
fractional power by their conditional moments.

\begin{lemma}\label{lem:det_range_error_powered_sketch}
Let $\mM \in \R^{m \times n}$ have rank $r$ and compact singular value decomposition
\[
\mM
=
\begin{bmatrix} \mU_1 & \mU_2 \end{bmatrix}
\begin{bmatrix}
\bm\Sigma_1 & \bm0\\
\bm0 & \bm\Sigma_2
\end{bmatrix}
\begin{bmatrix} \mV_1 & \mV_2 \end{bmatrix}^{\top},
\]
where $\sigma_1 \ge \cdots \ge \sigma_r>0$ and
\[
\begin{aligned}
\bm\Sigma_1=\operatorname{diag}(\sigma_1,\dots,\sigma_s),\qquad
\bm\Sigma_2=\operatorname{diag}(\sigma_{s+1},\dots,\sigma_r),
\end{aligned}
\]
for some $1\leq s<r$. Define
\[
\varrho_s \eqdef \frac{\sigma_{s+1}}{\sigma_s}.
\]
Let $p\ge 0$ be an integer, set $\ell=s+p$, and let $\bm\Omega\in\R^{n\times \ell}$. Define
\[
\begin{aligned}
\bm\Omega_1 := \mV_1^\top \bm\Omega \in \R^{s\times \ell},\qquad
\bm\Omega_2 := \mV_2^\top \bm\Omega \in \R^{(r-s)\times \ell}.
\end{aligned}
\]
Assume that $\bm\Omega_1$ has full row rank. For an integer $h\geq 0$, form
\[
\begin{aligned}
\mY := (\mM\mM^\top)^h \mM \bm\Omega,\qquad
\mQ := \operatorname{orth}(\mY),\qquad
\mP_{\mQ} := \mQ\mQ^\top,
\end{aligned}
\]
where $\operatorname{orth}(\mY)$ denotes any matrix whose columns form an orthonormal basis for $\operatorname{range}(\mY)$. Then
\begin{equation}
\label{eq:det_range_error_powered_sketch}
\|(\mI-\mP_{\mQ})\mM\|_F^2
\le
\|\bm\Sigma_2\|_F^2
+
\varrho_s^{4h}
\|\bm\Sigma_2\bm\Omega_2\bm\Omega_1^\dagger\|_F^2 .
\end{equation}
Consequently,
\begin{equation}
\label{eq:det_projected_energy_powered_sketch}
\|\mP_{\mQ}\mM\|_F^2
\ge
\|\bm\Sigma_1\|_F^2
-
\varrho_s^{4h}
\|\bm\Sigma_2\bm\Omega_2\bm\Omega_1^\dagger\|_F^2 .
\end{equation}
Moreover, if $\bm\Omega\sim \mathcal{N}(0,1)^{n\times \ell}$ with $\ell=s+p$ and $p\geq 2$, then $\bm\Omega_1$ has full row rank almost surely and
\begin{equation}
\label{eq:expected_projected_energy_powered_sketch}
\mathbb{E}_{\bm\Omega}\|\mP_{\mQ}\mM\|_F^2
\ge
\left[
\sum_{j=1}^{s}\sigma_j^2
-
\frac{s}{p-1}\varrho_s^{4h}\sum_{j>s}\sigma_j^2
\right]_+ .
\end{equation}
\end{lemma}

\begin{proof}
Set
\begin{equation*}
\mU_r := \begin{bmatrix} \mU_1 & \mU_2 \end{bmatrix},\quad\quad \mV_r := \begin{bmatrix} \mV_1 & \mV_2 \end{bmatrix},\quad\quad
\bm\Sigma :=
\begin{bmatrix}
\bm\Sigma_1 & \bm0\\
\bm0 & \bm\Sigma_2
\end{bmatrix}.
\end{equation*}
Then $\mM=\mU_r\bm\Sigma\mV_r^\top$. Since
\[
(\mM\mM^\top)^h\mM
=
\mU_r \bm\Sigma^{2h+1}\mV_r^\top,
\]
we have
\[
\mY
=
\mU_r
\begin{bmatrix}
\bm\Sigma_1^{2h+1}\bm\Omega_1\\
\bm\Sigma_2^{2h+1}\bm\Omega_2
\end{bmatrix}.
\]
Define
\[
\mG :=
\begin{bmatrix}
\bm\Sigma_1^{2h+1}\bm\Omega_1\\
\bm\Sigma_2^{2h+1}\bm\Omega_2
\end{bmatrix}
\in \R^{r\times \ell}.
\]
Then $\mY=\mU_r\mG$. Let $\mP_{\mG}$ denote the orthogonal projector onto
$\operatorname{range}(\mG)$. Since $\mU_r$ has orthonormal columns, the orthogonal projector onto
$\operatorname{range}(\mY)=\mU_r\operatorname{range}(\mG)$ is
\[
\mP_{\mQ}=\mU_r\mP_{\mG}\mU_r^\top.
\]
Therefore,
\begin{multline}\label{eq:reduce_to_singular_coordinates}
\|(\mI-\mP_{\mQ})\mM\|_F =
\|(\mI-\mU_r\mP_{\mG}\mU_r^\top)\mU_r\bm\Sigma\mV_r^\top\|_F \\ =
\|\mU_r(\mI-\mP_{\mG})\bm\Sigma\mV_r^\top\|_F =
\|(\mI-\mP_{\mG})\bm\Sigma\|_F .
\end{multline}
Since $\bm\Omega_1$ has full row rank, $\bm\Omega_1\bm\Omega_1^\dagger=\mI_s$. Define
\[
\mF
:=
\bm\Sigma_2^{2h+1}
\bm\Omega_2
\bm\Omega_1^\dagger
\bm\Sigma_1^{-(2h+1)}
\in \R^{(r-s)\times s}.
\]
Then
\begin{equation}\label{eq:range_embedding}
\mG\bm\Omega_1^\dagger \bm\Sigma_1^{-(2h+1)} =
\begin{bmatrix}
\bm\Sigma_1^{2h+1}\bm\Omega_1\\
\bm\Sigma_2^{2h+1}\bm\Omega_2
\end{bmatrix}
\bm\Omega_1^\dagger \bm\Sigma_1^{-(2h+1)} =
\begin{bmatrix}
\mI_s\\
\mF
\end{bmatrix}.
\end{equation}
Thus
\[
\operatorname{range}
\begin{bmatrix}
\mI_s\\
\mF
\end{bmatrix}
\subseteq
\operatorname{range}(\mG).
\]
By the least-squares characterization of the orthogonal projector,
\[
\|(\mI-\mP_{\mG})\bm\Sigma\|_F
=
\min_{\mX\in\R^{\ell\times r}}\|\bm\Sigma-\mG\mX\|_F.
\]
Using \eqref{eq:range_embedding}, for any $\mC\in\R^{s\times r}$,
\[
\begin{bmatrix}
\mI_s\\
\mF
\end{bmatrix}\mC
=
\mG\bm\Omega_1^\dagger \bm\Sigma_1^{-(2h+1)}\mC
\]
lies in the class of admissible approximants $\mG\mX$. Hence
\[
\|(\mI-\mP_{\mG})\bm\Sigma\|_F
\le
\left\|
\bm\Sigma
-
\begin{bmatrix}
\mI_s\\
\mF
\end{bmatrix}\mC
\right\|_F
\]
for every $\mC\in\R^{s\times r}$. Taking
\[
\mC :=
\begin{bmatrix}
\bm\Sigma_1 & \bm0
\end{bmatrix},
\]
we obtain
\begin{multline}\label{eq:range_error_before_power_bound}
\|(\mI-\mP_{\mG})\bm\Sigma\|_F^2
\le
\left\|
\begin{bmatrix}
\bm\Sigma_1 & \bm0\\
\bm0 & \bm\Sigma_2
\end{bmatrix}
-
\begin{bmatrix}
\mI_s\\
\mF
\end{bmatrix}
\begin{bmatrix}
\bm\Sigma_1 & \bm0
\end{bmatrix}
\right\|_F^2 \\=
\left\|
\begin{bmatrix}
\bm0 & \bm0\\
-\mF\bm\Sigma_1 & \bm\Sigma_2
\end{bmatrix}
\right\|_F^2 =
\|\bm\Sigma_2\|_F^2+\|\mF\bm\Sigma_1\|_F^2 .
\end{multline}
It remains to bound $\|\mF\bm\Sigma_1\|_F$. By the definition of $\mF$,
\[
\mF\bm\Sigma_1
=
\bm\Sigma_2^{2h}
\left(
\bm\Sigma_2
\bm\Omega_2
\bm\Omega_1^\dagger
\right)
\bm\Sigma_1^{-2h}.
\]
Using $\|\mA\mX\mB\|_F\leq \|\mA\|_{\mathrm{op}}\|\mX\|_F\|\mB\|_{\mathrm{op}}$, we get
\begin{equation}\label{eq:F_bound}
\|\mF\bm\Sigma_1\|_F
\le
\|\bm\Sigma_2^{2h}\|_{\mathrm{op}}
\|\bm\Sigma_2\bm\Omega_2\bm\Omega_1^\dagger\|_F
\|\bm\Sigma_1^{-2h}\|_{\mathrm{op}} =
\varrho_s^{2h}
\|\bm\Sigma_2\bm\Omega_2\bm\Omega_1^\dagger\|_F .
\end{equation}
Squaring \eqref{eq:F_bound} and substituting into
\eqref{eq:range_error_before_power_bound} gives
\[
\|(\mI-\mP_{\mG})\bm\Sigma\|_F^2
\le
\|\bm\Sigma_2\|_F^2
+
\varrho_s^{4h}
\|\bm\Sigma_2\bm\Omega_2\bm\Omega_1^\dagger\|_F^2.
\]
Together with \eqref{eq:reduce_to_singular_coordinates}, this proves
\eqref{eq:det_range_error_powered_sketch}.

Next, since $\mP_{\mQ}$ is an orthogonal projector,
\[
\|\mM\|_F^2
=
\|\mP_{\mQ}\mM\|_F^2
+
\|(\mI-\mP_{\mQ})\mM\|_F^2 .
\]
Also,
\[
\|\mM\|_F^2
=
\|\bm\Sigma_1\|_F^2+\|\bm\Sigma_2\|_F^2.
\]
Combining these identities with \eqref{eq:det_range_error_powered_sketch} yields
\[
\|\mP_{\mQ}\mM\|_F^2
\ge
\|\bm\Sigma_1\|_F^2
-
\varrho_s^{4h}
\|\bm\Sigma_2\bm\Omega_2\bm\Omega_1^\dagger\|_F^2,
\]
which proves \eqref{eq:det_projected_energy_powered_sketch}.

It remains to prove the Gaussian expectation bound. Suppose
$\bm\Omega\sim \mathcal{N}(0,1)^{n\times \ell}$ with $\ell=s+p$ and $p\geq 2$.
By rotational invariance of the Gaussian distribution, after extending
$\begin{bmatrix}\mV_1 & \mV_2\end{bmatrix}$ to an orthogonal matrix, the blocks
$\bm\Omega_1=\mV_1^\top\bm\Omega$ and
$\bm\Omega_2=\mV_2^\top\bm\Omega$ are independent standard Gaussian matrices.
Moreover, $\bm\Omega_1\in\R^{s\times (s+p)}$ has full row rank almost surely.

Conditioning on $\bm\Omega_1$, we have
\begin{align}
\mathbb{E}_{\bm\Omega_2}
\left[
\|\bm\Sigma_2\bm\Omega_2\bm\Omega_1^\dagger\|_F^2
\mid
\bm\Omega_1
\right]
&=
\operatorname{tr}
\left(
(\bm\Omega_1^\dagger)^\top
\mathbb{E}_{\bm\Omega_2}
\left[
\bm\Omega_2^\top
\bm\Sigma_2^2
\bm\Omega_2
\right]
\bm\Omega_1^\dagger
\right).
\end{align}
Since $\bm\Omega_2$ has independent standard Gaussian entries,
\[
\mathbb{E}_{\bm\Omega_2}
\left[
\bm\Omega_2^\top
\bm\Sigma_2^2
\bm\Omega_2
\right]
=
\operatorname{tr}(\bm\Sigma_2^2)\mI_{\ell}
=
\|\bm\Sigma_2\|_F^2\mI_{\ell}.
\]
Therefore,
\[
\mathbb{E}_{\bm\Omega_2}
\left[
\|\bm\Sigma_2\bm\Omega_2\bm\Omega_1^\dagger\|_F^2
\mid
\bm\Omega_1
\right]
=
\|\bm\Sigma_2\|_F^2
\|\bm\Omega_1^\dagger\|_F^2.
\]
Taking expectation over $\bm\Omega_1$ gives
\[
\mathbb{E}_{\bm\Omega}
\|\bm\Sigma_2\bm\Omega_2\bm\Omega_1^\dagger\|_F^2
=
\|\bm\Sigma_2\|_F^2
\mathbb{E}_{\bm\Omega_1}\|\bm\Omega_1^\dagger\|_F^2.
\]
Since $\bm\Omega_1$ has full row rank almost surely,
\[
\bm\Omega_1^\dagger
=
\bm\Omega_1^\top(\bm\Omega_1\bm\Omega_1^\top)^{-1},
\]
and hence
\[
\|\bm\Omega_1^\dagger\|_F^2
=
\operatorname{tr}\left((\bm\Omega_1\bm\Omega_1^\top)^{-1}\right).
\]
The matrix $\bm\Omega_1\bm\Omega_1^\top$ has the Wishart distribution
$W_s(\ell,\mI_s)$. Since $\ell=s+p$ and $p\geq 2$, the inverse moment exists and
\[
\mathbb{E}
\left[
(\bm\Omega_1\bm\Omega_1^\top)^{-1}
\right]
=
\frac{1}{\ell-s-1}\mI_s
=
\frac{1}{p-1}\mI_s.
\]
Thus
\[
\mathbb{E}_{\bm\Omega_1}\|\bm\Omega_1^\dagger\|_F^2
=
\frac{s}{p-1},
\]
and therefore
\[
\mathbb{E}_{\bm\Omega}
\|\bm\Sigma_2\bm\Omega_2\bm\Omega_1^\dagger\|_F^2
=
\frac{s}{p-1}\|\bm\Sigma_2\|_F^2.
\]
Taking expectations in \eqref{eq:det_projected_energy_powered_sketch}, we get
\[
\mathbb{E}_{\bm\Omega}\|\mP_{\mQ}\mM\|_F^2
\ge
\|\bm\Sigma_1\|_F^2
-
\frac{s}{p-1}\varrho_s^{4h}\|\bm\Sigma_2\|_F^2.
\]
Finally, using
\[
\|\bm\Sigma_1\|_F^2=\sum_{j=1}^s\sigma_j^2,
\qquad
\|\bm\Sigma_2\|_F^2=\sum_{j>s}\sigma_j^2,
\]
and the trivial bound $\|\mP_{\mQ}\mM\|_F^2\geq 0$, we obtain
\[
\mathbb{E}_{\bm\Omega}\|\mP_{\mQ}\mM\|_F^2
\ge
\left[
\sum_{j=1}^{s}\sigma_j^2
-
\frac{s}{p-1}\varrho_s^{4h}\sum_{j>s}\sigma_j^2
\right]_+ .
\]
This proves \eqref{eq:expected_projected_energy_powered_sketch}.
\end{proof}

\begin{lemma}
\label{lemma:batch_variance}
Under Assumption \ref{assume:general_bounded_moment}, we have
\begin{equation}\label{eq:general_bounded_moment_Gk}
\textstyle
\Exp{\left\| \mG_k - \nabla f(\mX_k) \right\|_F^\alpha \mid \mathcal{F}_k}
\le
\frac{\gC_\alpha}{B^{\alpha - 1}}
\left(
\sigma_0^\alpha
+
\sigma_1^\alpha \| \nabla f(\mX_k)\|_F^\alpha
\right)
\end{equation}
for some $0 < \gC_\alpha \leq 2$ depending only on $\alpha$ and independent of $B$.
Consequently, the same bound holds with conditioning on $\mX_k$ in place of $\mathcal F_k$.
\end{lemma}

\begin{proof}
From the definition $\mG_k \eqdef B^{-1}\sum_{i = 1}^B \mG_k^i$, conditional on $\mathcal F_k$ we have
\begin{align*}
\Exp{\left\| \mG_k - \nabla f(\mX_k) \right\|_F^\alpha \mid \mathcal F_k}
&=
\frac{1}{B^\alpha}
\Exp{\left\| \sum_{i = 1}^B
\left( \mG_k^i - \nabla f(\mX_k) \right)
\right\|_F^\alpha \mid \mathcal F_k} \\
&\leq
\frac{\gC_\alpha}{B^\alpha}
\sum_{i = 1}^B
\Exp{
\left\| \mG_k^i - \nabla f(\mX_k) \right\|_F^\alpha
\mid \mathcal F_k
},
\end{align*}
where the inequality follows from the von Bahr-Esseen inequality, applied conditionally on $\mathcal F_k$ to the independent mean-zero random matrices
$\mG_k^i-\nabla f(\mX_k)$, viewed as elements of the Hilbert space equipped with the Frobenius norm. For any $\alpha \in (1,2]$, one may take $\gC_\alpha \leq 2$. Using Assumption \ref{assume:general_bounded_moment}, we obtain
\begin{align*}
\Exp{\left\| \mG_k - \nabla f(\mX_k) \right\|_F^\alpha \mid \mathcal F_k}
&\leq
\frac{\gC_\alpha}{B^\alpha}
\sum_{i = 1}^B
\left(
\sigma_0^\alpha
+
\sigma_1^\alpha \| \nabla f(\mX_k)\|_F^\alpha
\right) \\
&=
\frac{\gC_\alpha}{B^{\alpha - 1}}
\left(
\sigma_0^\alpha
+
\sigma_1^\alpha \| \nabla f(\mX_k)\|_F^\alpha
\right).
\end{align*}
Since $\mX_k$ is $\mathcal F_k$-measurable, taking conditional expectation of the preceding bound with respect to $\mX_k$ yields the version conditioned on $\mX_k$.
\end{proof}

\section{Missing Proofs and Discussions of Section \ref{section:motivation}}

\subsection{Proof of Proposition \ref{prop:ns}}

\propns*

\begin{proof}
We prove the result for $\varphi (x) = \frac{3}{2} x - \frac{1}{2}x^3$. From the Newton-Schulz iterations, we have $\mZ_0 = \nicefrac{\mM}{\delta}$ and $\mZ_{t+1} = \varphi \left( \mZ_t \right)$ for $t = 0, 1, \cdots, q-1$. Define $s_i \eqdef \frac{\sigma_i}{\delta}$ for $i = 1, \cdots r$. Since $\delta \geq \|\mM\|_{\rm op}$, we have $s_i \in (0,1]$. First, we show for every $t\ge 0$,
\begin{equation*}
\mZ_t =  \mU \operatorname{diag}\bigl(p_t(s_1),\dots,p_t(s_r)\bigr) \mV^\top.
\end{equation*}
We prove this by induction on $t$. For $t=0$, this is immediate from the SVD of $\mM$:
\begin{equation*}
\mZ_0  =  \frac{\mM}{\delta}  =  \mU \operatorname{diag} \left(\frac{\sigma_1}{\delta}, \dots,\frac{\sigma_r}{\delta} \right) \mV^\top  =   \mU \operatorname{diag}(s_1,\dots,s_r) \mV^\top.        
\end{equation*}
Assume the representation holds at time $t$, i.e., $\mZ_t = \mU \mD_t \mV^\top$, where
\begin{equation*}
\mD_t:=\operatorname{diag}\bigl(p_t(s_1),\dots,p_t(s_r)\bigr).
\end{equation*}
Then $\mZ_t \mZ_t^\top \mZ_t =  (\mU \mD_t \mV^\top)(\mV \mD_t \mU^\top)(\mU \mD_t \mV^\top) =  \mU \mD_t^3 \mV^\top$. Therefore
\begin{equation*}
\mZ_{t+1} = \frac32 \mU \mD_t \mV^\top - \frac12 \mU \mD_t^3 \mV^\top =  \mU\left(\frac32 \mD_t - \frac12 \mD_t^3\right) \mV^\top.    
\end{equation*}
Since $\mD_t$ is diagonal,
\begin{equation*}
\frac32 \mD_t - \frac12 \mD_t^3 = \operatorname{diag}\bigl(\varphi(p_t(s_1)),\dots, \varphi(p_t(s_r))\bigr) = \operatorname{diag}\bigl(p_{t+1}(s_1),\dots,p_{t+1}(s_r)\bigr).    
\end{equation*}
The last line follows from the definition $p_{t+1} (x) = \varphi (p_t(x))$. Thus, the claim follows by induction.

By induction and Lemma \ref{lemma:NSpoly_prop}, we have $0\le p_q(x)\le 1$ and $p_q(x)\ge x$ for all $x\in[0,1]$. Hence, for $s_i=\sigma_i/\delta \in (0,1]$, we obtain $0<p_q(s_i)\le 1$. Therefore,
\begin{equation*}
    \|\gT (\mM)\|_{\mathrm{op}} = \left\| p_q \left( \nicefrac{\mM}{\delta} \right) \right\|_{\rm op} = \max_{1\le i\le r} p_q(s_i) \leq 1.
\end{equation*}
Thus, the second condition of Assumption~\ref{assume:general_inexact} holds with $\nu = 0$. For the first condition, we have
\begin{align*}
\langle \mM, \gT(\mM) \rangle & = \operatorname{tr}( \mM^\top \gT(\mM)) =  \operatorname{tr} \left(( \mV \bm\Sigma \mU^\top)(\mU \mD_q \mV^\top)\right)  =  \operatorname{tr}( \mV \bm\Sigma \mD_q \mV^\top) \\
& =  \operatorname{tr}(\bm\Sigma \mD_q) =  \sum_{i=1}^r \sigma_i p_q(s_i) =  \left(\frac{\sum_{i=1}^r \sigma_i p_q(s_i)}{\sum_{i=1}^r \sigma_i} \right)\| \mM \|_*.    
\end{align*}
Since $0<p_q(s_i)\le 1$, the stated value of $\gamma$ belongs to $[0,1)$. Because $\gT(\mM)$ is deterministic in this proposition, the same bounds hold for the conditional expectations in Assumption~\ref{assume:general_inexact}. This concludes the proof for the cubic polynomial. The proof for $\varphi(x) = \frac{1}{8}(15x-10x^3+3x^5)$ follows similarly using Lemma \ref{lemma:quintic_NSpoly_prop}.
\end{proof}

\subsection{Proof of Proposition \ref{prop:randomized_polar}}

\proprandomizedpolar*

\begin{proof}
Let $p_0(x)=x$ and $p_{t+1}(x)=\varphi(p_t(x))$. Fix a realization of
$\bm\Omega$ and write
$\mB=\mQ^\top\mM=\widehat \mU \widehat{\bm\Sigma}\widehat \mV^\top$,
where $\widehat r=\operatorname{rank}(\mB)$ and
$\widehat{\bm\Sigma}=\operatorname{diag}(\widehat\sigma_1,\ldots,\widehat\sigma_{\widehat r})$.
Since $\mQ$ has orthonormal columns, $\|\mB\|_{\mathrm{op}}\le \|\mM\|_{\mathrm{op}}$,
and hence $\delta\ge \|\mB\|_{\mathrm{op}}$. The scalar bounds in
Lemmas \ref{lemma:NSpoly_prop} and \ref{lemma:quintic_NSpoly_prop} imply
$0\le p_q(x)\le 1$ for all $x\in[0,1]$. Therefore,
\begin{align*}
\left\|\mZ_q\right\|_{\mathrm{op}}
&=
\left\|p_q\left(\nicefrac{\mB}{\delta}\right)\right\|_{\mathrm{op}}
=
\max_{1\le i\le \widehat r}
p_q\left(\nicefrac{\widehat\sigma_i}{\delta}\right)
\le 1.
\end{align*}
Since $\gT_{h,q}(\mM;\bm\Omega)=\mQ\mZ_q$ and $\mQ^\top\mQ=\mI$, this gives
$\|\gT_{h,q}(\mM;\bm\Omega)\|_{\mathrm{op}}\le 1$.

For the alignment term, we have
\begin{align*}
\langle \mM,\gT_{h,q}(\mM;\bm\Omega)\rangle
&=
\langle \mM,\mQ\mZ_q\rangle
=
\langle \mQ^\top\mM,\mZ_q\rangle
=
\langle \mB,\mZ_q\rangle \\
&=
\sum_{i=1}^{\widehat r}
\widehat\sigma_i
p_q\left(\nicefrac{\widehat\sigma_i}{\delta}\right)
\ge
\sum_{i=1}^{\widehat r}
\nicefrac{\widehat\sigma_i^2}{\delta}
=
\nicefrac{\|\mB\|_F^2}{\delta},
\end{align*}
where the inequality uses $p_q(x)\ge x$ on $[0,1]$. Let
$\mP_{\mQ}\eqdef\mQ\mQ^\top$. Since $\mQ^\top\mQ=\mI$,
\[
\|\mB\|_F^2
=
\|\mQ^\top\mM\|_F^2
=
\|\mP_{\mQ}\mM\|_F^2.
\]
Thus,
\[
\langle \mM,\gT_{h,q}(\mM;\bm\Omega)\rangle
\ge
\frac{\|\mP_{\mQ}\mM\|_F^2}{\delta}.
\]
Taking expectation with respect to $\bm\Omega$ and applying
Lemma \ref{lem:det_range_error_powered_sketch} gives
\[
\mathbb{E}_{\bm\Omega}
\left[\langle \mM,\gT_{h,q}(\mM;\bm\Omega)\rangle\right]
\ge
\frac{1}{\delta}
\left[
\sum_{j=1}^s \sigma_j^2
-
\frac{s}{p-1}\varrho_s^{4h}
\sum_{j>s}\sigma_j^2
\right]_+ .
\]
This completes the proof.
\end{proof}

\subsection{Uniform Bounds for Inexact Polar Decompositions}
\label{app:uniform_polar}

The following bounds provide deterministic oracle constants that are independent
of the momentum matrices, their ranks, and the optimization horizon.

\begin{lemma}[\textit{Uniform polar-oracle bounds}]
\label{lem:uniform_polar_bounds}
Let $d_0=\min\{m,n\}\ge1$, let $h,q$ be nonnegative integers, and let $p_q$
be the $q$-fold iterate of either the cubic or the theoretical quintic
Newton--Schulz polynomial in Section~\ref{section:motivation}, with $p_0(x)=x$.
For $\mM\ne\bm0$, set $\delta=\|\mM\|_F$ and define
\[
\gT_{\rm full}(\mM)=p_q(\mM/\delta),\qquad
\gT_{\rm G}(\mM;\bm\Omega)
=\mQ p_q(\mQ^\top\mM/\delta),
\]
where $\bm\Omega\sim\mathcal N(0,1)^{n\times\ell}$ for any integer $\ell\ge1$
and $\mQ$ is an orthonormal basis of the actual range of
$(\mM\mM^\top)^h\mM\bm\Omega$. Define both maps to be zero at $\mM=\bm0$.
With fresh independent Gaussian sketches at each iteration, these maps satisfy
Assumption~\ref{assume:general_inexact} with the following uniform choices:
\[
\begin{array}{c|cc}
 & \Bar\gamma & \Bar\nu \\
\hline
\gT_{\rm full} & 1-d_0^{-1/2} & 0 \\
\gT_{\rm G} & 1-d_0^{-3/2} & 0 .
\end{array}
\]
These choices hold for every input and every $h,q,\ell$ as above, including
when the input rank changes or is smaller than $\ell$.
\end{lemma}

\begin{proof}
Both scalar iterations preserve $[0,1]$ and satisfy $\varphi(x)\ge x$ there
(Lemmas~\ref{lemma:NSpoly_prop} and~\ref{lemma:quintic_NSpoly_prop}). Hence
$x\le p_q(x)\le1$ for $x\in[0,1]$, including $q=0$. Every singular value of
$\mM/\delta$ and of $\mQ^\top\mM/\delta$ lies in this interval. It follows
that both output operator norms are at most one. For the full-space map,
writing $\sigma_1,\ldots,\sigma_r$ for the positive singular values of $\mM$,
\[
\langle\mM,\gT_{\rm full}(\mM)\rangle
=\sum_{i=1}^r\sigma_i p_q(\sigma_i/\delta)
\ge\frac{\|\mM\|_F^2}{\delta}
=\|\mM\|_F
\ge d_0^{-1/2}\|\mM\|_*.
\]

For the Gaussian map, set $\mB=\mQ^\top\mM$. Its number of rows is the
dimension of the sketch range, which is $\min\{r,\ell\}$ almost surely;
no full-rank assumption on the sketch is needed. Orthogonality of $\mQ$
and the same singular-value calculation give
\begin{equation}\label{eq:uniform_polar_alignment}
\langle\mM,\gT_{\rm G}(\mM;\bm\Omega)\rangle
=\langle\mB,p_q(\mB/\delta)\rangle
\ge \frac{\|\mQ^\top\mM\|_F^2}{\delta}.
\end{equation}
We next show that a single Gaussian column already guarantees
\begin{equation}\label{eq:uniform_gaussian_capture}
\mathbb E_{\bm\Omega}\|\mQ^\top\mM\|_F^2
\ge \frac{\|\mM\|_F^2}{r}.
\end{equation}
Take a thin SVD $\mM=\mU\bm\Sigma\mV^\top$ and the first column
$\bm\omega$ of $\bm\Omega$. Then $\bm g=\mV^\top\bm\omega\sim\mathcal N(0,\mI_r)$
and $\bm y=(\mM\mM^\top)^h\mM\bm\omega$ belongs to the range of $\mQ$.
Put $a_i=\sigma_i^2$, $w_i=g_i^2$, and $t=2h+1$. Almost surely $\bm y\ne0$,
and projection onto its span gives
\[
\|\mQ^\top\mM\|_F^2
\ge\frac{\|\mM^\top\bm y\|_2^2}{\|\bm y\|_2^2}
=\frac{\sum_i w_i a_i^{t+1}}{\sum_i w_i a_i^t}
\ge\frac{\sum_i w_i a_i}{\sum_i w_i}.
\]
The last inequality follows by multiplying the positive denominators and using
\begin{align*}
&\left(\sum_iw_i a_i^{t+1}\right)\left(\sum_iw_i\right)
-\left(\sum_iw_i a_i^t\right)\left(\sum_iw_i a_i\right)
\\
&\qquad=\frac12\sum_{i,j}w_iw_j(a_i^t-a_j^t)(a_i-a_j)\ge0.
\end{align*}
Since the $g_i$ are identically distributed,
$\mathbb E[w_i/\sum_jw_j]=1/r$. Taking expectations proves
\eqref{eq:uniform_gaussian_capture}. Combining it with
\eqref{eq:uniform_polar_alignment} and $\delta=\|\mM\|_F$ yields
\[
\mathbb E_{\bm\Omega}\langle\mM,\gT_{\rm G}(\mM;\bm\Omega)\rangle
\ge\frac{\|\mM\|_F}{r}
\ge\frac{\|\mM\|_*}{r^{3/2}}
\ge d_0^{-3/2}\|\mM\|_*.
\]
At $\mM=\bm0$, both alignment bounds and the operator-norm bounds hold by
definition. At each optimizer iteration, conditioning on $\mathcal G_k$
fixes $\mM_k$ and leaves a fresh standard Gaussian sketch. Thus the bounds
hold conditionally on $\mathcal G_k$, and hence on $\mM_k$, with the stated
deterministic constants. The full-space map is deterministic given $\mM_k$.
\end{proof}

These conservative dimension bounds complement the spectrum-dependent bound
in Proposition~\ref{prop:randomized_polar}, which can provide better alignment
constants for particular inputs. The uniform bounds above require Frobenius
normalization and the two theoretical polynomials; they do not establish the
same constants for the empirical quintic, PolarExpress, or non-Gaussian
sketches. More generally, $\|\mM\|_{\rm op}\le\delta\le C\|\mM\|_F$ with a
fixed $C\ge1$ divides the displayed alignment constants by $C$. An arbitrary
larger scaling, or an additive normalization constant, need not satisfy this
uniform condition.

\subsection{Discussion on Proposition \ref{prop:randomized_polar}}
\label{appendix:discuss_prop}

Define
\[
A_{s,h}
\eqdef
\sum_{j=1}^s \sigma_j^2
-
\frac{s}{p-1}\varrho_s^{4h}\sum_{j>s}\sigma_j^2,
\qquad
\theta_{s,h}
\eqdef
\frac{[A_{s,h}]_+}{\delta\|\mM\|_*}.
\]
Then Proposition \ref{prop:randomized_polar} gives
\[
\mathbb{E}_{\bm\Omega}
\left[
\left\langle \mM,\gT_{h,q}(\mM;\bm\Omega)\right\rangle
\right]
\ge
\theta_{s,h}\|\mM\|_*,
\qquad
\left\|\gT_{h,q}(\mM;\bm\Omega)\right\|_{\mathrm{op}}\le 1.
\]
For this fixed input, whenever $\theta_{s,h}>0$, the alignment bound has the
form required in Assumption~\ref{assume:general_inexact}, with
\[
\gamma=1-\theta_{s,h}
=
1-\frac{A_{s,h}}{\delta\|\mM\|_*},
\]
where the second equality uses $A_{s,h}>0$. The second condition holds with
$\nu=0$. It remains to verify that $\theta_{s,h}\le 1$ and to discuss when
$\theta_{s,h}>0$. Positivity at each input, or over one finite trajectory,
does not itself give a deterministic alignment constant bounded away from
zero across runs and horizons. Theorem~\ref{theorem:inexact} uses such uniform
constants. Under Frobenius normalization,
Lemma~\ref{lem:uniform_polar_bounds} supplies them independently of whether
this particular spectral lower bound is positive.

\paragraph{Part I.}
Using $\delta \geq \|\mM\|_{\mathrm{op}}=\sigma_1$ from Algorithm
\ref{alg:rand_ns_update}, we have
\[
\delta\|\mM\|_*
=
\delta\sum_{j=1}^r\sigma_j
\geq
\sigma_1\sum_{j=1}^r\sigma_j
\geq
\sum_{j=1}^r\sigma_j^2
\geq
\sum_{j=1}^s\sigma_j^2
\geq
[A_{s,h}]_+.
\]
This proves $\theta_{s,h}\le 1$.

\paragraph{Part II.}
We now characterize when the quantity $A_{s,h}$ is strictly positive. For a
fixed target rank $s$, define the head and tail sums of squared singular values as
\[
H_s \eqdef \sum_{j=1}^s\sigma_j^2,
\qquad
T_s \eqdef \sum_{j>s}\sigma_j^2.
\]
Then
\[
[A_{s,h}]_+
=
\left[
H_s-\frac{s}{p-1}\varrho_s^{4h}T_s
\right]_+.
\]
Since $s<r$, we have $T_s>0$. Hence $A_{s,h}>0$ if and only if
\[
H_s>\frac{s}{p-1}\varrho_s^{4h}T_s,
\]
or equivalently,
\begin{equation}\label{eq:strict_positive_condition_equiv}
\varrho_s^{4h}
<
\frac{(p-1)H_s}{sT_s}.
\end{equation}
This condition shows two mechanisms that make the bound positive. The
oversampling parameter $p$ reduces the penalty through the factor $(p-1)^{-1}$,
while power iteration reduces it through the factor $\varrho_s^{4h}$, which is
small whenever there is a spectral gap at index $s$, namely
\[
\varrho_s=\frac{\sigma_{s+1}}{\sigma_s}<1.
\]

Since $h$ must be an integer, one may take
\begin{equation}
\label{eq:h_integer_choice}
h
=
\begin{cases}
0,
&
\text{if } H_s>\dfrac{s}{p-1}T_s,
\\[1.2em]
1+\left\lfloor
\dfrac{
\log\left( \dfrac{sT_s}{(p-1)H_s} \right)
}{
4\log(1/\varrho_s)
}
\right\rfloor,
&
\text{if } H_s\leq \dfrac{s}{p-1}T_s
\text{ and } \varrho_s<1.
\end{cases}
\end{equation}
With this choice, $A_{s,h}=H_s-\frac{s}{p-1}\varrho_s^{4h}T_s$ is strictly
positive. Finally, if $\varrho_s=1$, then power iteration does not improve the
positivity condition at that value of $s$, since $\varrho_s^{4h}=1$ for all
$h$. In this case, one must either choose a different target rank $s$ at which a
spectral gap is present, increase the oversampling parameter $p$, or rely on the
no-power condition $H_s>\frac{s}{p-1}T_s$.

\Needspace{16\baselineskip}
\section{Auxiliary Lemmas}

\subsection{Lemmas for Analysis of Exact Polar Decomposition}

Throughout this subsection, let $d_0 \eqdef \min\{m,n\}$. Let $\gP(\mM)$ denote an exact polar factor satisfying $\langle \mM,\gP(\mM)\rangle=\|\mM\|_*$ and $\|\gP(\mM)\|_{\mathrm{op}}\le 1$.

\begin{lemma}\label{lemma:exact_descent}
Define
\begin{equation*}
\mS_k  \eqdef  \widetilde{\mM}_k-\nabla f(\mX_k).
\end{equation*}
Then, under Assumption \ref{assume:lipschitz}, \algname{Muon} with exact polar decomposition satisfies
\begin{equation}\label{eq:basic_descent_sum}
\eta\sum_{k=0}^{K-1}\Exp{\|\nabla f(\mX_k)\|_F}
\le 
f(\mX_0)-f_\star
+ \frac{Ld_0}{2}K\eta^2
+ \eta \left( 1+\sqrt{d_0} \right)
\sum_{k=0}^{K-1}\Exp{\left\| \mS_k \right\|_F}.	
\end{equation}	
\end{lemma}

\begin{proof}
Since exact polar factors are invariant under multiplication by a positive scalar,
the update can be written as
$\mX_{k+1}=\mX_k-\eta \gP(\widetilde{\mM}_k)$.
Moreover, $\|\gP(\widetilde{\mM}_k)\|_F\le \sqrt{d_0}$, and hence
\[
\left\| \mX_{k+1}-\mX_k \right\|_F^2
\le \eta^2 d_0.
\]
By $L$-smoothness,	
\begin{align*}
f(\mX_{k+1}) & \le  f(\mX_k) - \eta \left\langle \nabla f(\mX_k), \gP(\widetilde{\mM}_k) \right\rangle
+ \frac{Ld_0}{2}\eta^2 \\
& = 
f(\mX_k)
- \eta \left\langle \widetilde{\mM}_k, \gP(\widetilde{\mM}_k) \right\rangle
+ \eta \left\langle \mS_k, \gP(\widetilde{\mM}_k) \right\rangle
+ \frac{Ld_0}{2}\eta^2 \\
& = 
f(\mX_k)
- \eta \left\| \widetilde{\mM}_k \right\|_*
+ \eta \left\langle \mS_k, \gP(\widetilde{\mM}_k) \right\rangle
+ \frac{Ld_0}{2}\eta^2 \\
& \leq 
f(\mX_k)
- \eta \left\| \widetilde{\mM}_k \right\|_F
+ \eta \sqrt{d_0}\left\| \mS_k \right\|_F
+ \frac{Ld_0}{2}\eta^2.		
\end{align*}
Since $\| \widetilde{\mM}_k \|_F =\left\| \nabla f(\mX_k)+\mS_k \right\|_F
\ge \|\nabla f(\mX_k)\|_F-\left\| \mS_k \right\|_F$, we obtain
\[
f(\mX_{k+1})
\leq
f(\mX_k)
- \eta \left\|\nabla f(\mX_k) \right\|_F
+ \eta(1+\sqrt{d_0})\left\| \mS_k \right\|_F
+ \frac{Ld_0}{2}\eta^2.
\]
Taking total expectation and summing from $k=0$ to $K-1$ yields the claim.	
\end{proof}

\begin{lemma}\label{lemma:eps_unroll}
Define
\begin{equation*}
\mD_k \eqdef \nabla f(\mX_{k-1})-\nabla f(\mX_k) \quad\quad \text{and}\quad\quad \bm\xi_k \eqdef \mG_k-\nabla f(\mX_k).
\end{equation*}
Then, for every $k\ge 1$,
\begin{equation}\label{eq:eps_unroll}
\mS_k  = \beta^k \mS_0
+ \beta^2\sum_{s=1}^k \beta^{k-s}\mD_s
+ (1-\beta)\Big((1+\beta)\bm\xi_k
+ \sum_{s=1}^{k-1}\beta^{k+1-s}\bm\xi_s
- \beta^k\bm\xi_0\Big).
\end{equation}	
\end{lemma}

\begin{proof}
Note that
\begin{align*}
\mS_k & = \widetilde{\mM}_k-\nabla f(\mX_k) \\
& =  \beta\widetilde{\mM}_{k-1}
+ (1-\beta)\big((1+\beta)\mG_k-\beta\mG_{k-1}\big)
- \nabla f(\mX_k) \\
& =  \beta\left(\mS_{k-1}+\nabla f(\mX_{k-1})\right)
+ (1-\beta)\big((1+\beta)\mG_k-\beta\mG_{k-1}\big)
- \nabla f(\mX_k).	
\end{align*}
Using $\mG_k=\nabla f(\mX_k)+\bm\xi_k$ gives
\begin{equation*}
\mS_k = 
\beta \mS_{k-1}
+ (1-\beta)\big((1+\beta)\bm\xi_k-\beta\bm\xi_{k-1}\big)
+ \beta^2\left(\nabla f(\mX_{k-1})-\nabla f(\mX_k)\right).	
\end{equation*}
Thus,
\[
\mS_k
=
\beta \mS_{k-1}
+ \beta^2 \mD_k
+ (1-\beta)\big((1+\beta)\bm\xi_k-\beta\bm\xi_{k-1}\big).
\]
Unrolling this recursion completes the proof.    
\end{proof}

\Needspace{15\baselineskip}
\begin{lemma}\label{lemma:bound_S_k}
Under Assumptions \ref{assume:lipschitz} and \ref{assume:general_bounded_moment}, for every $0\le k\le K-1$,
\begin{align}\label{eq:bound_S_k}
\Exp{\| \mS_k \|_F}
& \leq\beta^k \Exp{\left\| \mS_0 \right\|_F}
+ \frac{\beta^2}{1-\beta}L\eta\sqrt {d_0} \notag\\
&\quad + \frac{2 \sqrt{\pi} \gC_\alpha^{\nicefrac{1}{\alpha}}}{B^{\frac{\alpha-1}{\alpha}}} \Bigg[
\Gamma_\alpha \sigma_0 (1-\beta)^{\frac{\alpha-1}{\alpha}}
+ \sigma_1 (1-\beta)\bigg((1+\beta)\Exp{\|\nabla f(\mX_k)\|_F} \notag\\
&\quad +\sum_{s=1}^{k-1}\beta^{k+1-s}\Exp{\|\nabla f(\mX_s)\|_F}
+ \beta^k\|\nabla f(\mX_0)\|_F\bigg)
\Bigg],
\end{align}
where $\Gamma_\alpha \eqdef (2^\alpha+1)^{\nicefrac{1}{\alpha}}$.	
\end{lemma}

\begin{proof}
The case $k=0$ is immediate. Let $k\ge 1$. By $L$-smoothness,
\begin{equation*}
\| \mD_s\|_F
= 
\left\| \nabla f(\mX_{s-1})-\nabla f(\mX_s) \right\|_F  \le
L \left\| \mX_s-\mX_{s-1} \right\|_F \\
\le
L\eta\sqrt{d_0}.
\end{equation*}
Therefore,
\begin{equation*}
\beta^2\sum_{s=1}^k \beta^{k-s}\Exp{\left\| \mD_s \right\|_F}
\leq
\frac{\beta^2}{1-\beta}L\eta\sqrt{d_0}.
\end{equation*}
For this fixed $k$, define the deterministic coefficients
\begin{equation}\label{eq:nesterov_noise_weights}
\begin{aligned}
a_0 & :=  -(1-\beta)\beta^k, \\
a_s & :=  (1-\beta)\beta^{k+1-s} \quad \text{for } 1\le s\le k-1, \\
a_k & :=  (1-\beta)(1+\beta).
\end{aligned}
\end{equation}
For each finite $s$, the exact-polar step bound gives
$\|\mX_s-\mX_0\|_F\leq s\eta\sqrt{d_0}$; smoothness therefore ensures
$\Exp{\|\nabla f(\mX_s)\|_F^\alpha}<\infty$.
Applying Lemma~\ref{lemma:weighted_conditional_martingale} with these coefficients gives
\begin{equation}\label{eq:weighted_noise_bound}
\Exp{\Bigg\|\sum_{s=0}^k a_s\bm\xi_s\Bigg\|_F}
\leq
\frac{2\sqrt{\pi}\gC_\alpha^{\nicefrac{1}{\alpha}}}{B^{\frac{\alpha-1}{\alpha}}}
\left[
\sigma_0\left(\sum_{s=0}^k |a_s|^\alpha\right)^{1/\alpha}
+ \sigma_1\sum_{s=0}^k |a_s|\Exp{\|\nabla f(\mX_s)\|_F}
\right].
\end{equation}
Moreover, reindexing the middle sum and using $\beta^{\alpha k}\leq\beta$
and $\beta^{\alpha j}\leq\beta^j$ gives
\begin{align*}
\left(\sum_{s=0}^k |a_s|^\alpha\right)^{1/\alpha}
& =
(1-\beta)
\left(
\beta^{\alpha k}
+ \sum_{s=1}^{k-1}\beta^{\alpha(k+1-s)}
+ (1+\beta)^\alpha
\right)^{1/\alpha} \\
& \leq
(1-\beta)
\left(
\sum_{j=1}^{k}\beta^j+2^\alpha
\right)^{1/\alpha}  \leq
(1-\beta)
\left(
\frac{1}{1-\beta}+2^\alpha
\right)^{1/\alpha} \\
& = (1-\beta)^{\frac{\alpha-1}{\alpha}}
\left(1+(1-\beta)2^\alpha\right)^{1/\alpha}
\leq
\Gamma_\alpha(1-\beta)^{\frac{\alpha-1}{\alpha}}.
\end{align*}
Here $\Gamma_\alpha=(1+2^\alpha)^{1/\alpha}$, consistently with the statement.
Thus,
\begin{align*}
\Exp{\Bigg\|\sum_{s=0}^k a_s\bm\xi_s\Bigg\|_F}
& \leq 
\frac{2 \sqrt{\pi} \gC_\alpha^{\nicefrac{1}{\alpha}}}{B^{\frac{\alpha-1}{\alpha}}}
\Bigg[
\Gamma_\alpha \sigma_0 (1-\beta)^{\frac{\alpha-1}{\alpha}}
+ \sigma_1 (1-\beta)\bigg((1+\beta)\Exp{\|\nabla f(\mX_k)\|_F} \\
&\quad +
\sum_{s=1}^{k-1}\beta^{k+1-s}\Exp{\|\nabla f(\mX_s)\|_F}
+ \beta^k\|\nabla f(\mX_0)\|_F\bigg)
\Bigg].
\end{align*}
Combining this bound with Lemma \ref{lemma:eps_unroll} and the drift bound above yields \eqref{eq:bound_S_k}.        
\end{proof}

\begin{lemma}\label{lemma:bound_sum_S_k}
Under the conditions of Lemma \ref{lemma:bound_S_k}, we have
\begin{align}\label{eq:bound_sum_S_k}
\sum_{k=0}^{K-1}\Exp{\left\| \mS_k \right\|_F}
& \leq
\frac{\Exp{\left\| \mS_0 \right\|_F}}{1-\beta}
+ K\frac{\beta^2}{1-\beta}L\eta\sqrt{d_0}
+ K\frac{2\sqrt{\pi}\gC_\alpha^{\nicefrac{1}{\alpha}}}{B^{\frac{\alpha-1}{\alpha}}}
\Gamma_\alpha\sigma_0(1-\beta)^{\frac{\alpha-1}{\alpha}} \notag\\
&\quad +
\frac{2\sqrt{\pi}\gC_\alpha^{\nicefrac{1}{\alpha}}}{B^{\frac{\alpha-1}{\alpha}}}
\sigma_1
\left(
\|\nabla f(\mX_0)\|_F
+ \sum_{k=0}^{K-1}\Exp{\left\| \nabla f(\mX_k) \right\|_F}
\right).	
\end{align}   
\end{lemma}

\begin{proof}
We now sum each term of \eqref{eq:bound_S_k} for $k=0,\ldots,K-1$. Note that
\begin{align*}
\sum_{k=0}^{K-1}\beta^k \Exp{\left\| \mS_0 \right\|_F} & \leq \frac{\Exp{\left\| \mS_0 \right\|_F}}{1-\beta}, \\
\sum_{k=0}^{K-1}\frac{\beta^2}{1-\beta}L\eta\sqrt{d_0} & = K\frac{\beta^2}{1-\beta}L\eta\sqrt{d_0}, \\
\sum_{k=0}^{K-1} \frac{2\sqrt{\pi}\gC_\alpha^{\nicefrac{1}{\alpha}}}{B^{\frac{\alpha-1}{\alpha}}}
\Gamma_\alpha\sigma_0(1-\beta)^{\frac{\alpha-1}{\alpha}}
& = 
K\frac{2\sqrt{\pi}\gC_\alpha^{\nicefrac{1}{\alpha}}}{B^{\frac{\alpha-1}{\alpha}}}
\Gamma_\alpha\sigma_0(1-\beta)^{\frac{\alpha-1}{\alpha}}.	
\end{align*}
Similarly, we have
\begin{align*}
&(1-\beta)\sum_{k=0}^{K-1}
\left(
(1+\beta)\Exp{\left\|\nabla f(\mX_k)\right\|_F}
+ \sum_{s=1}^{k-1}\beta^{k+1-s}\Exp{\left\|\nabla f(\mX_s)\right\|_F}
\right) \\
&\le
\sum_{k=0}^{K-1}(1-\beta^2)\Exp{\left\|\nabla f(\mX_k)\right\|_F}
+ \beta^2\sum_{k=0}^{K-1}\Exp{\left\|\nabla f(\mX_k)\right\|_F} \\
&=
\sum_{k=0}^{K-1}\Exp{\left\|\nabla f(\mX_k)\right\|_F},	
\end{align*}
and
\[
\sum_{k=0}^{K-1}(1-\beta)\beta^k\|\nabla f(\mX_0)\|_F
\le
\|\nabla f(\mX_0)\|_F.
\]
Combining the above inequalities and using \eqref{eq:bound_S_k} gives \eqref{eq:bound_sum_S_k}.    
\end{proof}

\begin{theorem}\label{theorem:exact}
Suppose Assumptions \ref{assume:lipschitz} and \ref{assume:general_bounded_moment} hold. Then, for any $\eta>0$ and $\beta\in(0,1)$, \algname{Muon}(\ref{eq:muon_inexact}) with exact polar decomposition and
batch size
\[
B >
\left(
2\sqrt{\pi}(1+\sqrt{d_0})\gC_\alpha^{\nicefrac{1}{\alpha}}\sigma_1
\right)^{\frac{\alpha}{\alpha-1}}
\]
satisfies
\begin{align*}
\min_{0 \leq k \leq K-1}\Exp{\|\nabla f(\mX_k)\|_F}
& \leq 
\frac{f(\mX_0)-f_\star}{\eta \rho K}
+ \frac{Ld_0 \eta}{2 \rho}
+ (1+\sqrt{d_0})
\frac{\Exp{\left\| \mS_0 \right\|_F}}{(1-\beta)K\rho} \\
&\quad +
\frac{\beta^2}{(1-\beta)\rho}L\eta\sqrt{d_0}(1+\sqrt{d_0}) \\
&\quad +
(1+\sqrt{d_0})
\frac{2\sqrt{\pi}\gC_\alpha^{\nicefrac{1}{\alpha}}}{\rho B^{\frac{\alpha-1}{\alpha}}}
\Gamma_\alpha\sigma_0(1-\beta)^{\frac{\alpha-1}{\alpha}} \\
&\quad +
(1+\sqrt{d_0})
\frac{2\sqrt{\pi}\gC_\alpha^{\nicefrac{1}{\alpha}}\sigma_1\|\nabla f(\mX_0)\|_F}
{\rho B^{\frac{\alpha-1}{\alpha}}K},	
\end{align*}
where
\[
\rho
\eqdef
1-(1+\sqrt{d_0})
\frac{2\sqrt{\pi}\gC_\alpha^{\nicefrac{1}{\alpha}}\sigma_1}
{B^{\frac{\alpha-1}{\alpha}}}
>0.
\]    
\end{theorem}

\begin{proof}
From \eqref{eq:basic_descent_sum} and \eqref{eq:bound_sum_S_k},
\begin{align*}
\eta\sum_{k=0}^{K-1}\Exp{\|\nabla f(\mX_k)\|_F}
& \leq 
f(\mX_0)-f_\star
+ \frac{Ld_0}{2}K\eta^2
+ \eta(1+\sqrt{d_0})
\sum_{k=0}^{K-1}\Exp{\left\|\mS_k\right\|_F} \\
& \leq 
f(\mX_0)-f_\star
+ \frac{Ld_0}{2}K\eta^2
+ \eta(1+\sqrt{d_0})
\frac{\Exp{\left\|\mS_0\right\|_F}}{1-\beta} \\
& +
K\frac{\beta^2}{1-\beta}L\eta^2\sqrt{d_0}(1+\sqrt{d_0})  +
\eta(1+\sqrt{d_0})K
\frac{2\sqrt{\pi}\gC_\alpha^{\nicefrac{1}{\alpha}}}{B^{\frac{\alpha-1}{\alpha}}}
\Gamma_\alpha\sigma_0(1-\beta)^{\frac{\alpha-1}{\alpha}} \\
& +
\eta(1+\sqrt{d_0})
\frac{2\sqrt{\pi}\gC_\alpha^{\nicefrac{1}{\alpha}}}{B^{\frac{\alpha-1}{\alpha}}}
\sigma_1
\left(
\|\nabla f(\mX_0)\|_F
+ \sum_{k=0}^{K-1}\Exp{\left\|\nabla f(\mX_k)\right\|_F}
\right).
\end{align*}
Rearranging gives
\begin{align*}
\eta\rho\sum_{k=0}^{K-1}\Exp{\|\nabla f(\mX_k)\|_F}
& \leq
f(\mX_0)-f_\star
+ \frac{Ld_0}{2}K\eta^2
+ \eta(1+\sqrt{d_0})
\frac{\Exp{\left\|\mS_0\right\|_F}}{1-\beta} \\
&+K\frac{\beta^2}{1-\beta}L\eta^2\sqrt{d_0}(1+\sqrt{d_0})  +
\eta(1+\sqrt{d_0})K
\frac{2\sqrt{\pi}\gC_\alpha^{\nicefrac{1}{\alpha}}}{B^{\frac{\alpha-1}{\alpha}}}
\Gamma_\alpha\sigma_0(1-\beta)^{\frac{\alpha-1}{\alpha}} \\
&+\eta(1+\sqrt{d_0})
\frac{2\sqrt{\pi}\gC_\alpha^{\nicefrac{1}{\alpha}}\sigma_1\|\nabla f(\mX_0)\|_F}
{B^{\frac{\alpha-1}{\alpha}}}.
\end{align*}
Dividing by $\eta\rho K$ and using
\[
\sum_{k=0}^{K-1}\Exp{\|\nabla f(\mX_k)\|_F}
\ge
K\min_{0\le k\le K-1}\Exp{\|\nabla f(\mX_k)\|_F}
\]
completes the proof.        
\end{proof}

\subsection{Lemmas for Analysis of Inexact Polar Decomposition}

Throughout this subsection, let $d_0=\min\{m,n\}$ and let
$\Bar\gamma\in[0,1)$ and $\Bar\nu\ge0$ be deterministic upper bounds satisfying
\begin{equation*}
\gamma_k\le\Bar\gamma \qquad\text{and}\qquad
\nu_k\le\Bar\nu \qquad (0\le k<K),
\end{equation*}
and set $\mS_k \eqdef \widetilde{\mM}_k-\nabla f(\mX_k)$.
The lemmas below only require these bounds over the current horizon.
Theorem~\ref{theorem:inexact} applies them with constants that hold uniformly
over all iterations and horizons.

\begin{lemma}
\label{lem:basic_descent_inexact}
Under Assumptions \ref{assume:lipschitz} and \ref{assume:general_inexact},
\begin{align}\label{eq:basic_descent_inexact}
\eta \left( 1 - \Bar{\gamma} \right) \sum_{k = 0}^{K-1} \Exp{\left\| \nabla f(\mX_k) \right\|_F}
& \leq 
f(\mX_0) - f_\star
+ \eta \left( 1 + \sqrt{d_0} \left( 1 + \Bar{\nu} \right) \right)
\sum_{k = 0}^{K-1} \Exp{\left\| \mS_k \right\|_F} \notag \\
&\quad
+ \frac{Ld_0}{2} \eta^2 \left( 1 + \Bar{\nu} \right)^2 K.    
\end{align}
\end{lemma}

\begin{proof}
Let $\mathcal G_k=\sigma(\mathcal F_k,\mG_k^1,\ldots,\mG_k^B)$ be the
history after the current mini-batch and before the polar draw, as defined
in Section~\ref{section:motivation}.
Then $\mX_k$, $\mM_k$, $\widetilde{\mM}_k$, and $\mS_k$ are
$\mathcal G_k$-measurable. By the fresh-seed oracle model preceding Assumption~\ref{assume:general_inexact},
for every integrable function $\psi$ of the input and output,
\[
\Exp{\psi(\mM_k,\gT(\mM_k))\mid\mathcal G_k}
=\Exp{\psi(\mM_k,\gT(\mM_k))\mid\mM_k}.
\]
Consequently, both the alignment and operator-norm bounds in
Assumption~\ref{assume:general_inexact} hold conditionally on $\mathcal G_k$.

By $L$-smoothness and the update rule,
\begin{align*}
\Exp{f(\mX_{k+1}) \mid \mathcal G_k}
&\le
f(\mX_k)
-\eta \Exp{\left\langle \nabla f(\mX_k), \gT(\mM_k) \right\rangle \mid \mathcal G_k}
+\frac{L\eta^2}{2}\Exp{\left\| \gT(\mM_k) \right\|_F^2 \mid \mathcal G_k} \\
&=
f(\mX_k)
-\eta \Exp{\left\langle \widetilde{\mM}_k, \gT(\mM_k) \right\rangle \mid \mathcal G_k}
+\eta \Exp{\left\langle \mS_k, \gT(\mM_k) \right\rangle \mid \mathcal G_k} \\
&\hspace{2em}
+\frac{L\eta^2}{2}\Exp{\left\| \gT(\mM_k) \right\|_F^2 \mid \mathcal G_k}.
\end{align*}
For the alignment term, using $\widetilde{\mM}_k=(1-\beta)\mM_k$ gives
\begin{align*}
\Exp{\left\langle \widetilde{\mM}_k, \gT(\mM_k) \right\rangle \mid \mathcal G_k}
&=
(1-\beta)\Exp{\left\langle \mM_k, \gT(\mM_k) \right\rangle \mid \mathcal G_k} \\
&\ge
(1-\beta)(1-\gamma_k)\|\mM_k\|_* =
(1-\gamma_k)\|\widetilde{\mM}_k\|_*.
\end{align*}
For the error term, Cauchy-Schwarz, Jensen's inequality, and Assumption \ref{assume:general_inexact} imply
\begin{align*}
\Exp{\left\langle \mS_k, \gT(\mM_k) \right\rangle \mid \mathcal G_k}
&\le
\|\mS_k\|_F
\Exp{\left\| \gT(\mM_k) \right\|_F \mid \mathcal G_k} \\
&\le
\|\mS_k\|_F
\Exp{\left\| \gT(\mM_k) \right\|_F^2 \mid \mathcal G_k}^{1/2} \le
\sqrt{d_0}(1+\nu_k)\|\mS_k\|_F.
\end{align*}
Similarly,
\[
\Exp{\left\| \gT(\mM_k) \right\|_F^2 \mid \mathcal G_k}
\le
d_0\Exp{\left\| \gT(\mM_k) \right\|_{\rm op}^2 \mid \mathcal G_k}
\le
d_0(1+\nu_k)^2.
\]
Combining these bounds,
\begin{align*}
\Exp{f(\mX_{k+1}) \mid \mathcal G_k}
&\le
f(\mX_k)
-\eta(1-\gamma_k)\|\widetilde{\mM}_k\|_*
+\eta\sqrt{d_0}(1+\nu_k)\|\mS_k\|_F
+\frac{Ld_0}{2}\eta^2(1+\nu_k)^2 \\
&\le
f(\mX_k)
-\eta(1-\gamma_k)\|\widetilde{\mM}_k\|_F
+\eta\sqrt{d_0}(1+\nu_k)\|\mS_k\|_F
+\frac{Ld_0}{2}\eta^2(1+\nu_k)^2 \\
&\le
f(\mX_k)
-\eta(1-\gamma_k)\|\nabla f(\mX_k)\|_F
+\eta\left((1-\gamma_k)+\sqrt{d_0}(1+\nu_k)\right)\|\mS_k\|_F \\
&\hspace{2em}
+\frac{Ld_0}{2}\eta^2(1+\nu_k)^2 \\
&\le
f(\mX_k)
-\eta(1-\Bar{\gamma})\|\nabla f(\mX_k)\|_F
+\eta\left(1+\sqrt{d_0}(1+\Bar{\nu})\right)\|\mS_k\|_F \\
&\hspace{2em}
+\frac{Ld_0}{2}\eta^2(1+\Bar{\nu})^2.
\end{align*}
Taking total expectation, summing over $k=0,\ldots,K-1$, and using $f(\mX_K)\ge f_\star$ gives \eqref{eq:basic_descent_inexact}.
\end{proof}

\begin{lemma}\label{lem:inexact_sum_S_k}
Under Assumptions \ref{assume:lipschitz}, \ref{assume:general_inexact}, and \ref{assume:general_bounded_moment},
\begin{align}\label{eq:inexact_sum_S_k}
\sum_{k = 0}^{K-1} \Exp{\left\| \mS_k \right\|_F}
& \leq
\frac{\Exp{\left\| \mS_0 \right\|_F}}{1 - \beta}
+ K \frac{\beta^2}{1-\beta}L\eta\sqrt{d_0}(1 + \Bar{\nu})
+ K \frac{2 \sqrt{\pi} \gC_\alpha^{1/\alpha}}{B^{(\alpha-1)/\alpha}}
\Gamma_\alpha \sigma_0(1-\beta)^{(\alpha-1)/\alpha} \notag \\
&\quad
+ \frac{2 \sqrt{\pi} \gC_\alpha^{1/\alpha}}{B^{(\alpha-1)/\alpha}}
\sigma_1
\left(
\|\nabla f(\mX_0)\|_F
+ \sum_{k = 0}^{K-1} \Exp{\left\| \nabla f(\mX_k) \right\|_F}
\right),
\end{align}
where $\Gamma_\alpha \eqdef (2^\alpha+1)^{1/\alpha}$.
\end{lemma}

\begin{proof}
Let
\[
\mD_s \eqdef \nabla f(\mX_{s-1})-\nabla f(\mX_s),
\qquad
\bm\xi_s \eqdef \mG_s-\nabla f(\mX_s).
\]
By Lemma \ref{lemma:eps_unroll}, for every $k\ge 1$,
\[
\mS_k
=
\beta^k \mS_0
+\beta^2\sum_{s=1}^k \beta^{k-s}\mD_s
+(1-\beta)\left((1+\beta)\bm\xi_k
+\sum_{s=1}^{k-1}\beta^{k+1-s}\bm\xi_s
-\beta^k\bm\xi_0\right).
\]
For $s\ge 1$, $L$-smoothness and the update $\mX_s=\mX_{s-1}-\eta\gT(\mM_{s-1})$ give
\begin{equation*}
\Exp{\left\| \mD_s \right\|_F} \le
L\Exp{\left\| \mX_s-\mX_{s-1} \right\|_F} =
L\eta\Exp{\left\| \gT(\mM_{s-1}) \right\|_F} \le
L\eta\sqrt{d_0}(1+\Bar{\nu}),
\end{equation*}
where the last step uses Jensen's inequality and Assumption \ref{assume:general_inexact}. Hence
\[
\beta^2\sum_{s=1}^k \beta^{k-s}\Exp{\left\| \mD_s \right\|_F}
\le
\frac{\beta^2}{1-\beta}L\eta\sqrt{d_0}(1+\Bar{\nu}).
\]
For the noise term, apply Lemma~\ref{lemma:weighted_conditional_martingale}
with the deterministic coefficients in \eqref{eq:nesterov_noise_weights} and
the filtration $\mathcal F_s$ including all previous polar randomness.
To verify its gradient-moment condition, note that, for each finite $s$,
\[
\Exp{\|\mX_s-\mX_0\|_F^2}
\leq s\sum_{j=0}^{s-1}\Exp{\|\mX_{j+1}-\mX_j\|_F^2}
\leq s^2\eta^2d_0(1+\Bar{\nu})^2,
\]
so smoothness gives $\Exp{\|\nabla f(\mX_s)\|_F^\alpha}<\infty$ for $\alpha\leq2$.
Thus \eqref{eq:weighted_noise_bound} also holds here. Combining it with the
coefficient bound from Lemma~\ref{lemma:bound_S_k} and the drift bound above yields
\begin{align*}
\Exp{\left\| \mS_k \right\|_F}
&\le
\beta^k \Exp{\left\| \mS_0 \right\|_F}
+\frac{\beta^2}{1-\beta}L\eta\sqrt{d_0}(1+\Bar{\nu}) 
+\frac{2 \sqrt{\pi} \gC_\alpha^{1/\alpha}}{B^{(\alpha-1)/\alpha}}
\Bigg[
\Gamma_\alpha \sigma_0(1-\beta)^{(\alpha-1)/\alpha} \\
&\hspace{-2.5em}
+\sigma_1(1-\beta)
\left(
(1+\beta)\Exp{\left\|\nabla f(\mX_k)\right\|_F}
+\sum_{s=1}^{k-1}\beta^{k+1-s}\Exp{\left\|\nabla f(\mX_s)\right\|_F}
+\beta^k\|\nabla f(\mX_0)\|_F
\right)
\Bigg].
\end{align*}
Summing this inequality over $k=0,\ldots,K-1$ and using the same geometric-series bounds as in Lemma \ref{lemma:bound_sum_S_k} proves \eqref{eq:inexact_sum_S_k}.
\end{proof}

\section{Missing Proofs of Section \ref{section:main_results}}

\subsection{Proof of Lemma \ref{lemma:scaled_recurrence}}

\lemmascaledrecurrence*

\begin{proof}
From $\widetilde{\mC}_k \eqdef (1-\beta) \mC_k$, we obtain
\begin{equation*}
\widetilde{\mC}_k = (1-\beta)\mC_k = (1-\beta)(\beta\mC_{k-1}+\mG_k) = \beta\widetilde{\mC}_{k-1} + (1-\beta)\mG_k,
\end{equation*}
and from $\widetilde{\mM}_k \eqdef (1-\beta) \mM_k$, we get
\begin{equation*}
\widetilde{\mM}_k = (1-\beta)\mM_k = (1-\beta)(\beta\mC_k+\mG_k) = \beta\widetilde{\mC}_k + (1-\beta)\mG_k.
\end{equation*}
For $k\ge 1$, we want to express $\widetilde{\mM}_k$ in terms of $\widetilde{\mM}_{k-1}$. Note that
\begin{align*}
\widetilde{\mM}_k &= \beta\widetilde{\mC}_k + (1-\beta)\mG_k = \beta(\beta\widetilde{\mC}_{k-1} + (1-\beta)\mG_k) + (1-\beta)\mG_k\\
&= \beta^2\widetilde{\mC}_{k-1} + (1-\beta)(1+\beta)\mG_k.
\end{align*}
Also $\widetilde{\mM}_{k-1}=\beta\widetilde{\mC}_{k-1}+(1-\beta)\mG_{k-1}$, so
\[
\beta^2\widetilde{\mC}_{k-1}
=
\beta\widetilde{\mM}_{k-1}-\beta(1-\beta)\mG_{k-1}.
\]
Therefore
\begin{equation*}
\widetilde{\mM}_k = \beta\widetilde{\mM}_{k-1} + (1-\beta)\big((1+\beta)\mG_k-\beta\mG_{k-1}\big).	
\end{equation*}	
This completes the proof.
\end{proof}

\subsection{Proof of Theorem \ref{corollary:exact_with_alpha}}

\corollaryexactwithalpha*

\begin{proof}
Set $\hat{\beta} \eqdef 1 - \beta$. Then $\hat{\beta} = K^{-\nicefrac{\alpha}{3\alpha-2}}$, and Theorem \ref{theorem:exact}, with $\rho=\rho_{\rm ex}$, gives
\begin{align*}
\min_{0 \leq k \leq K-1} \Exp{\|\nabla f(\mX_k)\|_F}
& \leq 
\frac{f(\mX_0)-f_\star}{\eta \rho_{\rm ex} K}
+ \frac{Ld_0 \eta}{2 \rho_{\rm ex}}
+ \left( 1+\sqrt{d_0} \right)
\frac{\Exp{\left\| \mS_0 \right\|_F}}{\hat{\beta} K \rho_{\rm ex}} \\
&\quad
+ \frac{L \eta \sqrt{d_0} (1 + \sqrt{d_0})}{\hat{\beta} \rho_{\rm ex}} 
+ \left( 1+\sqrt{d_0} \right)
\frac{2 \sqrt{\pi} \gC_\alpha^{\nicefrac{1}{\alpha}}}{\rho_{\rm ex} B^{\frac{\alpha-1}{\alpha}}}
\Gamma_\alpha \sigma_0 \hat{\beta}^{\frac{\alpha-1}{\alpha}} \\
&\quad
+ \left( 1+\sqrt{d_0} \right)
\frac{2 \sqrt{\pi} \gC_\alpha^{\nicefrac{1}{\alpha}} \sigma_1 \|\nabla f(\mX_0)\|_F}
{\rho_{\rm ex} B^{\frac{\alpha-1}{\alpha}} K}.
\end{align*}
Substituting $\eta = K^{-\nicefrac{2\alpha-1}{3\alpha-2}}$ and $\hat{\beta}=K^{-\nicefrac{\alpha}{3\alpha-2}}$, and using the convention that $\gO(\cdot)$ hides constants depending only on $\alpha$ and $d_0$, yields
\begin{align*}
\min_{0 \leq k \leq K-1} \Exp{\|\nabla f(\mX_k)\|_F}
& \leq \gO \left(
\frac{f(\mX_0)-f_\star}{\rho_{\rm ex} K^{\nicefrac{\alpha - 1}{3\alpha - 2}}}
+ \frac{L}{\rho_{\rm ex} K^{\nicefrac{2\alpha - 1}{3\alpha - 2}}}
+ \frac{\Exp{\| \mS_0 \|_F}}{\rho_{\rm ex} K^{\nicefrac{2\alpha - 2}{3\alpha - 2}}} \right. \\
&\quad \left.
+ \frac{L}{\rho_{\rm ex} K^{\nicefrac{\alpha - 1}{3 \alpha - 2}}}
+ \frac{\sigma_0}{\rho_{\rm ex} B^{\frac{\alpha - 1}{\alpha}} K^{\nicefrac{\alpha - 1}{3 \alpha - 2}}}
+ \frac{\sigma_1 \| \nabla f(\mX_0)\|_F}{\rho_{\rm ex} B^{\frac{\alpha - 1}{\alpha}} K}
\right).	
\end{align*}
This completes the proof.	
\end{proof}

\subsection{Proof of Theorem \ref{theorem:inexact}}

\theoreminexact*

\begin{proof}
By Lemma \ref{lem:basic_descent_inexact},
\begin{align*}
\eta \left( 1 - \Bar{\gamma} \right)
\sum_{k = 0}^{K-1} \Exp{\left\| \nabla f(\mX_k) \right\|_F}
& \leq
f(\mX_0) - f_\star
+ \eta \left( 1 + \sqrt{d_0} \left( 1 + \Bar{\nu} \right) \right)
\sum_{k = 0}^{K-1} \Exp{\left\| \mS_k \right\|_F} \\
&\quad
+ \frac{Ld_0}{2} \eta^2 \left( 1 + \Bar{\nu} \right)^2 K .
\end{align*}
Applying Lemma \ref{lem:inexact_sum_S_k} gives
\begin{align*}
&\hskip-1cm \eta \left( 1 - \Bar{\gamma} \right)
\sum_{k = 0}^{K-1} \Exp{\left\| \nabla f(\mX_k) \right\|_F} \\
& \leq
f(\mX_0) - f_\star
+ \frac{Ld_0}{2} \eta^2 \left( 1 + \Bar{\nu} \right)^2 K 
+ \eta \left( 1 + \sqrt{d_0} \left( 1 + \Bar{\nu} \right) \right)
\frac{\Exp{\left\| \mS_0 \right\|_F}}{1 - \beta} \\
&\quad
+ K \frac{\beta^2}{1-\beta} L \eta^2 \sqrt{d_0} (1 + \Bar{\nu})
\left( 1 + \sqrt{d_0} \left( 1 + \Bar{\nu} \right) \right) \\
&\quad
+ K
\frac{2 \sqrt{\pi} \gC_\alpha^{\nicefrac{1}{\alpha}}}{B^{\nicefrac{\alpha-1}{\alpha}}}
\Gamma_\alpha \sigma_0(1-\beta)^{\nicefrac{\alpha-1}{\alpha}}
\eta \left( 1 + \sqrt{d_0} \left( 1 + \Bar{\nu} \right) \right) \\
&\quad
+
\frac{2 \sqrt{\pi} \gC_\alpha^{\nicefrac{1}{\alpha}}}{B^{\nicefrac{\alpha-1}{\alpha}}}
\sigma_1 \eta
\left( 1 + \sqrt{d_0} \left( 1 + \Bar{\nu} \right) \right)
\left(
\|\nabla f(\mX_0)\|_F
+
\sum_{k = 0}^{K-1} \Exp{\left\| \nabla f(\mX_k) \right\|_F}
\right).
\end{align*}
Rearranging the last term yields
\begin{align*}
&\hskip-1cm\eta
\left(
1 - \Bar{\gamma}
-
\left( 1+\sqrt{d_0} \left( 1 + \Bar{\nu} \right) \right)
\frac{2 \sqrt{\pi} \gC_\alpha^{\nicefrac{1}{\alpha}} \sigma_1}
{B^{\nicefrac{\alpha-1}{\alpha}}}
\right)
\sum_{k=0}^{K-1}\Exp{\|\nabla f(\mX_k)\|_F} \\
&\leq
f(\mX_0)-f_\star
+ \frac{Ld_0}{2} \eta^2 \left( 1 + \Bar{\nu} \right)^2 K
+ \eta \left( 1 + \sqrt{d_0} \left( 1 + \Bar{\nu} \right) \right)
\frac{\Exp{\left\| \mS_0 \right\|_F}}{1 - \beta} \\
&\quad
+ K \frac{\beta^2}{1-\beta} L \eta^2 \sqrt{d_0} (1 + \Bar{\nu})
\left( 1 + \sqrt{d_0} \left( 1 + \Bar{\nu} \right) \right) \\
&\quad
+ K
\frac{2 \sqrt{\pi} \gC_\alpha^{\nicefrac{1}{\alpha}}}{B^{\nicefrac{\alpha-1}{\alpha}}}
\Gamma_\alpha \sigma_0(1-\beta)^{\nicefrac{\alpha-1}{\alpha}}
\eta \left( 1 + \sqrt{d_0} \left( 1 + \Bar{\nu} \right) \right) \\
&\quad
+
\frac{2 \sqrt{\pi} \gC_\alpha^{\nicefrac{1}{\alpha}}}{B^{\nicefrac{\alpha-1}{\alpha}}}
\sigma_1 \eta
\left( 1 + \sqrt{d_0} \left( 1 + \Bar{\nu} \right) \right)
\|\nabla f(\mX_0)\|_F .
\end{align*}
By the assumed lower bound on $B$, the coefficient in parentheses is
$\rho>0$. Dividing by $\eta\rho K$ and using
\[
\sum_{k=0}^{K-1}\Exp{\|\nabla f(\mX_k)\|_F}
\ge
K \min_{0 \leq k \leq K-1}\Exp{\|\nabla f(\mX_k)\|_F}
\]
gives
\begin{align*}
\min_{0 \leq k \leq K-1} \Exp{\|\nabla f(\mX_k)\|_F}
&\le
\frac{f(\mX_0)-f_\star}{\eta \rho K}
+ \frac{Ld_0}{2 \rho} \eta \left( 1 + \Bar{\nu} \right)^2 
+ \left( 1+\sqrt{d_0} (1 + \Bar{\nu}) \right)
\frac{\Exp{\left\| \mS_0 \right\|_F}}{(1 - \beta) K \rho} \\
&\quad
+ \frac{\beta^2}{(1 - \beta) \rho}
L \eta \sqrt{d_0} (1 + \Bar{\nu})
\left( 1 + \sqrt{d_0} (1 + \Bar{\nu}) \right) \\
&\quad
+ \left( 1+\sqrt{d_0} (1 + \Bar{\nu}) \right)
\frac{2 \sqrt{\pi} \gC_\alpha^{\nicefrac{1}{\alpha}}}{\rho B^{\nicefrac{\alpha-1}{\alpha}}}
\Gamma_\alpha \sigma_0(1-\beta)^{\nicefrac{\alpha-1}{\alpha}} \\
&\quad
+ \left( 1+\sqrt{d_0} (1 + \Bar{\nu}) \right)
\frac{2 \sqrt{\pi} \gC_\alpha^{\nicefrac{1}{\alpha}} \sigma_1 \|\nabla f(\mX_0)\|_F}
{\rho B^{\nicefrac{\alpha-1}{\alpha}} K}.
\end{align*}
The stated bound follows by absorbing constants depending only on $\alpha$ and
$d_0$ into $\gO(\cdot)$. Because $\Bar\gamma$ and $\Bar\nu$ are uniform in the
horizon, a fixed batch size satisfying the theorem's threshold makes
$\rho>0$ independent of $K$. The displayed initialization term remains
explicit; interpreting the bound as a rate in $K$ also requires
$\Exp{\|\mS_0\|_F}$ to stay bounded across the chosen horizon-dependent runs.
\end{proof}

\subsection{Proof of Corollary \ref{corollary:inexact_without_alpha}}

\corollaryinexactwithoutalpha*

\begin{proof}
Since $\sigma_1=0$, the batch-size condition in Theorem \ref{theorem:inexact} is vacuous for any $B\ge1$, and $\rho=1-\Bar{\gamma}=\rho_0>0$ is independent of $K$ under the uniform oracle bounds. Applying Theorem \ref{theorem:inexact} gives
\begin{align*}
\min_{0 \leq k \leq K-1} \Exp{\|\nabla f(\mX_k)\|_F}
\leq \gO \Bigg(
&\frac{f(\mX_0)-f_\star}{\rho_0 \eta K}
+ \frac{L\eta(1+\Bar{\nu})^2}{\rho_0}
+ \frac{(1+\Bar{\nu})\Exp{\| \mS_0 \|_F}}{\rho_0(1-\beta)K} \\
&+ \frac{\beta^2 L\eta(1+\Bar{\nu})^2}{\rho_0(1-\beta)}
+ \frac{(1+\Bar{\nu})\sigma_0(1-\beta)^{\frac{\alpha-1}{\alpha}}}
{\rho_0 B^{\frac{\alpha-1}{\alpha}}}
\Bigg).
\end{align*}
Substituting $\eta=K^{-3/4}$ and $1-\beta=K^{-1/2}$ yields
\begin{align*}
\min_{0 \leq k \leq K-1} \Exp{\|\nabla f(\mX_k)\|_F}
\leq \gO \Bigg(
&\frac{f(\mX_0)-f_\star}{\rho_0 K^{\frac{1}{4}}}
+ \frac{L(1+\Bar{\nu})^2}{\rho_0 K^{\frac{3}{4}}}
+ \frac{(1+\Bar{\nu})\Exp{\| \mS_0 \|_F}}{\rho_0\sqrt{K}} \\
&+ \frac{L(1+\Bar{\nu})^2}{\rho_0 K^{\frac{1}{4}}}
+ \frac{(1+\Bar{\nu})\sigma_0}
{\rho_0 B^{\frac{\alpha - 1}{\alpha}} K^{\frac{\alpha - 1}{2 \alpha}}}
\Bigg),
\end{align*}
where $\beta^2\le1$ is absorbed into the constant. This proves the claim.
\end{proof}

\section{Experiment Details and Ablation Study}\label{app:expdetails}

This appendix describes the implementation, optimizer hyperparameters, and training schedules used in the experiments of Section~\ref{subsec:nanogpt}~and~\ref{subsec:cifar10}. We use the \texttt{modded-nanogpt} speedrun codebase of~\cite{Jordan2024Modded} and the CIFAR-10 Airbench codebase of~\cite{Jordan2024Cifar94}. We modify the orthogonalization methods inside \algname{Muon} to expose the inexact solver and our randomized low-rank projection. All data pipelines, model architecture, and learning-rate schedules are preserved from the original code.

\subsection{Randomized Kaczmarz-Style Sparse Sketching}\label{app:kaczmarz}
The randomized Kaczmarz-style sparse sketch~\cite{Drineas2006Fast, Strohmer2008Randomized} provides a more computationally efficient alternative to the dense Gaussian sketch $\bm\Omega \sim \gN(0,1)^{n \times \ell}$ used in Algorithm~\ref{alg:rand_ns_update}. The resulting sketching procedure is summarized in Algorithm~\ref{alg:kaczmarz_rand_ns_update}, where Steps~1--3 replace the Gaussian sketch construction in Step~2 of Algorithm~\ref{alg:rand_ns_update}.
\begin{algorithm}[ht]
	\caption{Lifted Randomized Kaczmarz-Style Sparse Sketching Polar Decomposition}\label{alg:kaczmarz_rand_ns_update}
	\begin{algorithmic}[1]
		\REQUIRE Matrix $\mM \in \mathbb{R}^{m \times n}$ with rank $r$, target rank $1 \leq s < r$, oversampling parameter $p \ge 2$ with $\ell=s+p\le \min(m,n)$, power iteration $h \ge 0$, and Newton-Schulz steps $q \ge 0$.
		
		\begin{phasebox}{phasepink}{Project Down}	
			\STATE Set $\ell := s+p$ and compute column sampling probabilities
			\[
			\pi_j := \frac{\|\mM_{:,j}\|_2^2}{\|\mM\|_F^2},
			\qquad j=1,\ldots,n
			\]
			\STATE Sample $\ell$ indices $i_1,\ldots,i_\ell$ i.i.d.\ from $\bm\pi$ with replacement
			\STATE Set
			\[
			\bm\Omega_{:,k}
			:=
			\frac{e_{i_k}}{\sqrt{\ell\,\pi_{i_k}}},
			\qquad k=1,\ldots,\ell,
			\]
			where $e_j$ is the $j$-th canonical basis vector
			\STATE Form $\mY := (\mM \mM^\top)^h \mM \bm\Omega$, and compute orthonormal basis $\mQ := \operatorname{orth}(\mY)$ 
			\STATE Set $\mB := \mQ^\top \mM \in \mathbb{R}^{\ell \times n}$
		\end{phasebox}
		\begin{phasebox}{phasegreen}{Iterative \\ Approximation}
			\STATE Set $\mZ_0 := \nicefrac{\mB}{\| \mB\|_{\mathrm{op}}}$
			\FOR{$t = 0, 1,\dots,q-1$}
			\STATE $ \mZ_{t+1} := \varphi \left( \mZ_t \right)$
			\ENDFOR
		\end{phasebox}
		\begin{phasebox}{phasepink}{Project Up}
			\STATE \textbf{Return} $\gT_{h,q}( \mM; \bm\Omega) := \mQ \mZ_q$
		\end{phasebox}
	\end{algorithmic}			
\end{algorithm}

\subsection{Pretraining of nanoGPT on FineWeb}\label{app:expdetails_nanogpt}

\paragraph{Inexact solver and randomized projection hyperparameters.}
For randomized \algname{Muon}, after Bayesian hyperparameter tuning, we implement Algorithm~\ref{alg:rand_ns_update} with the quintic-theoretical Newton-Schulz polynomial, using $q = 7$ inner iterations together with a randomized projection of target rank $s = 200$, oversampling parameter $p = 10$, and $h = 1$ power iterations. These are then applied to different \algname{Muon} variants. 

\paragraph{Hyperparameters and parameter routing.}
Following~\cite{jordan2024muon,amsel2025polar}, all \algname{Muon} variants route the embedding and language-model-head layers through \algname{AdamW} and apply \algname{Muon} to the remaining matrix-shaped parameters. The \algname{AdamW} and \algname{SGD}-Nesterov baselines instead apply their respective optimizer to every parameter. The \algname{Muon} learning rate and momentum are tuned via Bayesian optimization, yielding learning rate around $3.25 \times 10^{-2}$ and momentum around $\beta = 0.9665$; the \algname{SGD}-Nesterov baseline uses learning rate around $4.18 \times 10^{-4}$ with Nesterov momentum around $0.86$, and the \algname{AdamW} baseline uses learning rate around $1.80 \times 10^{-3}$. All remaining settings follow the \texttt{modded-nanogpt} speedrun configuration of~\cite{Jordan2024Modded}: training sequence length $2048$ tokens, the three-stage training schedule, and a total of $1490$ training iterations. Each configuration is averaged over $5$ random seeds, and validation perplexity is computed on a held-out shard of $10{,}485{,}760$ FineWeb tokens. All nanoGPT runs use $4\times$ NVIDIA H100 GPUs, and the wall-clock times reported in Table~\ref{tab:nanogpt_e5} are end-to-end training times measured on that configuration. For the inexact-solver ablation in Figure~\ref{fig:nanogpt_abl_e1_solvers}, the quintic-empirical solver uses coefficients $(a, b, c) = (3.4445, -4.7750, 2.0315)$ at every step, while the PolarExpress solver uses iteration-dependent coefficients $\{(a_t, b_t, c_t)\}_{t=1}^{9}$ given by
\[
\begin{aligned}
	&(8.1566, -22.4833, 15.8788),\quad (4.0429, -2.8089, 0.5000),\quad (3.8917, -2.7725, 0.5061),\\
	&(3.2858, -2.3681, 0.4645),\quad (2.3005, -1.6112, 0.3833),\quad (1.8631, -1.2042, 0.3422),\\
	&(1.8383, -1.1779, 0.3397),\quad (1.8382, -1.1779, 0.3396),\quad (1.8750, -1.2500, 0.3750).
\end{aligned}
\]

\subsection{CNN on CIFAR-10}\label{app:expdetails_cifar10}

\paragraph{Inexact solver and randomized projection hyperparameters.}
After hyperparameter tuning, we select the quintic-theoretical Newton-Schulz solver with $q = 7$ inner iterations and a randomized projection of target rank $s = 128$, oversampling parameter $p = 10$, and $h = 2$ power iterations. These are then applied to different \algname{Muon} variants.

\paragraph{Hyperparameters and parameter routing.}
Following the parameter-routing convention of the Airbench codebase~\citep{Jordan2024Cifar94}, all \algname{Muon} variants apply \algname{Muon} to the 4D convolutional filter weights and route the remaining parameters through \algname{SGD} with Nesterov momentum, whereas the \algname{AdamW} and \algname{SGD} baselines apply their respective optimizer to every parameter. After hyperparameter tuning, we use a batch size of $500$ for the main comparison reported in Table~\ref{tab:cifar10_e5} and Figure~\ref{fig:cifar10_e5}. Each configuration is averaged over $50$ random seeds. All CIFAR-10 runs use a single NVIDIA A100 GPU, and the wall-clock times reported in Table~\ref{tab:cifar10_e5} are end-to-end training times measured on that configuration. For the inexact-solver ablation in Figure~\ref{fig:abl_e1_solvers}, the quintic-empirical solver uses coefficients $(a, b, c) = (3.4445, -4.7750, 2.0315)$ at every step, while the PolarExpress solver uses iteration-dependent coefficients $\{(a_t, b_t, c_t)\}_{t=1}^{9}$ given by
\[
\begin{aligned}
	&(8.2872, -23.5959, 17.3004),\quad (4.1071, -2.9478, 0.5448),\quad (3.9487, -2.9089, 0.5518),\\
	&(3.3184, -2.4885, 0.5100),\quad (2.3007, -1.6689, 0.4188),\quad (1.8913, -1.2680, 0.3768),\\
	&(1.8750, -1.2500, 0.3750),\quad (1.8750, -1.2500, 0.3750),\quad (1.8750, -1.2500, 0.3750).
\end{aligned}
\]

\subsection{nanoGPT Benchmark Ablation Study.}\label{app:nano_ablation}
To assess the sensitivity of randomized \algname{Muon} with Nesterov's momentum to its hyperparameters on the nanoGPT benchmark, we present further ablation results in Figure~\ref{fig:nanogpt_ablations}. As shown in Figure~\ref{fig:nanogpt_abl_e1_solvers}, three of the four inexact solvers (Quintic theoretical, PolarExpress, and Quintic empirical) trace essentially identical convergence curves, while Cubic lags slightly behind throughout training. Figure~\ref{fig:nanogpt_abl_e4_seqlen} compares three sequence lengths ($1024$, $2048$, $4096$): the two shorter sequences reach similar final perplexity while $4096$ ends visibly higher. Because longer sequences process proportionally more tokens per step, they stop in fewer total steps.

\begin{figure}[t]
	\centering
	\begin{subfigure}[t]{0.48\textwidth}
		\centering
		\includegraphics[width=\linewidth]{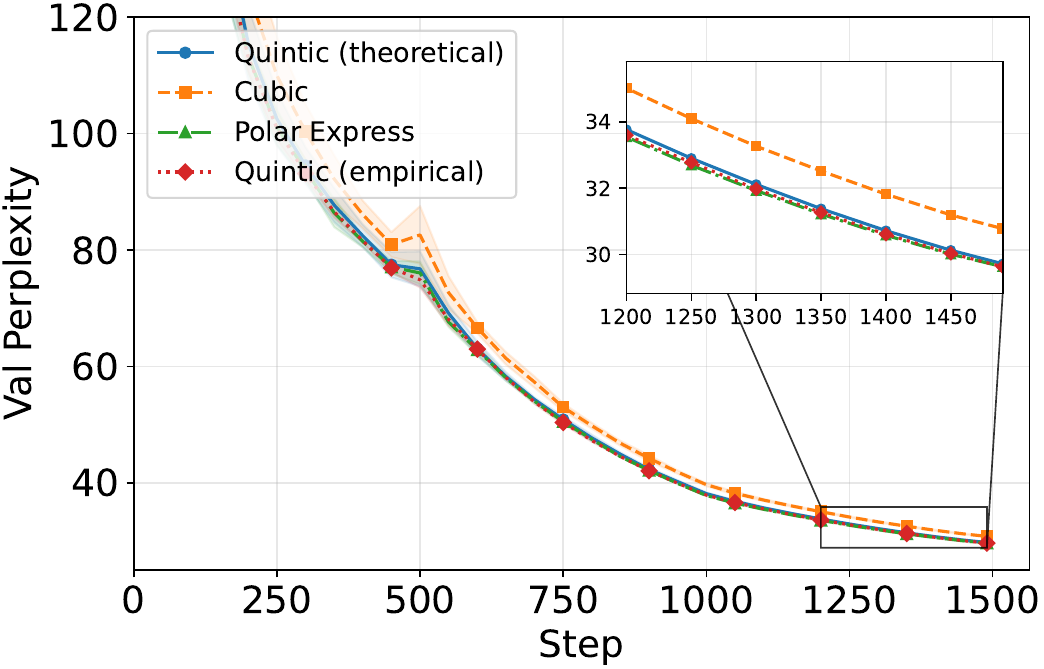}
		\caption{Inexact solver}\label{fig:nanogpt_abl_e1_solvers}
	\end{subfigure}\hfill
	\begin{subfigure}[t]{0.48\textwidth}
		\centering
		\includegraphics[width=\linewidth]{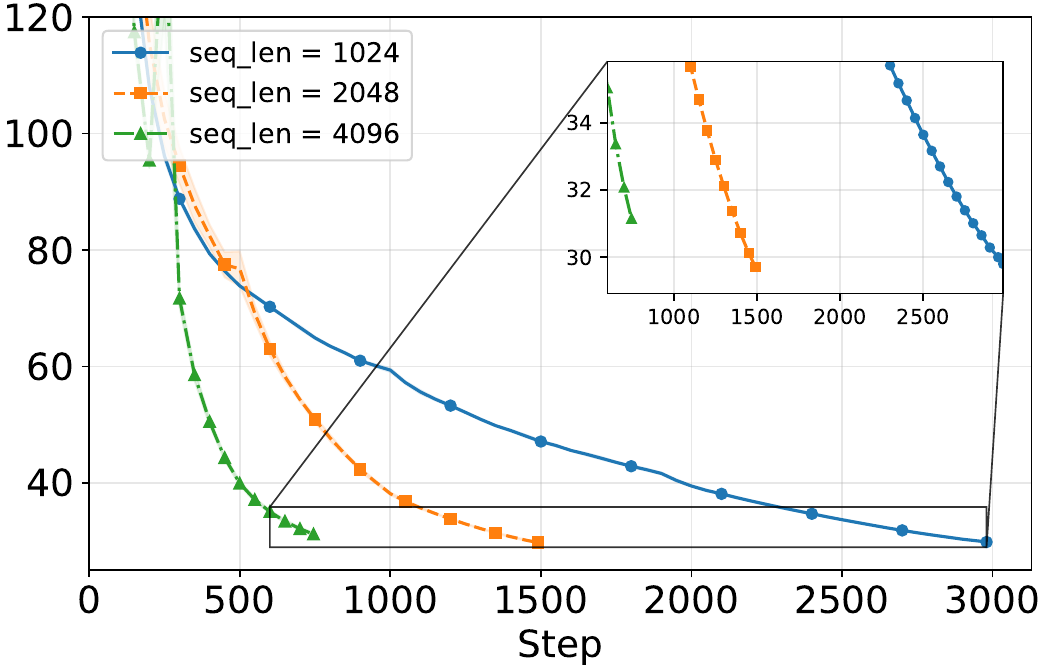}
		\caption{Sequence length}\label{fig:nanogpt_abl_e4_seqlen}
	\end{subfigure}
	\caption{nanoGPT ablations for randomized \algname{Muon} with Nesterov momentum: validation perplexity versus training steps when each hyperparameter is varied in turn and the others are held at the default configuration. Each curve is the mean over $5$ random seeds. The insets zoom into the final iterations for a clearer comparison.}
	\label{fig:nanogpt_ablations}
\end{figure}

\begin{figure}[t]
	\centering
	\begin{subfigure}[t]{0.48\textwidth}
		\centering
		\includegraphics[width=\linewidth]{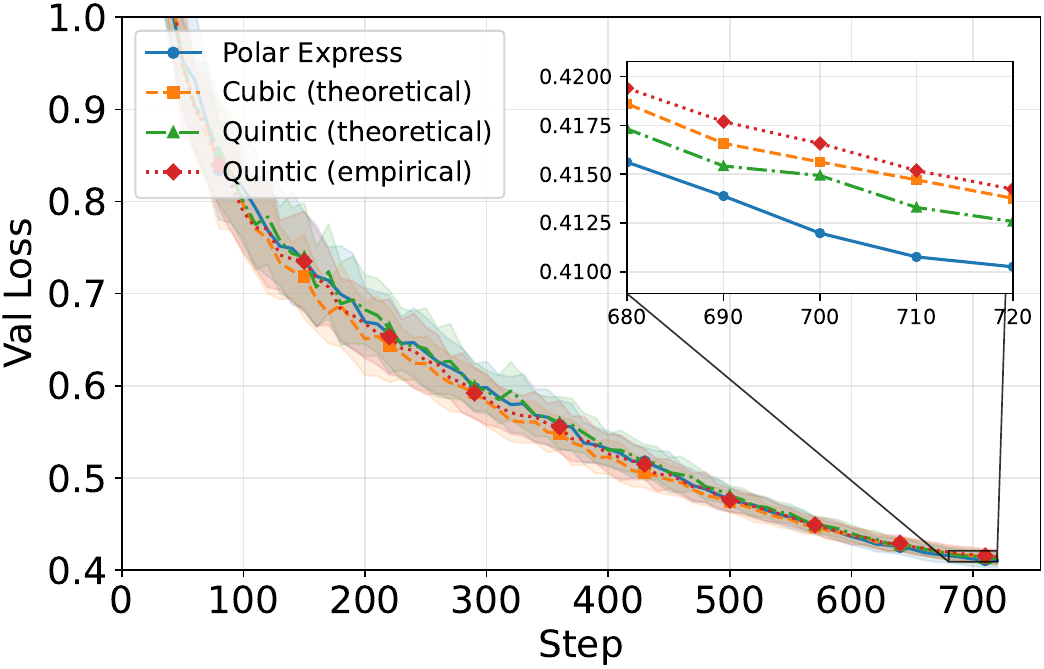}
		\caption{Inexact solver}\label{fig:abl_e1_solvers}
	\end{subfigure}\hfill
	\begin{subfigure}[t]{0.48\textwidth}
		\centering
		\includegraphics[width=\linewidth]{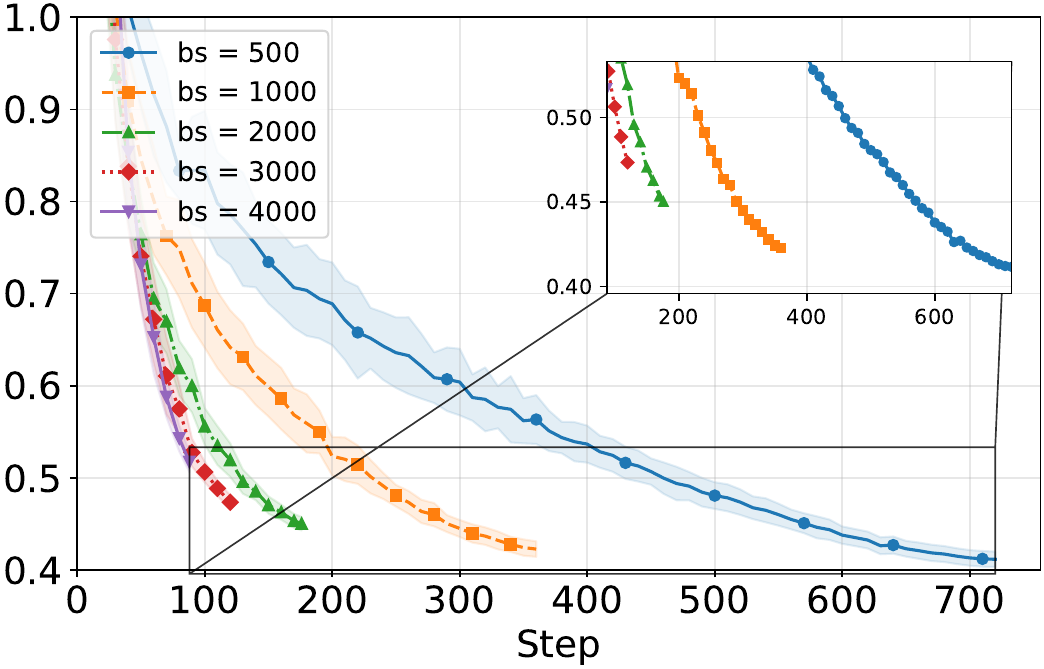}
		\caption{Batch size}\label{fig:abl_e4_bs}
	\end{subfigure}
	\caption{CIFAR-10 ablations for randomized \algname{Muon} with Nesterov momentum: validation loss versus training steps when each hyperparameter is varied in turn and the others are held at the default configuration. Each curve is the mean over $50$ random seeds. The insets zoom into the final iterations for a clearer comparison.}
	\label{fig:cifar10_ablations}
\end{figure}

\subsection{CIFAR-10 Benchmark Ablation Study.}\label{app:cifar_ablation}
We also investigate the sensitivity of randomized \algname{Muon} with Nesterov's momentum to its hyperparameters on the CIFAR-10 benchmark; we present further ablation results in Figure~\ref{fig:cifar10_ablations}. As shown in Figure~\ref{fig:abl_e1_solvers}, the four inexact solvers achieve comparable performance on this small-scale CNN: PolarExpress attains the lowest final validation loss with Quintic (theoretical) close behind, and their convergence trajectories are nearly indistinguishable. Figure~\ref{fig:abl_e4_bs} shows that smaller batch sizes reach a lower final validation loss but at a slower per-step rate, indicating that the additional update steps afforded by small batches help locate better minima at the cost of slower per-step progress.

\subsection{Empirical Study of the Inexact-Polar Quantities on nanoGPT}\label{app:gamma}

Theorem~\ref{theorem:inexact} uses deterministic oracle bounds that hold uniformly across iterations and horizons. To examine alignment in practice, we show the singular spectra of the momentum matrices every $50$ training steps on nanoGPT in Figure~\ref{fig:spectrum}, together with the realized alignment deficit $\widehat\gamma_k = 1 - \nicefrac{\langle \mM_k, \gT(\mM_k) \rangle}{\| \mM_k \|_*}$ for nonzero $\mM_k$ in Figure~\ref{fig:gamma_traj} and Table~\ref{tab:gamma_summary}. These sampled diagnostics differ from the deterministic conditional-expectation bounds $\gamma_k$ in Assumption~\ref{assume:general_inexact}; a finite observed maximum does not certify a uniform bound. Lemma~\ref{lem:uniform_polar_bounds} establishes such bounds for the stated Gaussian and full-space theoretical solvers under Frobenius normalization. We also plot the spectral bound of Proposition~\ref{prop:randomized_polar} along the same run in Figure~\ref{fig:bound}. All statistics are averaged over the same $5$ random seeds used in Section~\ref{subsec:nanogpt}.

\paragraph{Singular spectra of the momentum matrices.}
Figure~\ref{fig:spectrum} shows the spectra of $\mM_k$ for the Q, K, and V matrices. This visualization partially supports the low rank nature of the momentum updates, as the singular values decay rapidly.

\begin{figure}[h]
	\centering
	\begin{subfigure}[t]{0.32\textwidth}
		\centering
		\includegraphics[width=\linewidth]{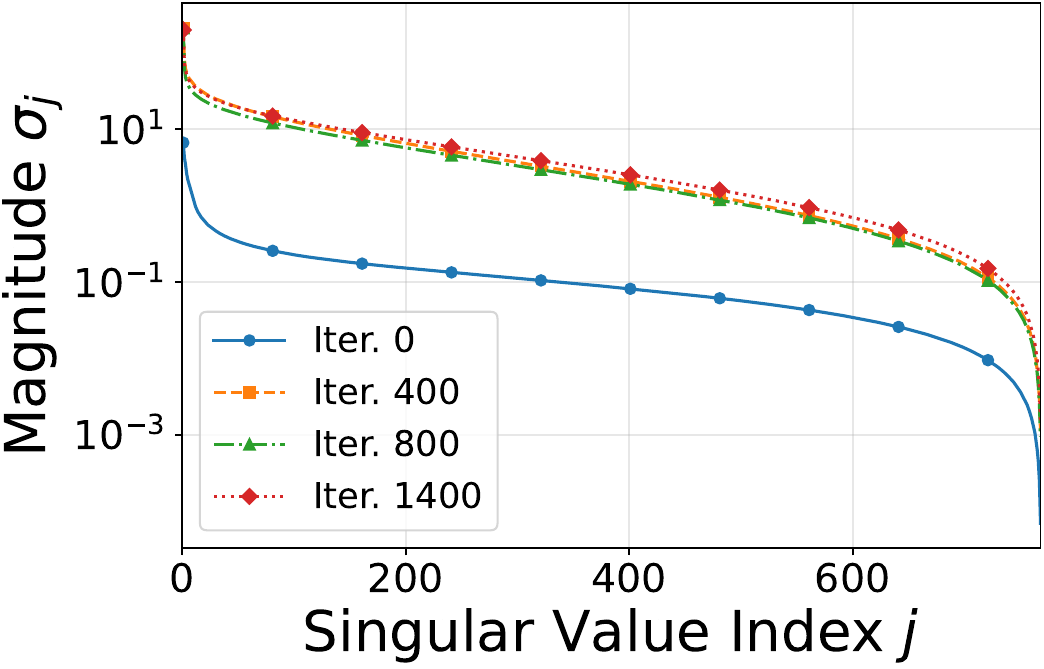}
		\caption{Q Matrix}\label{fig:spectrum_q}
	\end{subfigure}\hfill
	\begin{subfigure}[t]{0.32\textwidth}
		\centering
		\includegraphics[width=\linewidth]{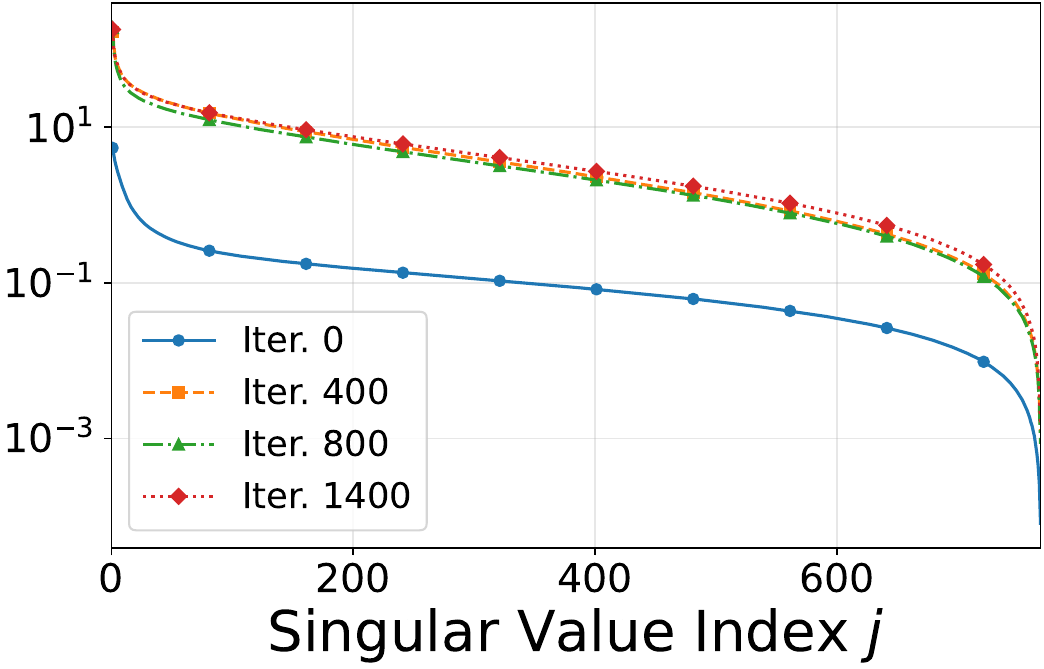}
		\caption{K Matrix}\label{fig:spectrum_k}
	\end{subfigure}\hfill
	\begin{subfigure}[t]{0.32\textwidth}
		\centering
		\includegraphics[width=\linewidth]{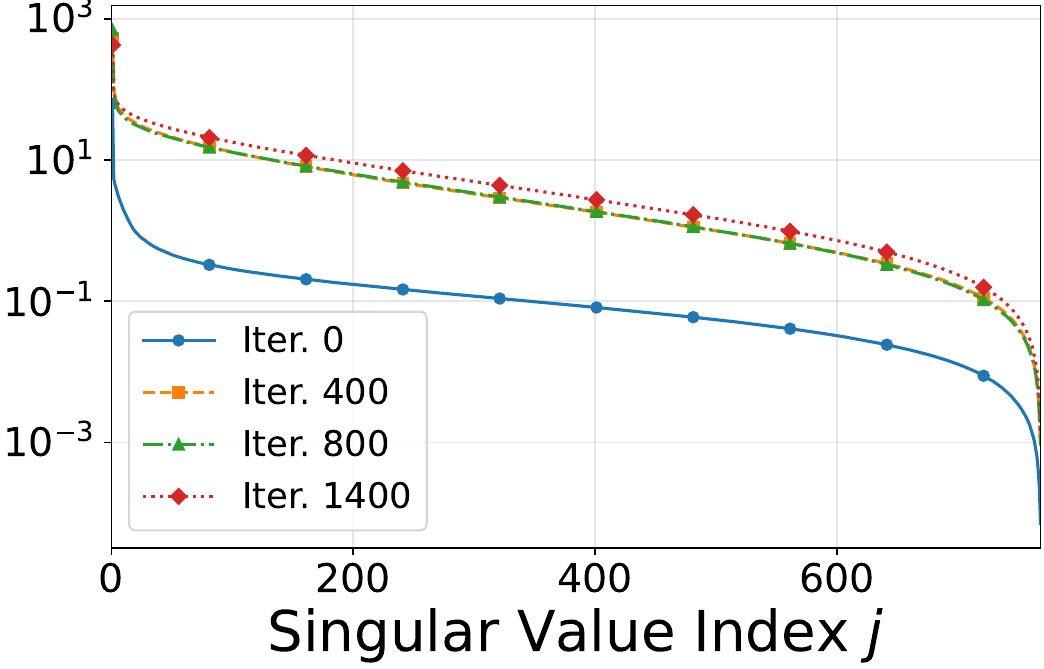}
		\caption{V Matrix}\label{fig:spectrum_v}
	\end{subfigure}
	\caption[Singular values of the nanoGPT momentum matrices]{\small Singular values of the nanoGPT momentum matrices $\mM_k$ for randomized \algname{Muon} (Nes.-Gaussian), on a logarithmic vertical axis. Each curve is averaged over the ten attention layers within a seed and then over five seeds. Each attention projection is a $768 \times 768$ matrix of full rank $768$.}
	\label{fig:spectrum}
\end{figure}

\begin{table}[h]
    \centering
    \small
    \caption{\small Aggregate realized alignment deficit $\widehat\gamma_k$ over the nanoGPT run. Each statistic is first computed within each seed over all sampled training steps and then aggregated across seeds.}
    \label{tab:gamma_summary}
    \begin{tabular}{lcc}
        \toprule
        \textbf{Method} & $\operatorname{mean}_k \widehat\gamma_k$ & $\max_k \widehat\gamma_k$ \\
        \midrule
        Rand \algname{Muon} (Polyak)        & $0.3409 \pm 0.0021$ & $0.5572 \pm 0.0093$ \\
        Rand \algname{Muon} (Nes.-Gaussian) & $0.3824 \pm 0.0019$ & $0.5887 \pm 0.0102$ \\
        Rand \algname{Muon} (Nes.-Kaczmarz) & $0.3136 \pm 0.0028$ & $0.5831 \pm 0.0107$ \\
        \bottomrule
    \end{tabular}
\end{table}

\begin{figure}[h]
	\centering
	\begin{subfigure}[t]{0.48\textwidth}
		\centering
		\includegraphics[width=\linewidth]{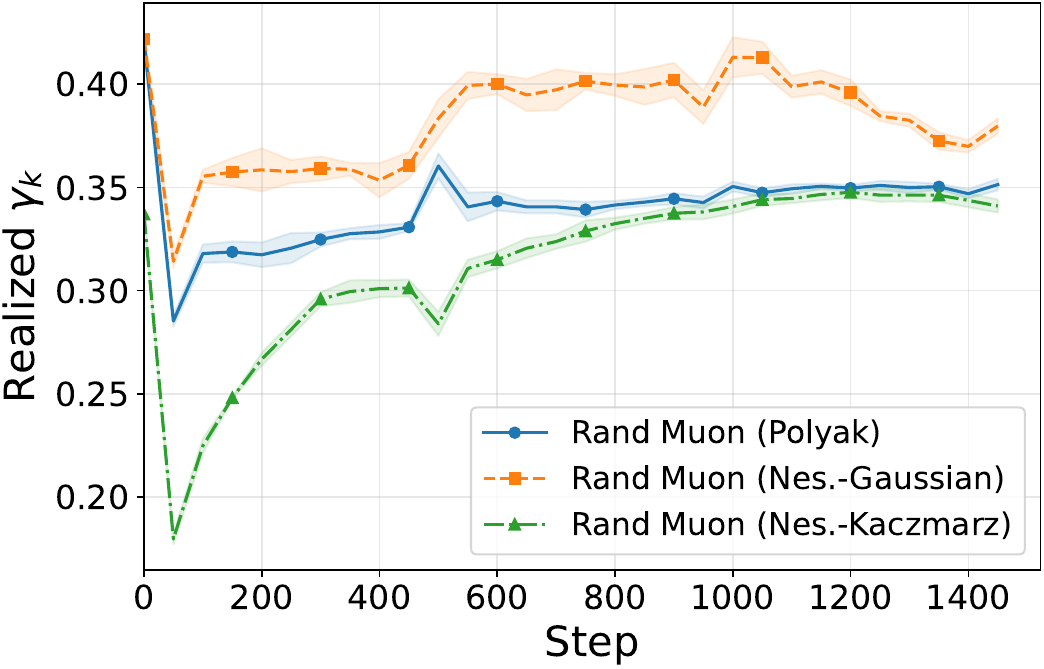}
		\caption{Realized alignment deficit $\widehat\gamma_k$}\label{fig:gamma_traj}
	\end{subfigure}\hfill
	\begin{subfigure}[t]{0.48\textwidth}
		\centering
		\includegraphics[width=\linewidth]{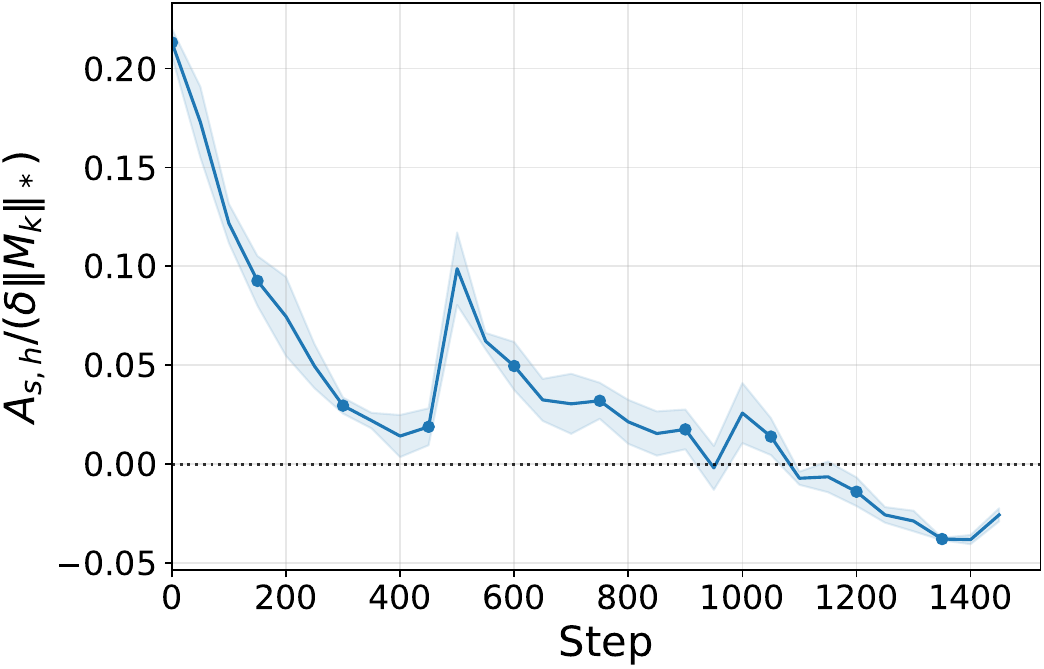}
		\caption{Bound of Proposition \ref{prop:randomized_polar}}\label{fig:bound}
	\end{subfigure}
	\caption[Realized alignment deficit and the bound of Proposition \ref{prop:randomized_polar}]{\small (a) Realized alignment deficit $\widehat\gamma_k$ over training for the three randomized variants. (b) The quantity $\nicefrac{A_{s,h}}{(\delta\|\mM_k\|_*)}$ of Proposition~\ref{prop:randomized_polar} for randomized \algname{Muon} (Nes.-Gaussian) at the tuned setting $s = 200$, $p = 10$, $h = 1$.}
	\label{fig:gamma_and_bound}
\end{figure}

\paragraph{Realized alignment deficits along the trajectory.}
Figure~\ref{fig:gamma_traj} and Table~\ref{tab:gamma_summary} report $\widehat\gamma_k$. After the first few hundred steps, the values are stable, and the reported means of the per-seed maxima are below $0.59$ for every randomized variant. This indicates favorable alignment on the sampled iterates. These diagnostics do not determine the theorem's descent coefficient $\rho$, which also depends on $\sigma_1$ and $B$; the identity $\rho=1-\Bar\gamma$ holds only when $\sigma_1=0$.

\begin{table}[h]
    \centering
    \small
    \caption{\small Aggregate realized alignment deficit $\widehat\gamma_k$ over the nanoGPT run. Each statistic is first computed within each seed over all sampled training steps and then aggregated across seeds.}
    \label{tab:gamma_summary}
    \begin{tabular}{lcc}
        \toprule
        \textbf{Method} & $\operatorname{mean}_k \widehat\gamma_k$ & $\max_k \widehat\gamma_k$ \\
        \midrule
        Rand \algname{Muon} (Polyak)        & $0.3409 \pm 0.0021$ & $0.5572 \pm 0.0093$ \\
        Rand \algname{Muon} (Nes.-Gaussian) & $0.3824 \pm 0.0019$ & $0.5887 \pm 0.0102$ \\
        Rand \algname{Muon} (Nes.-Kaczmarz) & $0.3136 \pm 0.0028$ & $0.5831 \pm 0.0107$ \\
        \bottomrule
    \end{tabular}
\end{table}

\paragraph{The bound of Proposition \ref{prop:randomized_polar}.}
Figure~\ref{fig:bound} shows $\nicefrac{A_{s,h}}{(\delta\|\mM_k\|_*)}$ without the positive part, at the values $s = 200$, $p = 10$ and $h = 1$, which were tuned for validation perplexity rather than to keep this quantity positive. With these hyperparameters, it does become mildly negative late in training. Increasing $p$, or increasing $h$ when $\varrho_s<1$, can improve this spectral bound under the conditions in Appendix~\ref{appendix:discuss_prop}. Its positivity is not required for the separate uniform Gaussian bound in Lemma~\ref{lem:uniform_polar_bounds}.

\section{Limitations}\label{app:limitation}
The inexact-polar theory depends on the abstract quantities $\gamma_k$ and $\nu_k$, but we do not provide an adaptive rule for choosing the number of polar iterations $q$ during training. 
Moreover, our analysis of randomized polar decomposition is primarily developed for Gaussian sketches; extending the theory to sparse or Kaczmarz-style sketches remains an open problem.
The current results also do not optimize the choice of the rank $s$, the oversampling parameter $p$, or the number of power iterations $h$, all of which could be selected adaptively in practice.

\end{document}